\documentclass[a4paper,11pt,reqno,oneside]{amsart}
\usepackage{amsfonts, amssymb, amsmath, amsthm, stmaryrd, verbatim}
\usepackage[all,cmtip]{xy}
\usepackage{mathtools}

\usepackage{multicol}
\usepackage[shortlabels]{enumitem}

\usepackage{tikz}

\makeatletter

\DeclareSymbolFont{cmlargesymbols}{OMX}{cmex}{m}{n}
\usepackage[vmargin=3.2cm]{geometry}

\usepackage{color}
\usepackage{hyperref}  
\hypersetup{%
  bookmarksnumbered=true,%
  colorlinks=true,%
  linkcolor=blue,%
  citecolor=blue,%
  filecolor=blue,%
  menucolor=blue,%
  urlcolor=blue,%
  pdfnewwindow=true,%
  pdfstartview=}
\usepackage{threeparttable}

\newtheorem{lemma}{Lemma}
\newtheorem{theorem}[lemma]{Theorem}
\newtheorem*{theorem*}{Theorem}
\newtheorem{proposition}[lemma]{Proposition}
\newtheorem{corollary}[lemma]{Corollary}

\theoremstyle{definition}
\newtheorem{definition}[lemma]{Definition}

\newtheorem{example}[lemma]{Example}

\newtheorem{remark}[lemma]{Remark}
\newtheorem*{remark*}{Remark}

\newtheorem{notation}[lemma]{Notation}

\newcommand{\F}{\mathbb{F}}
\newcommand{\pt}{\mathrm{pt}}

\newcommand{\Hom}{\mathrm{Hom}}
\newcommand{\End}{\mathrm{End}}

\newcommand{\Aut}{\mathrm{Aut}}

\newcommand{\op}{\mathrm{op}}
\newcommand{\inv}{\mathrm{inv}}

\DeclareMathOperator{\Sing}{Sing}

\DeclareMathOperator{\len}{len}

\newcommand{\Z}{\mathbf{Z}}

\newcommand{\Sq}{Sq}

\newcommand{\sharpop}{\mathbin{\sharp}}

\def\slashedarrowfill@#1#2#3#4#5{%
  $\m@th\thickmuskip0mu\medmuskip\thickmuskip\thinmuskip\thickmuskip
   \relax#5#1\mkern-7mu%
   \cleaders\hbox{$#5\mkern-2mu#2\mkern-2mu$}\hfill
   \mathclap{#3}\mathclap{#2}%
   \cleaders\hbox{$#5\mkern-2mu#2\mkern-2mu$}\hfill
   \mkern-7mu#4$%
}
\def\rightslashedarrowfill@{%
  \slashedarrowfill@\relbar\relbar\mapstochar\rightarrow}
\newcommand\xhto[2][]{%
  \ext@arrow 1579{\rightslashedarrowfill@}{#1}{#2}}

\newcommand\xto[2][]{%
  \ext@arrow 1579{\rightarrowfill@}{#1}{#2}}

\newcommand\xot[2][]{%
  \ext@arrow 5197{\leftarrowfill@}{#1}{#2}}

\providecommand*{\twoheadrightarrowfill@}{%
  \arrowfill@\relbar\relbar\twoheadrightarrow
}

\newcommand\xtwoheadrightarrow[2][]{%
  \ext@arrow 1579\twoheadrightarrowfill@{#1}{#2}%
}

\newcommand{\longto}{\xto{\quad}}

\newcommand{\longincl}{\xhookrightarrow{\quad}}

\newcommand{\tensor}{\otimes}
\newcommand{\smashprod}{\wedge}

\newcommand{\id}{\mathrm{id}}

\newcommand{\isom}{\approx}
\newcommand{\homot}{\simeq}

\newcommand{\homeom}{\approx}

\newcommand{\suspension}{\Sigma}
\newcommand{\loops}{\Omega}

\DeclareMathOperator{\ho}{Ho} 

\newcommand{\incl}{\hookrightarrow}

\DeclareMathOperator{\colim}{colim}

\providecommand{\ker}{}
\renewcommand{\ker}{\mathrm{Ker}\,}

\newcommand{\ev}{\mathrm{ev}}
\newcommand{\pr}{\mathrm{pr}}

\newcommand{\spaces}{\mathbf{Top}}
\newcommand{\pointedSpaces}{\spaces_\ast}
\newcommand{\coalgebras}{\mathbf{Coalg}}
\newcommand{\finiteSets}{\mathbf{Fin}}

\newcommand{\calA}{\mathcal{A}}
\newcommand{\calB}{\mathcal{B}}
\newcommand{\calC}{\mathcal{C}}

\newcommand{\calE}{\mathcal{E}}

\newcommand{\calG}{\mathcal{G}}

\newcommand{\calK}{\mathcal{K}}

\newcommand{\calO}{\mathcal{O}}

\newcommand{\calS}{\mathcal{S}}

\makeatother

\usepackage[utf8]{inputenc}
\usepackage[T5,T1]{fontenc}
\DeclareTextSymbolDefault{\ohorn}{T5}

\title{Homology operations revisited}
\date{\today}

\author{Anssi Lahtinen}
\address{%
University of Copenhagen\\
Department of Mathematical Sciences\\
Universitetsparken~5\\
2100 Copenhagen Ø\\
Denmark}
\email{lahtinen@math.ku.dk}


\begin{document}

\begin{abstract}
The mod $p$ homology of $E_\infty$-spaces is a classical topic
in algebraic topology traditionally approached in terms
of Dyer--Lashof operations. In this paper, we 
offer a new perspective on the subject by 
providing a detailed investigation of 
an alternative family of homology operations equivalent to, but 
distinct from, the Dyer--Lashof operations.
Among other things, we will relate these operations 
to the Dyer--Lashof operations, describe the algebra
generated by them, and use them to describe the 
homology of free $E_\infty$-spaces.
We will also investigate the relationship between the 
operations arising from the additive and multiplicative
$E_\infty$-structures on an $E_\infty$--ring space.
The operations have especially
good properties in this context, allowing for a simple and conceptual 
formulation of ``mixed Adem relations'' 
describing how the operations arising from the two 
different $E_\infty$-structures interact.
\end{abstract}

\maketitle

\setcounter{tocdepth}{1}
\tableofcontents

\section{Introduction}

Many spaces of fundamental importance in topology and algebra
are homotopy equivalent to $E_\infty$-spaces, i.e.\ 
spaces carrying a multiplication that is associative and commutative
up to an infinite sequence of coherent homotopies.
One source of examples of such spaces are the classifying spaces of 
small symmetric monoidal categories, with product induced by the 
symmetric monoidal product of the category; examples of this form include the 
disjoint unions
$\bigsqcup_{k\geq 0} B\Sigma_k$,
$\bigsqcup_{k\geq 0} B\Aut(F_k)$,
and
$\bigsqcup_{k\geq 0} BGL_k(R)$ 
of classifying spaces of symmetric groups, automorphism groups of free groups, 
and general linear groups over a ring $R$, respectively,
with multiplications arising from disjoint union of finite sets,
free products of finitely generated free groups,
and direct sums of free $R$-modules.
Another (and overlapping) source of examples of $E_\infty$-spaces
are infinite loop spaces,
i.e.\ spaces of the form $\loops^\infty X$ for a spectrum $X$.
Indeed,
grouplike $E_\infty$-spaces---ones whose set of components 
form a group under the multiplication induced by the 
$E_\infty$-structure---are essentially the same thing as 
connective spectra; see 
e.g.\ \cite[Pretheorem~2.3.2]{AdamsInfiniteLoopSpaces} 
and the references there. Interesting examples of infinite loop spaces 
include Eilenberg--Mac~Lane spaces and other spaces representing generalized
cohomology theories, 
as well as many spaces fundamental for the homotopy theoretic study of manifolds 
such as the classifying spaces $BO$, $BPL$, $BTOP$, and $BF$
of stable vector bundles, PL microbundles, topological microbundles, and 
spherical fibrations, respectively, 
with multiplications arising from Whitney sum of bundles,
and the associated ``coset spaces'' such as $F/PL$ and $PL/O$.
Moving beyond $E_\infty$-spaces,
many spaces of interest are in fact homotopy equivalent to spaces
carrying not just one, but two compatible
$E_\infty$-structures, an ``additive'' and a ``multiplicative''
one. Examples of such spaces, called $E_\infty$--ring spaces, include
$\bigsqcup_{k\geq 0} B\Sigma_k$, with
multiplications arising from disjoint union and direct product of 
finite sets;
$\bigsqcup_{k\geq 0} BGL_k(R)$ for a commutative ring $R$,
with multiplications arising from direct sum and 
tensor product of free $R$-modules; and 
the infinite loop space $\loops^\infty R$
of an $E_\infty$--ring spectrum $R$,
with multiplications arising from the infinite 
loop space structure on  $\loops^\infty R$ 
and the ring multiplication on $R$.

This paper is devoted to a new perspective 
on a very classical topic in algebraic topology,
namely the study of the mod $p$ homology of 
$E_\infty$-spaces. Traditionally, the mod $p$ homology of these
spaces has been studied in terms of a set of natural operations
called Dyer--Lashof operations \cite{DyerLashof}, 
originally due, in the case $p=2$,
to Kudo and Araki \cite{KudoAraki}.
For a detailed treatment of the mod $p$ homology 
of $E_\infty$-spaces in terms of the Dyer--Lashof operations,
we refer the reader to the monograph of Cohen, Lada and May \cite{HILS}.
The aim of this  paper is to develop the foundations of the mod $p$ homology 
of $E_\infty$-spaces in terms of another set of natural 
operations on the homology of $E_\infty$-spaces. 
We call these operations $E$-\emph{operations}.
Along the way, we will, among other things,
relate the $E$-operations to the Dyer--Lashof operations;
describe the algebra $\calE$ generated by the $E$-operations;
use the $E$-operations to describe the homology of the free $E_\infty$-space
generated by a  space; and describe how the $E$-operations
arising from the additive and multiplicative $E_\infty$-structures
of an $E_\infty$--ring space interact.

While the $E$-operations and the Dyer--Lashof operations 
encode essentially the same information---simple formulas
involving Steenrod operations explain how to pass back and forth 
between the two---the two sets of operations have markedly 
different properties. One difference is that, unlike the 
Dyer--Lashof algebra, the algebra $\calE$ generated by the 
$E$-operations is graded commutative.
Another is that the $E$-operations are defined in a 
straightforward way in terms of a homomorphism
induced by a continuous map, making the behaviour
of the $E$-operations transparent with respect to all manner of constructions
natural with respect to maps of spaces. For example,
unlike the Dyer--Lashof algebra, the algebra $\calE$ turns out to be 
an algebra over $\calA_p^\op$, the opposite of the mod $p$
Steenrod algebra (Proposition~\ref{prop:alginaopmods}),
and instead of Nishida relations,  
the interaction of the $E$-operations 
and the Steenrod operations on the homology of an $E_\infty$-space
is governed by a Cartan formula (Proposition~\ref{prop:modinaopmods}).

The $E$-operations have especially good properties in the context
of $E_\infty$--ring  spaces. For example, for $X$ an $E_\infty$--ring  space,
the relationship between the additive $E$-operations on 
the mod $p$ homology $H_\ast(X)$
and the multiplication on $H_\ast(X)$ induced by the multiplicative 
$E_\infty$-structure on $X$ is governed by an associativity formula
which contrasts with the more 
complicated formulas required in the context of Dyer--Lashof operations.
See Corollary~\ref{cor:rxyassoc} and Remark~\ref{rk:rxyassoc}.
Moreover, and more substantially, we obtain
a simple and conceptual statement of  
``mixed Adem relations'' (Theorem~\ref{thm:mixedademrels})
explaining how the operations arising from the 
multiplicative and additive $E_\infty$-structures interact.
The simplicity of the statement of this result stands in vivid contrast with the
complexity of the analogous result for Dyer--Lashof operations
\cite[Theorem~II.3.3]{HILS}.

The action of the ring $\calE$ extends to, and completely determines,
a natural action of a larger ring on the homology of $E_\infty$-spaces,
namely that of the direct sum
$\bigoplus_{k\geq 0} H_\ast(\Sigma_k)$ 
of mod $p$ homology groups
equipped with the
multiplication $\circ$ arising from Cartesian products of finite sets.
It is useful to consider the action of $\calE$ in the context
of this extended action; indeed, the  simple form of the 
``mixed Adem relations'' of Theorem~\ref{thm:mixedademrels}
depends crucially on us doing so. See Remark~\ref{rk:emixedademrels}.
The action of $(\bigoplus_{k\geq 0} H_\ast(\Sigma_k),\circ)$ on 
the homology of $E_\infty$-spaces is the analogue, in the context
of not-necessarily-grouplike $E_\infty$-spaces, of the action of 
the homology $H_\ast(QS^0)$ (with the product induced by composition of based maps
of spheres) on the homology of infinite loop spaces \cite[Theorem~I.1.4]{HILS}.

This paper is structured as follows. In section~\ref{sec:eops}, 
we will define the 
$E$-operations, and relate them to the Dyer--Lashof operations
(Propositions~\ref{prop:eqformulas} and \ref{prop:eqformulasps}).
In section~\ref{sec:cale}, we will study the algebra
$\calE$ generated by the $E$-operations.
We will show that $\calE$ embeds into the ring 
$(\bigoplus_{k\geq 0} H_\ast(\Sigma_k),\circ)$ 
(Proposition~\ref{prop:eprimeandeiso})
and, relying on the work of Turner \cite{Turner} and Ch{\ohorn}n \cite{Chon},
describe the ring structure of $\calE$
(Theorems~\ref{thm:edescp2}, \ref{thm:ehatdescpodd} and \ref{thm:edescpodd2}).
In section~\ref{sec:freespaces}, we will describe the homology of 
free $E_\infty$-spaces in terms of the $E$-operations
(Theorems~\ref{thm:hczdesc1} and \ref{thm:hczdesc2}).
In section~\ref{sec:coalgebraicalgebra},
we will study the action of the larger ring 
$(\bigoplus_{k\geq 0} H_\ast(\Sigma_k),\circ)$ 
on the homology of $E_\infty$-spaces
using the language of coalgebraic algebra.
We will see that 
$\bigoplus_{k\geq 0} H_\ast(\Sigma_k)$
has the structure of a commutative coalgebraic semiring
(Example~\ref{ex:sigmasemiring}),
describe the structure of this coalgebraic semiring
(Theorem~\ref{thm:sigmasemiringdesc}),
and show that the homology of an $E_\infty$-space
is a coalgebraic semimodule over 
$\bigoplus_{k\geq 0} H_\ast(\Sigma_k)$
(Theorem~\ref{thm:coalgsemimod}).
Section~\ref{sec:stabilityandtransgression}
is devoted to relating the $E$-operations
(indeed, the $(\bigoplus_{k\geq 0} H_\ast(\Sigma_k),\circ)$-module 
structures)
on $H_\ast(X)$ and $H_\ast(\loops X)$
when $X$ is an $E_\infty$-space.
We will show that the operations commute with the homology suspension map
(Theorem~\ref{thm:homologysuspension}) and that they are compatible
with the transgression in the Serre spectral sequence
of the path--loop fibration of $X$ (Theorem~\ref{thm:transgressionproperty}).
Finally, in section~\ref{sec:ringspaces}, we will study the homology of 
$E_\infty$--ring  spaces with an emphasis of relating the 
structures arising from the additive and multiplicative
$E_\infty$-structures. We will show that the homology of
an $E_\infty$--ring  space $X$ is a coalgebraic semiring
(Theorem~\ref{thm:semiringfromringspace}), and that this 
homology receives a coalgebraic semiring map from 
$\bigoplus_{k\geq 0} H_\ast(\Sigma_k)$
which allows one to express the operations arising from the additive
$E_\infty$-structure on $X$ in terms of the product induced by 
the multiplicative $E_\infty$-structure on $X$
(Theorem~\ref{thm:cpttoxringmap}).
Moreover, we prove the aforementioned result on ``mixed Adem relations''
(Theorem~\ref{thm:mixedademrels})
and also establish a ``mixed Cartan formula''
relating the operations arising from the multiplicative 
$E_\infty$-structure to the product arising from the additive one
(Theorem~\ref{thm:mixedcartan}).
As a companion to Theorem~\ref{thm:mixedademrels},
we also explain how to evaluate explicitly the products 
$r\sharpop s$ featuring in the statement of the theorem
 for all $r, s \in \bigoplus_{k\geq 0} H_\ast(\Sigma_k)$
(Remark~\ref{rk:sharopformulas} and Theorem~\ref{thm:esharopeformulas}).

While the $E$-operations seem, surprisingly, to have eluded previous detailed study,
they do make an appearance in \cite{HILS}, where they 
are defined in terms of the formula of equation~\eqref{eq:e0q} and 
are used to state the ``mixed Adem relations''
for Dyer--Lashof operations. See \cite[Definition~II.3.2 and Theorem~II.3.3]{HILS}.
Independent of the author, they have also been studied recently
by Weinan Lin, as the author learned after the main part of this research
at the prime 2 had been completed \cite{Lin}. 
Also taking the formula of equation~\eqref{eq:e0q} as the 
definition of the operations, Lin has independently proved some of the 
results contained here in the case $p=2$; these include 
Theorem~\ref{thm:eandehatbases},
Corollary~\ref{cor:rxyassoc}, and the ``mixed Cartan formula'' of 
equation~\eqref{eq:mixedcartanat2}. Lin's results also 
include a version of ``mixed Adem relations,''
although formulated differently from Theorem~\ref{thm:mixedademrels}.

\subsection*{Notations and conventions}
Throughout the paper, $p$ will be a fixed prime. 
When the case $p=2$ requires special attention,
we will indicate the changes required for this case
inside square brackets~[~].
Unless otherwise indicated, 
homology and cohomology is with mod $p$ coefficients.
We write $\calA_p$ for the mod $p$ Steenrod algebra,
and $P^k$ for the $k$-th Steenrod reduced power operation.
[When $p=2$, we define $P^k$ to be the $k$-th Steenrod square: 
$P^k = \Sq^k$.]
We write $\beta$ for the Bockstein operation both in homology
and cohomology.
All spaces are assumed to be compactly generated and weak Hausdorff.
We write $\spaces$ for the category of 
such spaces and continuous maps,
and $\pointedSpaces$ for the category of pointed
such spaces and
basepoint-preserving continuous maps.

By a \emph{coalgebra}, 
we mean a graded graded cocommutative coassociative counital 
coalgebra over $\F_p$. We will usually write $\psi$ for the
coproduct and $\varepsilon$ for the counit of coalgebras,
and will make frequent use of the Sweedler notation
\begin{align*}
	\psi(x) 
	&= 
	\sum_{\psi(x)}  x' \tensor x'' 
	= 
	\sum x'\tensor x''
	\\
	\psi^{n-1} (x)
	&= 
	\sum_{\psi^{n-1}(x)} x_{(1)} \tensor \cdots \tensor x_{(n)} 
	= 
	\sum x_{(1)} \tensor \cdots \tensor x_{(n)} 
\end{align*}
for coproducts and iterated coproducts.
We write $\coalgebras$ for the 
category of coalgebras and homomorphisms of such,
and note that passing to homology groups gives a 
functor $H_\ast \colon \spaces \to \coalgebras$:
the coproduct on $\psi\colon H_\ast(X) \to H_\ast(X)\tensor H_\ast(X)$ 
for a space $X$ is given 
by the map induced by the diagonal map of $X$, 
while the counit $\varepsilon\colon H_\ast(X) \to \F_p$ 
is given by the map induced by the 
unique map from $X$ to the one-point space $\pt$.

Given elements of (multi)graded objects are implicitly 
assumed to be homogeneous with respect to all the gradings present.
For a ring $R$, we write $F_R(S)$ for the free graded commutative 
$R$-algebra generated by a graded set or a graded $R$-module $S$.

\section{The \texorpdfstring{$E$}{E}-operations}
\label{sec:eops}

In this section, we will define the $E$-operations,
and relate them to the Dyer--Lashof operations.
Let us first fix our definition of an $E_\infty$-space, 
however. We assume that the reader is familiar with 
operads.  In this paper, an operad $\calO$ will always mean 
an operad in $\spaces$, and  we require our operads to be 
reduced in the sense 
that $\calO(0)$ equals the one-point space $\pt$.
We remind the reader that associated to an operad $\calO$,
there is a monad $O\colon \spaces \to \spaces$
such that an $\calO$-algebra structure $Z$ on a 
space is equivalent data to an $O$-algebra structure on $Z$,
and such that a continuous map is a morphism of $\calO$-algebras
precisely when it is a morphism of $O$-algebras.
Explicitly,
\[
	OZ = \bigsqcup_{n\geq 0} \calO(n) \times_{\Sigma_n} Z^n
\]
for a space $Z$. We will usually write $\gamma$ for the 
composition maps in an operad, and 
$\theta \colon \calO(n)\times_{\Sigma_n} Z^n \to Z$
and $\theta \colon OZ \to Z$
for $\calO$- and $O$-algebra structure maps
of an $\calO$-algebra (or equivalently, $O$-algebra) $Z$.

\begin{definition}
\label{def:einftyoperad}
By an \emph{$E_\infty$-operad}, we mean
an operad $\calC$
such that for all $n\geq 0$, the space $\calC(n)$ is 
contractible, the action of $\Sigma_n$
 on $\calC(n)$
is free, and $\calC(n)$ is $\Sigma_n$-equivariantly
homotopy equivalent to a $\Sigma_n$-CW complex.
\end{definition}

\begin{remark}
\label{rk:equivariantcwproperty}
Assuming the last property in Definition~\ref{def:einftyoperad}
is somewhat nonstandard,
but harmless: for any operad $\calO$ 
satisfying the first two properties, 
an application of the product-preserving
CW approximation functor $|\Sing_\bullet(-)|$ to $\calO$
yields an operad $\calC$ satisfying all three properties
as well as an operad map $\calC\to \calO$
via which any $\calO$-algebra acquires the structure of 
a $\calC$-algebra. 
\end{remark}

For an $E_\infty$-operad $\calC$ with associated monad $C$,
we have a canonical homotopy equivalence
\[
	C(\pt) 
	=
	\bigsqcup_{k\geq 0} \calC(k)/\Sigma_k 
	\homot 
	\bigsqcup_{k\geq 0}B\Sigma_k
\]
which we will use to identify 
$H_\ast(C(\pt))$ with $\bigoplus_{k\geq 0}H_\ast(\Sigma_k)$.

\smallskip

\emph{From now on until 
the beginning of section~\ref{sec:ringspaces},
we will work with an arbitrary but fixed 
$E_\infty$-operad $\calC$, and will write $C$ for the associated monad.}

\begin{definition}
By an \emph{$E_\infty$-space} we mean an algebra over 
$\calC$, and by an $E_\infty$-map 
between $E_\infty$-spaces,
we mean a map of $\calC$-algebras.
\end{definition}

\noindent
We will not broach the topic of maps between 
$E_\infty$-spaces over different $E_\infty$-operads in this paper.

If $X$ is an $E_\infty$-space, then the map 
$\theta\colon \pt = \calC(0) \times X^0 \to X$
gives $X$ a basepoint $1$ invariant under the $\calC$-action on $X$.
Moreover, fixing $c_2 \in \calC(2)$,
we obtain on $X$ a multiplication
\[
	\cdot \colon X\times X \longto X,
	\qquad
	x\cdot y = \theta(c_2;x,y)
\]
which, up to homotopy, is independent of the element 
$c_2\in \calC(2)$ chosen.
The multiplication $\cdot$ and the unit $1$ make $X$
into a commutative monoid in $\ho(\spaces)$,
and consequently  induce on $H_\ast(X)$ 
the structure of a graded commutative
ring. We continue to write $\cdot$ and $1$ for the multiplication
and the unit for this ring structure as well.

As preparation for the definition of the $E$-operation,
let us now recall the homology of the symmetric group $\Sigma_p$.
Let $\pi$ be the cyclic group of order $p$, and choose a 
distinguished generator $t\in\pi$.
We remind the reader that the homology
of $\pi$ is one-dimensional in each degree.
As in \cite[Definition~1.2]{MaySteenrodOps},
from the standard  resolution
\[\xymatrix{
	\cdots
	\ar[r]
	&
	\F_p[\pi] e_3
	\ar[r]^{d}
	&
	\F_p[\pi] e_2
	\ar[r]^{d}
	&
	\F_p[\pi] e_1
	\ar[r]^{d}
	&
	\F_p[\pi] e_0
	\ar[r]^-{\varepsilon}
	&
	\F_p
}\]
of $\F_p$ by free $\F_p[\pi]$-modules
where 
\[
	d(e_{2i+2}) = (1+t+\cdots t^{p-1}) e_{2i+1},
	\quad
	d(e_{2i+1}) = (t-1)e_{2i},
	\quad\text{and}\quad	
	\varepsilon(e_0) =1
\]
for all $i\geq 0$,
we obtain for each $n\geq 0$ a distinguished generetor $e_n \in H_n(\pi)$.
These generators satisfy 
\[
	\beta(e_{2n}) = e_{2n-1}
\]
for all $n\geq 1$, where $\beta$ denotes the Bockstein operation, and
\begin{equation}
\label{eq:picoprodpodd} 
	\psi (e_{2n}) = \sum_{i+j=n} e_{2i} \tensor e_{2j}
	\quad\text{and}\quad
	\psi (e_{2n+1}) 
	= 
	\sum_{i+j=n} (e_{2i+1} \tensor e_{2j} + e_{2i} \tensor e_{2j+1})
\end{equation}
for $n\geq 0$
when $p$ odd,
\begin{equation}
\label{eq:picoprodp2} 
	\psi (e_n) = \sum_{i+j=n} e_{i} \tensor e_{j}
\end{equation}
for $n\geq 0$ when $p=2$, and
\begin{equation}
\label{eq:picounit}
 	\varepsilon(e_n) 
	=
    \begin{cases}
	1 & \text{if $n=0$}
	\\
	0 & \text{otherwise}
    \end{cases}
\end{equation}
for all $p$,
where $\psi$ is the coproduct on $H_\ast(\pi)$
induced by the diagonal map $\pi \to \pi \times \pi$
and $\varepsilon\colon H_\ast(\pi) \to \F_p$ is the counit induced by 
the map from $\pi$ to the trivial group.
(To match our statement about $\psi$ with 
\cite[Definition~1.2]{MaySteenrodOps},
notice that the map $\psi$ defined there is a diagonal approximation;
cf. \cite[Exercise for section~V.1]{Brown}.)
Writing $x \in H^1(\pi)$ for the dual of $e_1$ 
and $y\in H^2(\pi)$ for the dual of $e_2$,
we have 
\[
	H^\ast(\pi)
	= 
    \begin{cases}
    \F_2[x] & \text{if $p=2$}
    \\
    \Lambda(x) \tensor \F_p[y] & \text{if $p$ is odd}
    \end{cases}
\]
and $e_{2n}$ is the dual of $y^n$ and $e_{2n+1}$ is the dual of $xy^n$
for all $n\geq 0$
[$e_n$ is the dual of $x^n$ for all $n\geq 0$ when $p=2$].
Finally, equipping $H_\ast(\pi)$ with the ring structure
induced by the multiplication of $\pi$, we have 
\[
	H_\ast(\pi) 
	=
    \begin{cases}
	\Gamma^{}_{\F_2}(u) & \text{if $p=2$}
    \\
	\Lambda(u)\tensor \Gamma^{}_{\F_p}(v) & \text{if $p$ is odd} 
    \end{cases}
\]
where $u = e_1$ and $v=e_2$, and 
\[
	e_{2n} = v^{[n]}
	\qquad\text{and}\qquad
	e_{2n+1} = u v^{[n]}
\]
for all $n\geq 0$ [$e_n = u^{[n]}$ for all $n\geq 0$ when $p=2$]
where $a^{[n]}$ denotes the $n$-th divided power of $a$.


The homology of the symmetric group $\Sigma_p$ is now easily understood
in terms of the homology of $\pi$. Embed $\pi$ into $\Sigma_p$
by sending the chosen generator to the 
$p$-cycle $(12\ldots p)$, say. 
An easy transfer argument then shows that 
the inclusion of $\pi$ into $\Sigma_p$ induces an isomorphism
\[
	H^\ast(\Sigma_p) \xto{\ \isom\ } H^\ast(\pi)^{GL_1(\F_p)}
\]
where $GL_1(\F_p) = \F_p^\times$ acts on $\pi$ via the $\F_p$-vector space
structure on $\pi$. Moreover, it is easy to check that
\[
	H^\ast(\pi)^{GL_1(\F_p)}
	= 
    \begin{cases}
     \F_2[x] & \text{if $p=2$}\\
     \Lambda(xy^{p-2}) \tensor \F_p[y^{p-1}] & \text{if $p$ is odd}
    \end{cases}
\]
Dualizing, we see that the map induced  on homology
by the inclusion of $\pi$ into $\Sigma_p$ factors 
as an epimorphism followed by an isomorphism
\[
	H_\ast(\pi) 
	\longto 
	H_\ast(\pi)_{GL_1(\F_p)} 
	\xto{\ \isom\ }
	H_\ast(\Sigma_p)
\]
where
the first map is an isomorphism when $p=2$
and an isomorphism in degrees of the form
$2(p-1)n-\epsilon$ for $\epsilon\in\{0,1\}$ and $n\geq \epsilon$
and zero in all other degrees when $p$ is odd.

\begin{definition}
 For $p=2$, write $E^0_n$ for the image of $e_n$ in $H_\ast(\Sigma_p)$,
and for $p$ odd, write $E^0_n$ for the image of $e_{2(p-1)n}$
in $H_\ast(\Sigma_p)$ and $E^1_n$ for the image of $e_{2(p-1)n-1}$
in $H_\ast(\Sigma_p)$.
\end{definition}

Notice that 
\begin{equation}
\label{eq:e1bocke0}
	E^1_n = \beta E^0_n. 
\end{equation}
We are now ready to define the operations of our interest.

\begin{definition}
\label{def:eops}
Let $X$ be an $E_\infty$-space.
Let $k\geq 0$, and consider the maps
\[
	(\calC(k)/\Sigma_k) \times X 
	\xto{\ \delta\ }
	\calC(k)\times_{\Sigma_k} X^k 
	\xto{\ \theta\ } 
	X
\]
where $\delta$ is induced by the diagonal map $X \to X^k$ and $\theta$
is given by the action of $\calC$ on $X$. Observe that $\calC(k)/\Sigma_k$
is a model for $B\Sigma_k$, and define a pairing $\circ$ as the composite
\begin{equation}
\label{eq:pairingdef}
	\circ 
	\colon 
	H_\ast(\Sigma_k) \tensor H_\ast(X) 
	\xto{\ \times\ }
	H_\ast((\calC(k)/\Sigma_k) \times X) 
	\xto{\ \theta_\ast\delta_\ast\ }
	H_\ast(X).
\end{equation}
The $E$-\emph{operations} on $H_\ast(X)$ are the maps
\begin{equation}
\label{eq:e0def}
\begin{split}
	E_n^0 \circ &\colon H_\ast(X) \longto H_{\ast+2n(p-1)}(X)	
	\\
	[E_n^0 \circ &\colon H_\ast(X) \longto H_{\ast+n}(X)	
	\quad\text{if $p=2$}]
\end{split}
\end{equation}
for $n\geq 0$ and, when $p$ is odd, 
\begin{equation}
\label{eq:e1def}
	E_n^1 \circ \colon H_\ast(X) \longto H_{\ast+2n(p-1)-1}(X)
\end{equation}
for $n\geq 1$
obtained by taking $k=p$ and fixing the first 
variable in~\eqref{eq:pairingdef}.
\end{definition}

\noindent
Clearly the pairings \eqref{eq:pairingdef}
are natural with respect to maps of $E_\infty$-spaces,
wherefore the $E$-operations are.

Our next goal is to show how to express
the $E$-operations $E^\epsilon_n\circ$ in terms of the
Dyer--Lashof operations $Q^n$ and the Bockstein $\beta$ and vice versa.
Recall that we write $\calA_p$ for the mod $p$ Steenrod algebra,
and notice that dualizing the left action of $\calA_p$ on cohomology
yields a left action of the opposite algebra $\calA_p^\op$
on homology. We  write 
$a_\ast$ for the action of an element $a\in \calA_p^\op$ on homology,
so that 
\begin{equation}
\label{eq:astardef}
	\langle \phi, a_\ast(x) \rangle 
	= 
	(-1)^{\deg(a)\deg(\phi)} \langle a\phi,x\rangle
\end{equation}
where $x$ is a homology class, $\phi$ is a cohomology class,
and $\langle,\!\rangle$ denotes the pairing between homology and
cohomology.
We denote by $P^n$ the $n$-th Steenrod reduced $p$-th power operation
when $p$ is odd and the $n$-th squaring operation $\Sq^n$ when 
$p=2$, and write 
$\chi$ for the Hopf algebra antipode of $\calA_p^\op$.

\begin{proposition}
\label{prop:eqformulas}
For any $E_\infty$-space $X$, we have
\begin{align}
	\label{eq:e0q}
	(-1)^nE^0_n \circ x &= \sum_{k\geq 0} Q^{n+k} P^k_\ast (x)
	\\
	\label{eq:qe0}
	Q^n(x) &= \sum_{k\geq 0} (-1)^{n+k} E^0_{n+k}\circ (\chi P^k)_\ast(x)
	\\
\intertext{and, for $p$ odd,}
	\label{eq:e1q}
	(-1)^n E^1_n \circ x 
	&= 
	\sum_{k\geq 0} \beta Q^{n+k}P^k_\ast(x)
		-	
	\sum_{k\geq 0}  Q^{n+k}P^k_\ast \beta(x)	
	\\
	\label{eq:qe1}
 	\beta Q^n(x) 
	&= 
	\sum_{k\geq 0} (-1)^{n+k} E^1_{n+k} \circ (\chi P^k)_\ast (x)
		+
	\sum_{k\geq 0} (-1)^{n+k} E^0_{n+k} \circ \beta(\chi P^k)_\ast (x)
 \end{align}
for all $x \in H_\ast(X)$ and $n$.
\end{proposition}

\begin{proof}
Observe that $E^0_n \circ x$ may alternatively be described 
in terms of the maps 
\[
	(\calC(p)/\pi) \times X 
	\xto{\ \delta\ }
	\calC(p)\times_{\pi} X^p 
	\xto{\ \theta\ } 
	X
\]
where $\delta$ is induced by the diagonal $X \to X^p$
and $\theta$ is given by the action of $\calC$ on $X$ as 
\[
	E^0_n \circ x = \theta_\ast \delta_\ast(e_{2n(p-1)} \times x).
\] 
[When $p=2$, replace $e_{2n(p-1)}$ by $e_n$.]
Equation \eqref{eq:e0q} now follows from
\cite[Proposition 9.1]{MaySteenrodOps}
and the definition of the Dyer--Lashof operations \cite[p.7]{HILS}.
Equation \eqref{eq:qe0} follows formally from \eqref{eq:e0q}
using the properties of the antipode $\chi$.
Equation \eqref{eq:e1q} follows from the identity 
\[
	\beta( E^0_n \circ x)
	=
	 E^1_n \circ x +  E^0_n \circ \beta (x)
\]
and \eqref{eq:e0q}. Finally, equation \eqref{eq:qe1} follows by 
applying $\beta$ to \eqref{eq:qe0}.
\end{proof}

Proposition~\ref{prop:eqformulas} 
shows in particular that the Dyer--Lashof operations and the 
$E$-operations 
encode essentially the same information. 
We will now recast Proposition~\ref{prop:eqformulas}
into power series form which, as the proof makes clear,
in fact contains some more information. 

\begin{notation}
Define the formal power series
\[
	Q(s) = \sum_{n\geq 0} Q^n s^n,
	\qquad
	P_\ast(s) = \sum_{n\geq 0} P^n_\ast s^n,
	\qquad
	P^{\inv}_\ast(s) = \sum_{n\geq 0} (\chi P^n)_\ast s^n
\]
and, for $\epsilon \in \{0,1\}$,
\[
	E^{\epsilon}(s) = \sum_{n\geq \epsilon} E^\epsilon_n s^n.
\]
\end{notation}

By the defining property of $\chi$, we  have 
the power series identity
\[
	P_\ast(s) P^{\inv}_\ast(s) 
	= 
	P^{\inv}_\ast(s) P_\ast(s) 
	= 
	1.
\]

\begin{proposition} 
\label{prop:eqformulasps}
We have
\begin{align}
	\label{eq:e0qps}
	E^0(-s) \circ x &= Q(s)P_\ast(s^{-1}) x
	\\
	\label{eq:qe0ps}
	Q(s) x &= E^0(-s) \circ  P^{\inv}_\ast(s^{-1}) x
	\\
\intertext{and, for $p$ odd,}
	\label{eq:e1qps}
	E^1(-s) \circ x 
	&= 
	\beta Q(s) P_\ast(s^{-1}) x - Q(s) P_\ast(s^{-1}) \beta x 
	\\
	\label{eq:qe1ps}
	\beta Q(s) x
	&= 
	E^1(-s) \circ P^{\inv}_\ast(s^{-1})x
	   + 
	E^0(-s) \circ \beta P^{\inv}_\ast(s^{-1}) x 
\end{align}
for all $x\in H_\ast(X)$.
\end{proposition}

\begin{proof}
As in the proof of Proposition~\ref{prop:eqformulas},
\eqref{eq:qe0ps},
\eqref{eq:e1qps} and \eqref{eq:qe1ps}
follow from  \eqref{eq:e0qps}, 
so it is enough to show \eqref{eq:e0qps}.
Expanding the right hand side of \eqref{eq:e0qps}, we get
\begin{equation}
\label{eq:rhsexpanded} 
	Q(s) P_\ast(s^{-1})x 
	= \sum_{a,b\geq 0 } Q^a P^b_\ast(x) s^{a-b}
	= \sum_{n\in\Z}\, \sum_{b \geq \max\{0,-n\}} Q^{n+b} P^b_\ast(x)s^n
\end{equation}
It follows from \eqref{eq:e0q}
that the coefficients of $s^n$  on  the left and right hand sides of 
\eqref{eq:e0qps} agree when $n \geq 0$.
It remains to show that the coefficient of $s^n$ in 
\eqref{eq:rhsexpanded} is 0 when $n < 0$.

Let $\alpha$ and $u$ be commuting formal variables, and write
\[
	v = \alpha u,
	\quad
	s = u(1-\alpha^{-1})^{p-1} 
	= 
	u\left(\tfrac{1-\alpha}{\alpha}\right)^{p-1}
	\ \text{and}\quad
	t = v(1-\alpha)^{p-1} 
	= 
	u\alpha^p\left(\tfrac{1-\alpha}{\alpha}\right)^{p-1}.
\]
By Steiner's formulation of the Nishida relations 
\cite[equation (3a)]{Steiner}, we then have
\[
	P_\ast(u^{-1})Q(v)x = Q(t)P_\ast(s^{-1})x.
\]
Expanding the left and right hand sides gives
\begin{align*}
	P_\ast(u^{-1}) Q(v) x 
	&= 
	\sum_{k,\ell \geq 0} P^k_\ast Q^\ell(x) u^{-k} v^{\ell}
	= 
	\sum_{k,\ell \geq 0} P^k_\ast Q^\ell(x) u^{\ell-k} \alpha^{\ell}\\
	&=
	\sum_{n\in \Z} \, \sum_{\ell \geq \max\{0,n\}} P^{\ell-n}_\ast Q^\ell(x) u^n \alpha^\ell
\end{align*}
and
\begin{align*}
    Q(t)P_\ast(s^{-1}) x
    &= 
    \sum_{a,b\geq 0} Q^a P^b_\ast(x) t^a s^{-b}
    = 
    \sum_{a,b \geq 0}  Q^a P^b_\ast(x) u^{a-b} 
    \left(\frac{1-\alpha}{\alpha}\right)^{(p-1)(a-b)} \alpha^{pa}
    \\
    &=
    \sum_{n\in\Z}\, \sum_{b\geq\max\{0,-n\}} Q^{n+b}P^b_\ast(x)
    u^n\left(\frac{1-\alpha}{\alpha}\right)^{(p-1)n} \alpha^{p(n+b)}
\end{align*}
in $H_\ast(X)((\alpha,u))$.
Comparing the coefficients of $u^n$, we see that
\[
    \sum_{b\geq\max\{0,-n\}} Q^{n+b}P^b_\ast(x)
    \left(\frac{1-\alpha}{\alpha}\right)^{(p-1)n} \alpha^{p(n+b)}
    =
    \sum_{\ell \geq \max\{0,n\}} P^{\ell-n}_\ast Q^\ell(x)  \alpha^\ell
\]
in $H_\ast(X)((\alpha))$ for all $n \in \Z$.
Multiplying by $(1-\alpha)^{-(p-1)n}$ now yields
\[
 	\sum_{b\geq\max\{0,-n\}} Q^{n+b}P^b_\ast(x)
    \alpha^{p(n+b)-(p-1)n}
    =
    \sum_{\ell \geq \max\{0,n\}}
    P^{\ell-n}_\ast Q^\ell(x)  \alpha^\ell (1-\alpha)^{-(p-1)n}.
\]
When $n<0$, the left and right hand sides are polynomials in $\alpha$.
Substituting $\alpha = 1$, we see that
\[
	\sum_{b\geq\max\{0,-n\}} Q^{n+b}P^b_\ast(x) = 0
\]
when $n<0$, as desired.
\end{proof}

Notice that from equations
\eqref{eq:picoprodpodd}, \eqref{eq:picoprodp2}
and
\eqref{eq:picounit},
we obtain the following power series description of the 
coproduct and counit on $H_\ast(\Sigma_p)$ induced by 
the diagonal map of $\Sigma_p$ and 
the the homomorphism onto the trivial group, respectively:
\begin{equation}
\label{eq:e0coalgdesc} 
	\psi E^0(s) = E^0(s)\tensor E^0(s)
	\qquad\text{and}\qquad
	\varepsilon E^0(s) = 1
\end{equation}
for all $p$
and 
\begin{equation}
\label{eq:e1coalgdesc} 
	\psi E^1(s) = E^1(s)\tensor E^0(s) +  E^0(s)\tensor E^1(s)
	\qquad\text{and}\qquad
	\varepsilon E^1(s) = 0
\end{equation}
for $p$ odd.

Our next goal is to derive external and internal Cartan 
formulas for the $E$-operations.
If $X$ and $Y$ are $E_\infty$-spaces,
we equip the product $X\times Y$ with the structure of
an $E_\infty$-space via the coordinate-wise action of $\calC$. 
In terms of algebra structures over the monad $C$,
the $E_\infty$-structure on $X\times Y$ is given by the map
\[
	C(X\times Y)
	\xto{\ (C\pr_X,C\pr_Y)\ } 
	CX \times CY 
	\xto{\ \theta_X \times \theta_Y\ }
	X\times Y	
\]
where $\pr_X$ and $\pr_Y$ are the projections from 
$X\times Y$ onto $X$ and $Y$, respectively, 
and $\theta_X$ and $\theta_Y$ are the 
$C$-algebra structures on $X$ and $Y$, respectively.

\begin{proposition}[External Cartan formula]
\label{prop:extcartan}
Suppose $X$ and $Y$ are $E_\infty$-spaces. Then
\[
	E^0(s) \circ (x\times y) =  (E^0(s) \circ x) \times (E^0(s) \circ y)
\]
and, for $p$ odd,
\[
	E^1(s) \circ (x\times y) 
	=  
	(E^1(s) \circ x) \times (E^0(s) \circ y)
	+
	(-1)^{\deg(x)} 	(E^0(s) \circ x) \times (E^1(s) \circ y)
\]
in $H_\ast(X\times Y)\llbracket s\rrbracket$ for all $x \in H_\ast(X)$ and $y\in H_\ast(Y)$.
\end{proposition}
\begin{proof}
The $E_\infty$-structure on $X\times Y$ is given by the maps
\begin{equation}
\label{eq:prodcactn}
	\theta
	\colon
	\calC(n) \times (X\times Y)^n \longto X\times Y,
	\quad
	(c,(x_i,y_i)_{i=1}^n) 
	\mapsto 
	(\theta_X(c,(x_i)_{i=1}^n),\theta_Y(c,(y_i)_{i=1}^n))
\end{equation}
where $\theta_X$ and $\theta_Y$ are given by the $\calC$-actions on 
$X$ and $Y$, respectively. In particular, the composite
\[
	\calC(n)/\Sigma_n \times (X\times Y) 
	\xto{\ \delta\ }
	\calC(n) \times_{\Sigma_n} (X\times Y)^n
	\xto{\ \theta\ }
	X\times Y
\]
agrees with the composite
\begin{multline*}
	\calC(n)/\Sigma_n \times (X\times Y) 
	\xto{\ \Delta \times 1\ }
	\calC(n)/\Sigma_n \times \calC(n)/\Sigma_n \times X\times Y
	\\
	\xto{\ 1 \times \tau \times 1\ }
	\calC(n)/\Sigma_n \times X\times \calC(n)/\Sigma_n \times Y
	\xto{\ \theta_X \delta \times \theta_Y\delta\ }	
	X \times Y 
\end{multline*}
where $\tau$ denotes the map exchanging the two coordinates.
The claim now follows from equations
\eqref{eq:e0coalgdesc} and
\eqref{eq:e1coalgdesc} by taking $n=p$.
\end{proof}

\begin{remark}
\label{rk:genextcartan}
The proof of Proposition~\ref{prop:extcartan} 
in fact proves the following stronger result:
Suppose $X$ and $Y$ are $E_\infty$-spaces. Then 
\[
	r\circ (x\times y) 
	= 
	\sum (-1)^{\deg(x)\deg(r'')}(r'\circ x)\times (r'' \circ y)
\]
for all $x\in H_\ast(X)$, $y\in H_\ast(Y)$, and 
$r\in H_\ast(\Sigma_n)$, $n\geq 0$.
\end{remark}

\begin{lemma}
\label{lm:muopcompat}
Let $X$ be an $E_\infty$-space. 
Then the multiplication 
$\cdot \colon X\times X \to X$ on $X$ is compatible with 
the pairing $\circ$ of Definition~\ref{def:eops} in the sense
that the square
\begin{equation}
\label{sq:musq}
\vcenter{\xymatrix{
	\calC(k)/\Sigma_k \times (X \times X)
	\ar[r]^-{\theta\delta}
	\ar[d]_{\id \times \cdot}
	&
	X\times X
	\ar[d]^{\cdot}
	\\
	\calC(k)/\Sigma_k \times X 
	\ar[r]^-{\theta\delta}
	&
	X
}}
\end{equation}
commutes up to homotopy for every $k\geq 0$.
\end{lemma}

\begin{proof}
The composites through the top right hand and
bottom left hand corners in \eqref{sq:musq} are the maps
\[
	([c],a,b)
	\mapsto 
	\theta(c_2;\theta(c;a^k),\theta(c;b^k))
	=
	\theta(\gamma(c_2;c^k,c^k);a^k,b^k)
	=
	\theta(\gamma(c_2;c^k,c^k)\sigma;(a,b)^k)
\]
and
\[
	([c],a,b)
	\mapsto 
	\theta(c;\theta(c_2;(a,b))^k)
	=
	\theta(\gamma(c;c_2^p);(a,b)^k),
\] 
respectively, where $c_2\in \calC(2)$ is some fixed element,
$\gamma$ refers to composition in the operad $\calC$,
and $\sigma\in \Sigma_{2k}$ is the permutation corresponding
to the coordinate interchange bijection
\[
	\{1,\ldots,p\}\times \{1,2\} 
	\xto{\ \isom\ }
	\{1,2\} \times \{1,\ldots,k\}
\]
under the identification of the source and target 
with the set $\{1,\ldots,2k\}$ via lexicographic ordering of elements.
The two maps $\calC(k) \to \calC(2k)$ given by
\[
	c \longmapsto \gamma(c_2;c^k,c^k)\sigma
	\qquad\text{and}\qquad
	c \longmapsto \gamma(c;c_2^k)
\]
are both $\phi$-equivariant for the same homomorphism 
$\phi\colon\Sigma_k \to \Sigma_{2k}$,
namely the composite
\[
	\Sigma_k
	\longto 
	\Sigma_{\{1,\ldots,k\}\times\{1,2\}}
	\xto{\ \isom\ } 
	\Sigma_{2k}
\]
where the first map is given by $\sigma\mapsto \sigma\times \id$
and the second is induced by the lexicographic ordering on 
$\{1,\ldots,k\}\times\{1,2\}$.
By the assumption that $\calC$ is an $E_\infty$-operad, 
it follows that the two maps are $\Sigma_k$-equivariantly homotopic.
The claim follows.
\end{proof}

From Proposition~\ref{prop:extcartan}
and Lemma~\ref{lm:muopcompat}, we obtain

\begin{proposition}[Internal Cartan formula]
\label{prop:intcartan}
Suppose $X$ is an $E_\infty$-space. Then
\[
	E^0(s) \circ (xy) =  (E^0(s) \circ x) (E^0(s) \circ y)
\]
and, for $p$ odd,
\[
	E^1(s) \circ (x y) 
	=  
	(E^1(s) \circ x)  (E^0(s) \circ y)
	+
	(-1)^{\deg(x)} 	(E^0(s) \circ x) (E^1(s) \circ y)
\]
in $H_\ast(X)\llbracket s\rrbracket$ for all $x,y \in H_\ast(X)$. \qed
\end{proposition}

\begin{remark}
Later, in Theorems~\ref{thm:homologysuspension} 
and \ref{thm:transgressionproperty},
we will show that the $E$-operations  are stable
in the sense that they commute with the homology suspension map,
and that the operations commute with transgression 
in the Serre spectral sequence of the path--loop fibration
\[
	\loops X \longto PX \longto X
\]
of a simply-connected $E_\infty$-space $X$.
Compare with \cite[Theorem I.1.1.(7)]{HILS}.
\end{remark}

For later reference, we record here the (well-known)
action of $\calA_p^\op$ on $H_\ast(\Sigma_p)$.
\begin{notation}
Write 
\[
	\tilde{E}^0 (s) = E^0(s^{p-1})
\]
and, for $p$ odd,
\[
	\tilde{E}^1 (s)= s^{-1} E^1(s^{p-1}).
\]
\end{notation}
The following proposition follows easily by dualizing the
well-known action of $\calA_p$ on $H^\ast(\pi)$,
using $\calA_p^\op$-linearity of the map 
$H_\ast(\pi) \to H_\ast(\Sigma_p)$,
and making use of equation~\eqref{eq:e1bocke0}.

\begin{proposition}
\label{prop:steenrodaction}
For $p=2$, the action of $\calA_p^\op$ on $H_\ast(\Sigma_p)$
is determined by the identity
\[
	P_\ast(t) \tilde{E}^0(s) = \tilde{E}^0(s+s^2t).
\]
For $p$ odd, it is determined by the identities
\[
	P_\ast(t) \tilde{E}^\epsilon (s) = \tilde{E}^\epsilon (s+s^pt),
	\qquad
	\beta \tilde{E}^0 (s) = s \tilde{E}^1(s)
	\qquad\text{and}\qquad
	\beta \tilde{E}^1 (s) = 0
\]
for  $\epsilon\in\{0,1\}$.
Explicitly, for all $p$,
\[
	P^k_\ast E^0_\ell = {{(p-1)(\ell-k)}\choose k} E^0_{\ell-k}
\]
and, for $p$ odd,
\begin{gather}
	P^k_\ast E^1_\ell = {{(p-1)(\ell-k)-1}\choose k} E^1_{\ell-k},
	\\
	\beta E^0_n = 
        \begin{cases}
        E^1_n & \text{if $n\geq 1$}\\
        0	& \text{if $n=0$}
        \end{cases}
    \qquad\text{and}\qquad
    \beta E^1_n = 0
\end{gather}
for all $\ell$, $k$ and $n$. \qed
\end{proposition}
Our convention is that ${n\choose k} = \frac{n!}{k! (n-k)!}$
when $0\leq k \leq n$ and
${n\choose k} = 0$ otherwise. 
 

\section{The algebra of \texorpdfstring{$E$}{E}-operations}
\label{sec:cale}

Let $X$ be an $E_\infty$-space.
The $E$-operations 
$E_n^\epsilon\circ$ of Definition~\ref{def:eops}
make $H_\ast(X)$ into a module over the 
free algebra 
$TH_\ast (\Sigma_p) = \bigoplus_{k\geq 0} H_\ast(\Sigma_p)^{\tensor k}$
generated by $H_\ast(\Sigma_p)$.
We define the ring $\calE$ to be the quotient of $TH_\ast(\Sigma_p)$
by the relations that hold for the action of 
$TH_\ast(\Sigma_p)$ on the homology of all $E_\infty$-spaces.
In other words, $\calE$ is the quotient of $TH_\ast(\Sigma_p)$
by the ideal 
\[
	\bigcap_{\mathclap{\text{$X$ an $E_\infty$-space}}}\;
	\ker\big(TH_\ast(\Sigma_p) \longto \End_\ast(H_\ast(X))\big)
\]
where $\End_\ast(H_\ast(X))$ denotes the graded endomorphism ring of 
$H_\ast(X)$ and the map $TH_\ast(\Sigma_p) \to \End_\ast(H_\ast(X))$
is given by the $TH_\ast(\Sigma_p)$-module structure on $H_\ast(X)$. 
By construction, the 
ring $\calE$ acts naturally on the homology of all $E_\infty$-spaces;
it is the analogue, in the context of $E$-operations,
of the Dyer--Lashof algebra.
Our goal in this section is to describe the algebra $\calE$.
A description of the ring structure on $\calE$ is given by 
Theorems~\ref{thm:edescp2}, \ref{thm:ehatdescpodd} and 
\ref{thm:edescpodd2} below.

Our first aim is to identify $\calE$ as a subring of 
$\bigoplus_{k\geq 0}H_\ast(\Sigma_k)$ with respect to a 
ring structure we now describe.
It is easily verified that the homomorphisms
\begin{equation}
\label{eq:prodsigma}
	\circ
	\colon
	\Sigma_m \times \Sigma_n \longto \Sigma_{mn},
	\quad
	m,n \geq 0,
\end{equation}
obtained by choosing bijections
\begin{equation}
\label{eq:prodbij} 
	\{1,\ldots,m\} \times \{1,\ldots,n\} 
	\xto{\ \isom\ }
	\{1,\ldots,mn\}
\end{equation}
define on 
the direct sum
$\bigoplus_{k\geq0} H_\ast(\Sigma_k)$
a product $\circ$ making $\bigoplus_{k\geq0} H_\ast(\Sigma_k)$
into a ring;
notice that  
the homomorphisms \eqref{eq:prodsigma}
are, up to conjugacy, independent of the choice
of the bijections \eqref{eq:prodbij},
wherefore in particular the maps they induce on homology are
independent of the choices made.
The unit of the ring, denoted $[1]$,
is the canonical generator of $H_0(\Sigma_1)$.

\begin{proposition}
\label{prop:circaction}
Suppose $X$ is an $E_\infty$-space.
Then the pairings of equation~\eqref{eq:pairingdef} make $H_\ast(X)$
into a module over the ring 
$(\bigoplus_{k\geq0} H_\ast(\Sigma_k),\circ)$.
\end{proposition}
\begin{proof}
We need to verify the unitality and associativity of the pairings.
The composite
\[
	(\calC(1)/\Sigma_1) \times X 
	\xto{\ \delta\ }
	\calC(1) \times_{\Sigma_1} X
	\xto{\ \theta\ }
	X
\]
is homotopic to projection onto the second factor. Therefore 
$[1] \circ x = x$ for all $x\in H_\ast(X)$, showing unitality.
Expanding definitions, we see that the
composite
\[
	H_\ast(\Sigma_m) \tensor H_\ast(\Sigma_n) \tensor H_\ast(X)
	\xto{\ 1\tensor \circ\ }
	H_\ast(\Sigma_m)  \tensor H_\ast(X)
	\xto{\ \circ\ }
	H_\ast(X)
\]
is induced by the map
\[
	(\calC(m)/\Sigma_m) \times (\calC(n)/\Sigma_n) \times X \longto X,
	\quad
	([c_1],[c_2],x)
	\longmapsto 
	\theta(c_1, \theta(c_2,x^n)^m).
\]
By the axioms of operad actions, this map agrees with the map
\[
	([c_1],[c_2],x) \longmapsto (\theta(\gamma(c_1,c_2^m), x^{mn}))
\]
where $\gamma$ denotes composition in the operad $\calC$.
On homology, the latter map induces the composite
\[
	H_\ast(\Sigma_m) \tensor H_\ast(\Sigma_n) \tensor H_\ast(X)
	\xto{\ \circ \tensor 1 \ }
	H_\ast(\Sigma_{mn})  \tensor H_\ast(X)
	\xto{\ \circ\ }
	H_\ast(X),
\]
proving associativity.
\end{proof}

\begin{proposition}
\label{prop:for1}
Consider the free $E_\infty$-space  
\[
	C(\pt) = \bigsqcup_{k\geq 0} \calC(k)/\Sigma_k
\] 
generated by the one-point space $\pt$.
Let $x_0 \in H_\ast(C(\pt))$
be the canonical generator of 
$H_0(\calC(1)/\Sigma_1) \subset  H_\ast(C(\pt))$.
Then the $(\bigoplus_{k\geq0} H_\ast(\Sigma_k),\circ)$-module 
structure on $H_\ast(C(\pt))$ is free of rank $1$
with basis $x_0$.
\end{proposition}

\begin{proof}
For every  $k\geq0$, the map
\[
	H_\ast(\Sigma_k)
	\longto 
	H_\ast(C(\pt)),
	\quad
	r
	\longmapsto
	r\circ x_0
\]
is induced by the composite
\[
	\calC(k)/\Sigma_k 
	\isom
	\calC(k)/\Sigma_k \times \pt
	\xto{\ 1\times \tilde{x}_0\ }
	\calC(k)/\Sigma_k \times C(\pt)
	\xto{\ \delta\ }
	\calC(k) \times_{\Sigma_k} C(\pt)^k
	\xto{\ \theta\ }
	C(\pt)
\]
where $\tilde{x}_0\in C(\pt)$ 
is a point representing the class $x_0$.
But as $\tilde{x}_0$, we may choose the 
point given by the unit of the operad $\calC$,
in which case the above composite equals
the inclusion map 
$\calC(k)/\Sigma_k \incl C(\pt)$.
The claim follows.
\end{proof}

Write $\calE'$ for the subring of 
$(\bigoplus_{k\geq0} H_\ast(\Sigma_k),\circ)$
generated by $H_\ast(\Sigma_p)$,
and notice that $\calE'$ can equivalently 
be described as the image of the ring homomorphism 
$TH_\ast(\Sigma_p) \to (\bigoplus_{k\geq0} H_\ast(\Sigma_k),\circ)$
induced by the inclusion of $H_\ast(\Sigma_p)$ into 
$\bigoplus_{k\geq0} H_\ast(\Sigma_k)$.
By restriction of the module structure of 
Proposition~\ref{prop:circaction}, 
$\calE'$ acts naturally on the homology of every $E_\infty$-space.

\begin{proposition}
\label{prop:eprimeandeiso}
There exists an isomorphism $\calE' \xto{\ \isom\ }\calE$ of rings 
under $TH_\ast(\Sigma_p)$
under which the actions of $\calE$ and $\calE'$
on the homology of $E_\infty$-spaces agree.
\end{proposition}

\begin{proof}
Consider the diagram
\begin{equation}
\label{diag:thspend}
\vcenter{\xymatrix@C-2em{
	& TH_\ast(\Sigma_p)
	\ar[dl]
	\ar[dr]
	\\
	\calE'
	\ar@{-->}[rr]
	\ar[dr]
	&&
	\calE
	\ar[dl]
	\\
	&
	\End_\ast(H_\ast(X))
}}
\end{equation}
where $X$ is an $E_\infty$-space,
the maps from $TH_\ast(\Sigma_p)$ are the canonical 
epimorphisms, and the maps into $\End_\ast(H_\ast(X))$
are given by the $\calE$- and $\calE'$-module structures on $H_\ast(X)$.
The square of solid arrows commutes since 
the two composites of solid maps from $TH_\ast(\Sigma_p)$
to $\End_\ast(H_\ast(X))$ agree when restricted to 
$H_\ast(\Sigma_p)$.
By the construction of $\calE$, there results a ring epimorphism
$\calE' \to \calE$ making both of the triangles in the diagram
commute for every $E_\infty$-space $X$.
By Proposition~\ref{prop:for1}, the map 
$\calE' \to \End_\ast(H_\ast(X))$ 
in diagram~\eqref{diag:thspend} 
is a monomorphism when $X = C(\pt)$,
so the epimorphism $\calE' \to \calE$ is also a monomorphism
and hence an isomorphism.
\end{proof}

In view of Proposition~\ref{prop:eprimeandeiso},
we will from now on identify the ring 
$\calE$ with $\calE'$,
and consider $\calE$ as a subring of 
$(\bigoplus_{k\geq0} H_\ast(\Sigma_k),\circ)$.

\begin{remark}
\label{rk:largerring}
By Proposition~\ref{prop:eprimeandeiso},
the action of $\calE$ on the homology of 
$E_\infty$-spaces extends to an action of the larger ring
$(\bigoplus_{k\geq0} H_\ast(\Sigma_k),\circ)$.
We will return to the action of 
$(\bigoplus_{k\geq0} H_\ast(\Sigma_k),\circ)$
on the homology of $E_\infty$-spaces later in 
section~\ref{sec:coalgebraicalgebra}.
At this point, it suffices to say that the action of 
the larger ring is completely determined by that of $\calE$. 
See Remark~\ref{rk:largerring2}.
\end{remark}

The ring $\calE$ has more structure than the structure of a ring.
First, notice that the ring $\calE$ is bigraded with respect to 
homological degree and \emph{length}, the elements of degree $d$ and 
length $\ell$ being those belonging to $\calE \cap H_d(\Sigma_{p^\ell})$.
We write $\deg(r)$ and $\len(r)$ for the degree and length
of an element $r \in \calE$, respectively, and $\calE[n]$ for the 
subspace of length $n$ elements of $\calE$, so that 
$\calE[n] = \calE \cap H_\ast(\Sigma_{p^n})$. As the ring 
$(\bigoplus_{k\geq0} H_\ast(\Sigma_k),\circ)$
is graded commutative with respect to homological degree,
so is $\calE$: we have
\[
	r \circ s = (-1)^{\deg(r)\deg(s)} s\circ r
\]
for all $r,s\in \calE$. To discuss further structure on $\calE$,
we digress for a moment to discuss bialgebras.

\begin{definition}
\label{def:bialgebra}
Recall that by a  coalgebra, we mean a graded 
graded cocommutative coassociative counital
coalgebra over $\F_p$.
By a (commutative) \emph{bialgebra} we mean a (commutative) monoid in the 
category $\coalgebras$ of coalgebras, 
where the monoidal structure on 
$\coalgebras$ is given by tensor product.
\end{definition}

Thus a bialgebra $\calB$ consists of a graded 
$\F_p$ vector space $\calB$
together with a coproduct  and counit
\[
	\psi\colon \calB \longto \calB\tensor \calB
	\qquad\text{and}\qquad
	\varepsilon \colon \calB\longto \F_p
\]
as well as a product and unit
\[
	\mu\colon \calB\tensor \calB \longto \calB
	\qquad\text{and}\qquad
	\eta \colon \F_p \longto \calB
\]
such that $\psi$ and $\varepsilon$
make $\calB$ into a (cocommutative, coassociative and counital)
coalgebra,
$\mu$ and $\eta$ make $\calB$ into an (associative and unital,
but not necessarily commutative)
algebra, and such that $\mu$ and $\eta$ are homomorphisms of
coalgebras (or, equivalently,
$\psi$ and $\varepsilon$ are homomorphisms of algebras).

\begin{example}
The mod $p$ Steenrod algebra $\calA_p$ is a noncommutative bialgebra, 
as is its opposite $\calA_p^\op$.
\end{example}

\begin{example}
\label{ex:hastxbialg}
Suppose $X$ is an $E_\infty$-space. Then the 
product $\cdot$ and unit $1$ given by the $E_\infty$-structure on $X$
as well as the coproduct $\psi$ and counit $\varepsilon$
induced by the diagonal map of $X$ and the map from $X$ to 
the one-point space make $H_\ast(X)$ into a commutative bialgebra.
\end{example}

If $\calB$ is a bialgebra, then the tensor product of 
the tensor product $M\tensor N$ 
of two graded $\calB$-modules also has the structure of 
an $\calB$-module, with product given by
\begin{equation}
\label{eq:tensoraction} 
	r\circ (x\tensor y) 
	= 
	\sum (-1)^{\deg(r'')\deg(x)}(r'\circ x) \tensor (r''\circ y)
\end{equation}
for $r\in\calB$, $x\in M$ and $y\in N$.

\begin{definition}
Suppose $\calB$ is a bialgebra. By a $\calB$-\emph{algebra}
we mean an $\F_p$-algebra $A$ such that $A$ is a $\calB$-module
and such that the product
$A\tensor A \to A$
and unit  $\F_p \to A$ of $A$
are $\calB$-module maps.
Similarly, by a $\calB$-\emph{bialgebra}, we mean a bialgebra $A$
which is a module over $\calB$ and whose 
product, unit, coproduct, and counit are 
$\calB$-module maps.
\end{definition}

\begin{example}
Let $X$ be an $E_\infty$-space.
As  all of the bialgebra structure maps 
on $H_\ast(X)$ in Example~\ref{ex:hastxbialg}
are induced by maps of spaces
and hence are $\calA_p^\op$-linear,
$H_\ast(X)$ is an $\calA_p^\op$-bialgebra.
\end{example}

We return to our discussion of the ring $\calE$.

\begin{example}
\label{ex:extrastrone}
The diagonal maps 
$\Sigma_k \to \Sigma_k\times \Sigma_k$
and the unique homomorphism from $\Sigma_k$ 
to the trivial group induce on 
$\bigoplus_{k\geq0} H_\ast(\Sigma_k)$
a coalgebra structure
making it into a bialgebra (with product given by $\circ$).
It is readily checked that 
$\calE \subset \bigoplus_{k\geq0} H_\ast(\Sigma_k)$
is a subbialgebra.
Concretely, the coproduct $\psi$ and the counit $\varepsilon$ 
on $\calE$ are determined on generators
by equations
\eqref{eq:e0coalgdesc}
and 
\eqref{eq:e1coalgdesc}
and on products of generators by multiplicativity of 
$\psi$ and $\varepsilon$.
Moreover, $\calE$ is closed under the action 
of the opposite Steenrod algebra $\calA_p^\op$ on 
$\bigoplus_{k\geq0} H_\ast(\Sigma_k)$,
and all of the bialgebra structure maps on $\calE$
are $\calA_p^\op$-linear (as they are ultimately induced by 
space level maps). Thus $\calE$ is an $\calA_p^\op$-bialgebra.
\end{example}

In particular, we have

\begin{proposition}
\label{prop:alginaopmods}
Let $r_1,r_2 \in \calE$. Then for all $a\in \calA_p^\op$
\[
	a_\ast(r_1\circ r_2) = \sum a'_\ast r_1 \circ  a''_\ast r_1.
\]
In particular, for all $k\geq 0$
\[
	P^k_\ast (r_1 \circ r_2) 
	= 
	\sum_{k_1 + k_2 = k} P^{k_1}_\ast r_1 \circ P^{k_2}_\ast r_2.
\]
and, when $p$ is odd,
\begin{equation*}
	\pushQED{\qed}
	\beta (r_1 \circ r_2) 
	= 
	\beta (r_1) \circ r_2 + (-1)^{\deg r_1} r_1 \circ \beta(r_2)
	\qedhere
	\popQED
\end{equation*}
\end{proposition}

We now turn to the action of $\calE$ on the homology 
of an $E_\infty$-space $X$. 
As the product $\circ \colon \calE\tensor H_\ast(X) \to H_\ast(X)$
also ultimately arises from space level maps,
it too is $\calA_p^\op$-linear. Thus we also have

\begin{proposition}
\label{prop:modinaopmods}
Let $X$ be an $E_\infty$-space, and 
let $r \in \calE$ and $x\in H_\ast(X)$.
Then for all $a\in \calA_p^\op$
\[
	a_\ast(r\circ x) = \sum a'_\ast r \circ  a''_\ast x.
\]
In particular, for all $k\geq 0$
\[
	P^k_\ast (r \circ x) 
	= 
	\sum_{k_1 + k_2 = k} P^{k_1}_\ast r \circ P^{k_2}_\ast x.
\]
and, when $p$ is odd,
\begin{equation*}
	\pushQED{\qed}
	\beta (r \circ x) 
	= 
	\beta (r) \circ x + (-1)^{\deg r} r \circ \beta(x)
	\qedhere
	\popQED
\end{equation*}
\end{proposition}

\begin{remark}
Propositions~\ref{prop:alginaopmods} and 
\ref{prop:modinaopmods} contrast 
with the properties of the Dyer--Lashof operations:
the Dyer--Lashof algebra is not an algebra over $\calA_p^\op$,
and instead of the formulas of Proposition~\ref{prop:modinaopmods},
the interaction of the Dyer--Lashof operations with
the Steenrod operations is described by the Nishida relations.
\end{remark}

The following result follows from Proposition~\ref{prop:extcartan}.
\begin{proposition}
\label{prop:eactiononprod}
Let $X$ and $Y$ be $E_\infty$-spaces.
Then the $\calE$-action on $H_\ast(X) \tensor H_\ast(Y)$ 
arising from the induced $E_\infty$-space structure on $X\times Y$
agrees with the tensor product action described in 
equation~\eqref{eq:tensoraction}	\qed
\end{proposition}

The following proposition augments Example~\ref{ex:hastxbialg}.
\begin{proposition}
\label{prop:ebialgstr}
Let $X$ be an $E_\infty$-space. 
Then $H_\ast(X)$ is an $\calE$-bialgebra, so that 
\begin{align}
	\label{eq:elinprod}
	r\circ (xy) &= \sum_{\psi(r)} (-1)^{\deg(r'')\deg(x)} (r'\circ x)(r''\circ y)
	\\
	\label{eq:elinunit}
	r\circ 1 &= \varepsilon(r)1
	\\
	\label{eq:elincoprod}
	\psi(r \circ x) 
	&= 
	\sum_{\psi(r)}\sum_{\psi(x)} (-1)^{\deg(r'')\deg(x')}(r'\circ x') \tensor(r''\circ x'')
	\\
	\label{eq:elincounit}
	\varepsilon(r\circ x) &= \varepsilon(r)\varepsilon(x)
\end{align}
for all $r\in \calE$ and $x,y\in H_\ast(X)$.
\end{proposition}
\begin{proof}
Our task is to show that the bialgebra structure maps on $H_\ast(X)$
are all $\calE$-linear. 
Observe that the one-point space is trivially a $\calC$-algebra.
The basepoint inclusion map $\pt\to X$, the unique map $X\to \pt$,
and the diagonal map $X \to X\times X$ are all maps of $\calC$-algebras,
implying the $\calE$-linearity of the unit, counit, and coproduct
maps of $H_\ast(X)$. On the other hand, 
the $\calE$-linearity of the product map 
on $H_\ast(X)$ follows from Proposition~\ref{prop:intcartan}.
Equations \eqref{eq:elinprod} through \eqref{eq:elincounit} follow.
\end{proof}

To aid our study of the ring $\calE$, we will now introduce
a closely related ring $\hat\calE$ which surjects onto $\calE$.
Write $V_n = \pi^n$ for the rank $n$ elementary abelian $p$-group,
and consider the free algebra
\[
	TH_\ast(\pi) 
	= 
	\bigoplus_{n\geq 0} H_\ast(\pi)^{\tensor n} 
	= 
	\bigoplus_{n\geq 0} H_\ast(V_n)
\]
generated by $H_\ast(\pi)$ on $H_\ast(X)$,
where the second identification is provided by the Künneth theorem.
The ring $TH_\ast(\pi)$ is also bigraded, with the 
the two gradings given by homological degree and the direct 
sum decomposition above,
and the epimorphism $H_\ast(\pi) \to H_\ast(\Sigma_p)$ induces
an epimorphism 
\begin{equation}
\label{eq:thpitoe}
	TH_\ast(\pi) \longto \calE
\end{equation}
of bigraded rings.
On $H_\ast(V_n)$, the epimorphism is induced by the composite 
\begin{equation}
\label{eq:vntospn}
	V_n = \pi^n \longincl \Sigma_p^n \longincl \Sigma_{p^n}
\end{equation}
where the first homomorphism is an $n$-fold product of the 
inclusion of $\pi$ into $\Sigma_p$, and the second one is 
obtained by iterating homomorphisms of 
equation~\eqref{eq:prodsigma}. Alternatively, the 
composite \eqref{eq:vntospn} can be described (up to conjugacy)
as the composite
\begin{equation}
\label{eq:vntos}
	 V_n \longincl \Sigma_{V_n} \xto{\ \isom\ } \Sigma_{p^n}
\end{equation}
where the first map is the Cayley embedding sending an element 
$v\in V_n$ to the bijection given by multiplication by $v$,
and where the second map is induced by the choice of a bijection
between elements of $V_n$ and the set $\{1,2,\ldots p^n\}$.
From this point of view, it is clear that the image of $V_n$
in $\Sigma_{p^n}$ is normalized by a copy of $GL_n(\F_p)$
sitting inside $\Sigma_{p^n}$, with $GL_n(\F_p)$ acting 
on $V_n$ as the automorphism group of $V_n$.
Thus, on $H_\ast(V_n) \subset TH_\ast(\pi)$, 
the map of equation~\eqref{eq:thpitoe} factors through 
the $GL_n(\F_p)$-coinvariants $H_\ast(V_n)_{GL_n}$,
and we obtain a factorization of \eqref{eq:thpitoe}
into a composite
\begin{equation}
\label{eq:coinvfact}
	T H_\ast(\pi) 
	= 
	\bigoplus_{n\geq 0} H_\ast(V_n)
	\longto
	\bigoplus_{n\geq 0} H_\ast(V_n)_{GL_n}
	\longto
	\bigoplus_{n\geq 0} \calE[n] 
	= 
	\calE
\end{equation}
of two epimorphism. 

\begin{definition}
We let 
\[
	\hat\calE = \bigoplus_{n\geq 0} H_\ast(V_n)_{GL_n}
	\qquad\text{and}\qquad
	\hat\calE[n] = H_\ast(V_n)_{GL_n}.
\]

\end{definition}
It is easy to verify that the bigraded
ring structure on $T H_\ast(\pi)$ induces on
$\hat\calE$ the structure of a bigraded ring
making both of the  morphisms in \eqref{eq:coinvfact} 
into morphisms of bigraded rings. In particular,
we have a bigraded ring epimorphism 
\[
	\hat\calE \longto \calE
\]
induced by the homomorphisms of equation~\eqref{eq:vntos}.
As in $\calE$, we write $\circ$ and $[1]$ 
for the product and the unit of $\hat\calE$,  respectively,
and we refer to the two gradings as length and degree. At length 1,
the map $\hat\calE \to \calE$ is an isomorphism, and via
this isomorphism, we interpret the generators $E^\epsilon_n \in \calE$
as elements of $\hat\calE$.

When $p=2$, the map $\hat\calE \to \calE$ is an isomorphism---see 
e.g.\ \cite[Corollary 3.28]{MadsenMilgram}.
For $p$ odd, the two rings are not isomorphic, however. 
The following two results follow from the arguments of 
\cite[Section~4]{Turner} for $p=2$ and \cite[Section~4]{Chon} for $p$ odd.
In essence, one uses the transformation 
$T = \left(\begin{smallmatrix} 1 & 1 \\ 0 & 1 \end{smallmatrix}\right) \in GL_2(\F_p)$
to show that the claimed relations hold in 
$H_\ast(V_2)_{GL_2}$, and the fact that
$\Sigma_n \leq GL_n(\F_p)$, $T\oplus I_{n-2}$ and $a \oplus I_{n-1}$
for $a\in \F_p^\times$
generate $GL_n(\F_p)$ to show that no further relations are
necessary. Here $\oplus$ refers to block sum of matrices:
$A\oplus B = \left(\begin{smallmatrix} A & 0 \\ 0 & B \end{smallmatrix}\right)$.

\begin{theorem}
\label{thm:edescp2}
When $p=2$, the ring $\calE = \hat\calE$ is the polynomial algebra
\[
  	\F_2[E^0_n \mid n\geq 0]
\]
modulo all relations implied by the power series identity
\begin{equation}
   	E^0(s) \circ E^0(t) = E^0(s) \circ E^0(s+t).
\end{equation}
\qed
\end{theorem}

\begin{theorem}
\label{thm:ehatdescpodd}
When $p$ is odd, the ring $\hat\calE$ is the free graded
commutative algebra
\[
	F_{\F_p}(E^\epsilon_n \mid \epsilon \in \{0,1\},\, n\geq \epsilon)
\]
modulo all relations implied by the power series identities
\begin{align}
	\label{eq:ehrel00}
	\tilde{E}^{0}(s) \circ \tilde{E}^{0}(t) 
	&=
	\tilde{E}^{0}(s) \circ \tilde{E}^{0}(s+t) 
	\\
	\label{eq:ehrel01}
	\tilde{E}^{0}(s) \circ \tilde{E}^{1}(t) 
	&=
	\tilde{E}^{0}(s) \circ \tilde{E}^{1}(s+t) 
	\\   
	\label{eq:ehrel10}
	\tilde{E}^{1}(s) \circ \tilde{E}^{0}(t) 
	&=
	\tilde{E}^{1}(s) \circ \tilde{E}^{0}(s+t) 
	+
	\tilde{E}^{0}(s) \circ \tilde{E}^{1}(s+t) 
	\\
   	\label{eq:ehrel11}
	\tilde{E}^{1}(s) \circ \tilde{E}^{1}(t) 
	&=
	\tilde{E}^{1}(s) \circ \tilde{E}^{1}(s+t).
\end{align}
\qed
\end{theorem}

\begin{remark}
Ch{\ohorn}n omits relation \eqref{eq:ehrel10} in the statement
of \cite[Proposition 4.7]{Chon}.
The relation is not a consequence of \eqref{eq:ehrel00}, \eqref{eq:ehrel01}
and \eqref{eq:ehrel11}, however, as is easily seen e.g.\ by considering 
$H_7(V_2)$ in the case $p=3$.
The relation follows by comparing the coefficients of $\epsilon$
on the two sides of the equation on p.\ 367, line $-12$ of \cite{Chon}.
\end{remark}

\begin{definition}
\label{def:bocksteindegree}
We define on $\hat\calE$ a third grading, the \emph{Bockstein grading},
by declaring the \emph{Bockstein degree} of each generator $E^\epsilon_n$ 
of $\hat\calE$ to be $\epsilon$. We write 
$\deg_\beta(r)$ for the Bockstein degree of an element $r\in \hat\calE$.
\end{definition}

Notice that the Bockstein degree of an element is bounded from above
by the length of the element. 
In view of Theorems~\ref{thm:edescp2} and \ref{thm:ehatdescpodd}, 
Definition~\ref{def:bocksteindegree} 
does yield a well-defined grading on $\hat\calE$ 
since the relations imposed are homogeneous
with respect to Bockstein degree. Of course, the Bockstein degree
is only interesting for odd primes $p$; for $p=2$,
the whole ring $\hat\calE$ is concentrated in Bockstein degree $0$.

We now proceed to describe additive bases for $\calE$ and $\hat\calE$.

\begin{definition}
\label{def:seqdefs}
Consider a sequence of pairs of integers
\begin{equation}
\label{eq:Iseq} 
	I = ((\epsilon_1,i_1),\ldots,(\epsilon_n,i_n)).
\end{equation}
We call $I$ \emph{legitimate} if 
 $\epsilon_s \in \{0,1\}$ for every $s$
[if $\epsilon_s = 0$ for all $s$ when $p=2$]
and $i_s \geq \epsilon_s$ for all $s$. For such $I$, we
write $E_I$ for the product
\[
	E_I = E^{\epsilon_1}_{i_1} \circ  \cdots \circ E^{\epsilon_n}_{i_n}
\]
(which we will interpret as an element of either $\calE$ or $\hat\calE$
depending on context).
Taking our cue from the terminology for elements of $\calE$ and $\hat\calE$,
we call $n$ the \emph{length} of $I$, and the sum
$\epsilon_1 + \cdots + \epsilon_n$, denoted $\deg_\beta(I)$,
the \emph{Bockstein degree} of $I$. 
We define the \emph{minimum} of $I$, denoted $\min(I)$, to be the 
minimum $\min\{i_1,\ldots i_n\}$,
and define 
\[
	m(I) 
	= 
    \begin{cases}
     1 & \text{if $\epsilon_s = 1$ for some $1\leq s \leq n$} \\
     0 & \text{otherwise}
    \end{cases}
\]
and
\[
	b(I)
	= 
	\begin{cases}
	\epsilon_1 & \text{if $n > 0$}
	\\
    0 & \text{if $n=0$}
    \end{cases}
\]
 We call $I$ \emph{ascending} if 
\[
	i_1 \leq i_2 \leq \cdots \leq i_n.
\]
In particular, and by convention, the empty sequence $\emptyset$
is an ascending legitimate sequence having
length $0$, minimum $+\infty$, Bockstein degree $0$, and 
$m(\emptyset) = 0$. We interpret
$E_\emptyset = [1]$.

For $I$ as above,
we write $\langle I \rangle$ for the sequence
\[
	\langle I \rangle 
	= 
	\big((\epsilon_1, i_1), (\epsilon_2, pi_2 - \epsilon_2),\ldots,
	 (\epsilon_n, p^{n-1} i_n - \epsilon_n \tfrac{p^{n-1}-1}{p-1})\big).
\]
A sequence $J$ is called \emph{allowable}  if
$J =\langle I \rangle$
for some legitimate ascending sequence $I$.
In particular, the empty sequence is allowable.
\end{definition}

%

The following result can be read off from
\cite[Theorem 3.6 and Proposition 3.7]{Chon}
for $p$ odd
and from \cite[Theorem 3.10]{Turner}
for $p=2$. We remind the reader that 
$\hat\calE = \calE$ when $p=2$.

\begin{theorem}
\label{thm:eandehatbases}
\mbox{}
\begin{enumerate}
\item 
    For every $n\geq 0$, the elements $E_J\in \hat\calE[n]$ 
    for $J$ an allowable sequence of length $n$ 
    satisfying $\min(J) \geq m(J)$
    form a basis for  $\hat\calE[n]$ as an $\F_p$-vector space.
\item
	For every $n\geq 0$, the elements $E_J\in \calE[n]$ 
    for $J$ an allowable sequence of length $n$ 
    satisfying $\min(J) \geq \deg_\beta(J)/2$
    form a basis for $\calE[n]$ as an $\F_p$-vector space.
    \qed
\end{enumerate}
\end{theorem}

\begin{corollary}
\label{cor:edesc}
The ring $\calE$ is the quotient of  the ring $\hat\calE$ obtained by 
imposing the relation 
$E_J = 0$ for all allowable
$J$ with $m(J) \leq \min(J)  < \deg_\beta(J)/2$. \qed
\end{corollary}

As the relations imposed in Corollary~\ref{cor:edesc}
are homogeneous with respect to 
Bockstein degree, we conclude 
that the Bockstein 
grading on $\hat\calE$ passes to a grading
on the quotient $\calE$ of $\hat\calE$.
We refer to this grading as the Bockstein 
grading, and write $\deg_\beta(r)$ for the Bockstein
degree of an element $r\in\calE$.

Our next aim is to obtain a more flexible description of $\calE$ as
a quotient of $\hat\calE$. This description is given in 
Theorem~\ref{thm:edescpodd2} below.
Recall that by properties of Steenrod power operations, 
for any space $Z$ and class $z\in H_\ast(Z)$,
$P^\ell_\ast(z) \neq 0$ implies $2p\ell \leq \deg(z)$
[$2\ell \leq \deg(z)$ if $p=2$].
For elements of $\calE$, it is possible 
to obtain a better bound. We remind the reader of our 
convention that given elements of multigraded objects
are assumed to be homogeneous with respect to all gradings
in sight.
\begin{proposition}
\label{prop:bocksteinbound}
Suppose  $r\in \calE$. Then $P_\ast^\ell(r) \neq 0$ implies
\[
	2p\ell \leq \deg(r) - \deg_\beta(r)
	\qquad 
	[2\ell \leq \deg(r) \text{ when $p=2$}].
\]
\end{proposition}
\begin{proof}
When $p=2$, there is nothing to prove. Suppose $p$ is odd.
It is enough to consider elements $r$ of the form
\[
	r = E^1_{i_1} \circ \cdots \circ E^1_{i_k} \circ r'
\]
where $k = \deg_\beta(r)$, $1\leq i_1,\ldots,i_k$, and $r'\in \calE$.
For such an $r$, we have
\begin{align*}
 	P^\ell_\ast (r)
	&= 
	P^\ell_\ast (E^1_{i_1} \circ \cdots \circ E^1_{i_k} \circ r')
	\\
	&= \,\sum_{\mathclap{\substack{
			\ell_1,\ldots,\ell_{k+1}\,\geq\, 0 
			\\ 
			\ell_1+\cdots+\ell_{k+1} \,=\, \ell}
	   }}\, P^{\ell_1}_\ast E^1_{i_1} 
	      \circ \cdots \circ 
	      P^{\ell_k}_\ast E^1_{i_k}
	      \circ
	      P^{\ell_{k+1}}_\ast r'
	\\
	&= \,\sum_{\mathclap{\substack{
			\ell_1,\ldots,\ell_{k+1}\,\geq\, 0 
			\\ 
			\ell_1+\cdots+\ell_{k+1} \,=\, \ell}
	   }}\;\;\;\; 
	     \textstyle	   
	   	{(p-1)(i_1-\ell_1) - 1 \choose \ell_1} E^1_{i_1-\ell_1} 
	      \circ \cdots \circ 
	      {(p-1)(i_k-\ell_k) - 1 \choose \ell_k} E^1_{i_k-\ell_k}
	      \circ
	      P^{\ell_{k+1}}_\ast r',    
\end{align*}
where the second equality holds by 
Propositions~\ref{prop:alginaopmods} and \ref{prop:modinaopmods}
and  the last equality follows from Proposition~\ref{prop:steenrodaction}.
If $P^\ell_\ast(r) \neq 0$, the above sum must contain a term with
\[
	\ell_s \leq (p-1)(i_s-\ell_s)-1 
	\quad\text{for $1\leq s \leq k$}
	\quad\text{and}\quad
	2p\ell_{k+1} \leq \deg(r').
\]
or, what is the same,
\[
	p\ell_s \leq (p-1)i_s-1 
	\quad\text{for $1\leq s \leq k$}
	\quad\text{and}\quad
	2p\ell_{k+1} \leq \deg(r').
\]
Therefore
\begin{align*}
	2p\ell 
	&= 2p(\ell_1 + \cdots + \ell_{k+1})
	\\
	&\leq 2(p-1)(i_1 +  \cdots + i_k) - 2k + \deg(r')
	\\
	&= \deg(E^1_{i_1} \circ \cdots \circ E^1_{i_k}) - k + \deg(r')
	\\
	&= \deg(r) - \deg_\beta(r). 
\end{align*}
as desired.
\end{proof}

\begin{proposition}
\label{prop:lowencong}
Suppose $X$ is an $E_\infty$-space.
Let $x \in H_\ast(X)$
and  $r\in \calE$,
and suppose $\epsilon \in \{0,1\}$ and 
\[
	\epsilon \leq n \leq \frac{\deg_\beta(r) + \deg(x)}{2}
	\qquad
\]
[$\epsilon=0$ and $0\leq n \leq \deg(x)$ when $p=2$].
Then
\[
	(-1)^n E^\epsilon_n \circ r \circ x 
	\equiv
    \begin{cases}
     (P^\ell_\ast(r) \circ x)^p 
     &
     \text{\footnotesize\parbox{4.6cm}{if 
     $\epsilon=0$,
     $\deg_\beta(r) + \deg(x) = 2n$ and 
     $\deg(r) - \deg_\beta(r) = 2p\ell$}}
     \\
     0 & \text{otherwise}
    \end{cases}
\]
when $p$ is odd and
\[
	E^0_n \circ r \circ x
	\equiv 
	\begin{cases}
    (P^\ell_\ast(r) \circ x)^2 
     &
     \text{{if 
     $\deg(x)=n$ and 
     $\deg(r) = 2\ell$}}
     \\
     0 & \text{otherwise}
    \end{cases}
\]
when $p=2$, where the congruences are
modulo elements of the form $s\circ a_\ast x$ with $s\in \calE$
and $a\in \calA_p^\op$, $\deg(a)>0$.
\end{proposition}
\begin{proof}
We present the proof in the case of an odd prime $p$,
leaving the minor modifications required in the case $p=2$
to the reader.
Consider first the case $\epsilon = 0$.
By Propositions~\ref{prop:eqformulas} and \ref{prop:modinaopmods},
we have
\begin{equation}
\label{eq:congcomp}
\begin{aligned}
	& (-1)^n E^0_n \circ r \circ x 
	\\
	&= \sum_{\ell\geq 0} Q^{n+\ell} P^\ell_\ast(r\circ x)
	\\
	&= \sum_{\mathclap{\ell_1,\ell_2 \geq 0}} Q^{n+\ell_1+\ell_2} 
		(P^{\ell_1}_\ast(r) \circ P^{\ell_2}_\ast(x))
	\\
	&= \sum_{\ell \geq 0} Q^{n+\ell} (P^\ell_\ast(r)\circ x)
		+
		\,\sum_{\mathclap{\ell_1 \geq 0,\,\ell_2 \geq 1}}\, Q^{n+\ell_1+\ell_2} 
		(P^{\ell_1}_\ast(r) \circ P^{\ell_2}_\ast(x)) 
	\\
	&= \sum_{\ell \geq 0} Q^{n+\ell} (P^\ell_\ast(r)\circ x)
		+
		\,\sum_{\mathclap{\ell_1 \geq 0,\,\ell_2 \geq 1,\,t\geq 0}}\,
		 E^0_{n+\ell_1+\ell_2+t} 
		 \circ 
		 (\chi P^t)_\ast (P^{\ell_1}_\ast(r) \circ P^{\ell_2}_\ast(x)) 
	\\
	&= \sum_{\ell \geq 0} Q^{n+\ell} (P^\ell_\ast(r)\circ x)
		+
		\,\sum_{\mathclap{\substack{
			\ell_1 \geq 0,\,\ell_2 \geq 1 \\
			t_1, t_2 \geq 0}}}\,
		 E^0_{n+\ell_1+\ell_2+t_1+t_2}
		  \circ 
		 (\chi P^{t_1})_\ast P^{\ell_1}_\ast(r) 
		 \circ 
		 (\chi P^{t_2})_\ast P^{\ell_2}_\ast(x) 
	\\
	&\equiv \sum_{\ell \geq 0} Q^{n+\ell} (P^\ell_\ast(r)\circ x).
\end{aligned}
\end{equation}
By Proposition~\ref{prop:bocksteinbound} and 
properties of Dyer--Lashof operations, 
\[
	Q^{n+\ell} (P^\ell_\ast(r)\circ x) \neq 0 
\]
implies the inequalities
\[
	2p\ell \leq \deg(r) - \deg_\beta(r)
	\qquad\text{and}\qquad
	2(n+\ell)\geq \deg(r) - 2(p-1)\ell + \deg(x),
\]
which we may rewrite as 
\begin{equation}
\label{eq:nzineqpair} 
	\deg(r)+\deg(x) \leq 2n+2p\ell \leq 2n+\deg(r) - \deg_\beta(r).
\end{equation}
Thus the sum on the last line of \eqref{eq:congcomp}
contains no nonzero terms when $2n < \deg(x)+\deg_\beta(r)$.
Moreover, in the case $2n = \deg(x)+\deg_\beta(r)$,
\eqref{eq:nzineqpair} is satisfied precisely when 
\[
	2p\ell = \deg(r) - \deg_\beta(r)
\]
in which case 
\[
	Q^{n+\ell} (P^\ell_\ast(r)\circ x) = (P^\ell_\ast(r)\circ x)^p
\]
by properties of Dyer--Lashof operations.
This concludes the proof in the case $\epsilon = 0$.
Applying $\beta$ to the congruence
derived in the case $\epsilon = 0$,
using Propositions~\ref{prop:alginaopmods} and~\ref{prop:modinaopmods},
and making use of the case $\epsilon = 0$ again,
we obtain the string of congruences
\[
	0
	\equiv
	\beta(E^0_n\circ r \circ x)
	\equiv
	E^1_n\circ r \circ x + E^0_n\circ \beta(r) \circ x
	\equiv
	E^1_n\circ r \circ x,	
\]
proving the claim when $\epsilon =1$.
Notice that $\deg_\beta \beta(r)  = \deg_\beta(r)+1$
by Propositions~\ref{prop:steenrodaction}
and~\ref{prop:alginaopmods}.
\end{proof}

We note the following application
of Proposition~\ref{prop:lowencong}
to the basis elements of Theorem~\ref{thm:eandehatbases}.

\begin{proposition}
\label{prop:ejcong}
Let $X$ is an $E_\infty$-space, let $x \in H_\ast(X)$,
and let
\[
	J = \bigl\langle ((\epsilon_1,i_1),\ldots,(\epsilon_n,i_n))\bigr\rangle
\]
be an allowable sequence.
Then
\[
	E_J\circ x 
	\equiv
    \begin{cases}
     (-1)^{\min(J)} (E_{J'}\circ x)^p 
     &
	 \text{if $2\min(J) = \deg(x) + \deg_\beta(J)$ and $b(J) = 0$}
     \\
     0 
     & 
     \text{if $2\min(J) < \deg(x) + \deg_\beta(J)$}
    \end{cases}
\]
when $p$ is odd and
\[
	E_J\circ x 
	\equiv
    \begin{cases}
     (E_{J'}\circ x)^2 
     &
	 \text{if $\min(J) = \deg(x)$}
     \\
     0 
     & 
     \text{if $\min(J) < \deg(x)$}
    \end{cases}
\]
when $p=2$, where $J'$ denotes the allowable sequence
$J' = \bigl\langle ((\epsilon_2,i_2),\ldots,(\epsilon_n,i_n))\bigr\rangle$
and the congruences are
modulo elements of the form $s\circ a_\ast x$ with $s\in \calE$
and $a\in \calA_p^\op$, $\deg(a)>0$.
\end{proposition}
\begin{proof}
In view of 
Proposition~\ref{prop:lowencong},
it is enough to show that
\[
	P^\ell_\ast (
		E^{\epsilon_2}_{p i_2-\epsilon_2} 
		\circ \cdots \circ 
		E^{\epsilon_n}_{p^{n-1}i_n - \epsilon_n (p^{n-1}-1)/(p-1)}
	)
	= E_{J'}
\]
for $\ell = \sum_{a=2}^n p^{a-2}((p-1)i_a - \epsilon_a)$
when $2\min(J) = \deg(x) + \deg_\beta(J)$ and $b(J) = 0$
[when $\min(J) = \deg(x)$].
By Propositions~\ref{prop:alginaopmods} and \ref{prop:steenrodaction},
we have 
\begin{align*}
 	&P^\ell_\ast (
		E^{\epsilon_2}_{p i_2-\epsilon_2} 
		\circ \cdots \circ 
		E^{\epsilon_n}_{p^{n-1}i_n - \epsilon_n (p^{n-1}-1)/(p-1)}
	)
	\\
	&=
	\sum_{\mathclap{\ell_2+\cdots+\ell_n = \ell}}\,
	P^{\ell_2}_\ast(E^{\epsilon_2}_{p i_2-\epsilon_2})
		\circ \cdots \circ 
	P^{\ell_n}_\ast (E^{\epsilon_n}_{p^{n-1}i_n - \epsilon_n (p^{n-1}-1)/(p-1)})
	)
	\\
	&=
	\sum_{\mathclap{\ell_2+\cdots+\ell_n = \ell}}\quad\,
	 {(p-1)(pi_2 -\epsilon_2-\ell_2)-\epsilon_2 \choose \ell_2}
	 \cdots
	 {(p-1)(p^{n-1}i_n -\epsilon_n\frac{p^{n-1}-1}{p-1}-\ell_n)-\epsilon_n \choose \ell_n}
	 \\
	 &\qquad\qquad\qquad
	 \cdot E^{\epsilon_2}_{p i_2 - \epsilon_2 -\ell_2}
	 \circ
	 \cdots
	 \circ
	 E^{\epsilon_n}_{p^{n-1}i_n - \epsilon_n (p^{n-1}-1)/(p-1) -\ell_n}.
\end{align*}
For the binomial coefficients to be nonzero, we must have
\[
	\ell_a 
	\leq 
	(p-1)(p^{a-1}i_a - \epsilon_a\tfrac{p^{a-1}-1}{p-1} - \ell_a) -\epsilon_a
\]
or, what is the same,
\begin{equation}
\label{eq:ineqella} 
	\ell_a 
	\leq 
	p^{a-2}( (p-1) i_a - \epsilon_a )
\end{equation}
for all $2\leq a \leq n$, 
and the condition $\ell_2+\cdot + \ell_n = \ell$
then implies that equality must hold in \eqref{eq:ineqella}
for all $a$.
Thus 
\[
	P^\ell_\ast (
		E^{\epsilon_2}_{p i_2-\epsilon_2} 
		\circ \cdots \circ 
		E^{\epsilon_n}_{p^{n-1}i_n - \epsilon_n (p^{n-1}-1)/(p-1)}
	)
	= 
	E^{\epsilon_2}_{i_2} 
	\circ \cdots \circ
	E^{\epsilon_n}_{p^{n-2} i_n - \epsilon_n(p^{n-2}-1)/(p-1)}
\]
as desired.
\end{proof}

Together with Theorem~\ref{thm:ehatdescpodd},
the following theorem gives a full description of the
ring $\calE$ in terms of generators and relations
for $p$ odd. We remind the reader 
that in the case $p=2$, the ring $\calE$
was already described in Theorem~\ref{thm:edescp2}.

\begin{theorem}
\label{thm:edescpodd2}
For $p$ odd, the ring $\calE$ is the quotient of the ring $\hat\calE$
obtained by imposing the relation
\begin{equation}
\label{eq:fromehattoerel}
	E^\epsilon_n \circ r = 0 
\end{equation}
for all $r\in \hat\calE$, $\epsilon \in \{0,1\}$ and 
$\epsilon \leq n < \frac{1}{2}(\epsilon + \deg_\beta(r))$.
\end{theorem}

\begin{proof}
Let $C(\pt)$ and
$x_0$ be as in the statement of 
Proposition~\ref{prop:for1}. 
Then by Proposition~\ref{prop:for1}, the map
\[
	\calE \longto H_\ast(C(\pt)), \quad r \longmapsto r\circ x_0.
\]
is an embedding.
By Proposition~\ref{prop:lowencong}, $E^\epsilon_n \circ r \circ x_0 = 0$
in $H_\ast(C(\pt))$ and hence $E^\epsilon_n \circ r = 0$ in $\calE$ 
for all $r\in \calE$, $\epsilon \in \{0,1\}$ and 
$\epsilon \leq n < \frac{1}{2}(\epsilon + \deg_\beta(r))$,
showing that the claimed relations hold in $\calE$.
To see that the relations are sufficient to reduce $\hat\calE$
into $\calE$, it is enough to observe that the relations of 
Corollary~\ref{cor:edesc} are special cases of the 
relations considered here.
\end{proof}

\section{The homology of free \texorpdfstring{$E_\infty$}{E-infinity}-spaces}
\label{sec:freespaces}

We continue to assume that $\calC$ is an $E_\infty$-operad
and that $C$ is the associated monad. 
One of the main features of the Dyer--Lashof operations is
that it is possible to use them to give an explicit 
description of the homology of the ``free $E_\infty$-space''
\[
	C Z  = \bigsqcup_{n\geq0} \calC(n) \times_{\Sigma_n} Z^n
\]
generated by a space $Z$ as a functor of $H_\ast(Z)$.
Our goal in this section is to give such a description 
$H_\ast(CZ)$ in terms of the $E$-operations $E^\epsilon_n\circ$.
In fact, we will work in slightly greater generality 
and describe the homology of the 
free $E_\infty$-space $\tilde{C}(Z,z_0)$ generated
by a pointed space $(Z,z_0)$. The free $E_\infty$-space
generated by an unpointed space $Z$ is recovered as 
$CZ = \tilde{C}(Z_+)$.

Let us describe the space $\tilde{C} (Z,z_0)$.
Writing $\calC[\spaces]$ for the category of $\calC$-algebras 
in $\spaces$, the space $\tilde{C}(Z,z_0)$ can be understood formally
by observing that the adjunction
\[\xymatrix@C+1em{
	\spaces 
	\ar@<1ex>[r]^{C}
	&	
	\calC[\spaces]
	\ar@<1ex>[l]^{\mathrm{forget}}
}\]
factors through the category $\pointedSpaces$ of pointed spaces
as a composite adjunction
\begin{equation}
\label{eq:compositeadj}
\xymatrix@C+1em{
	\spaces 
	\ar@<1ex>[r]^{(-)_+}
	&
	\ar@<1ex>[l]^{\mathrm{forget}}
	\pointedSpaces
	\ar@<1ex>[r]^{\tilde{C}}
	&	
	\ar@<1ex>[l]^{\mathrm{forget}}
	\calC[\spaces].
}
\end{equation}
(We remind the reader that the basepoint of a $\calC$-algebra $X$ 
is the image of the map
$\pt = \calC(0)\times X^0 \to X$
given by the action of $\calC$ on $X$.)
 Concretely, the space $\tilde{C}(Z,z_0)$ is
the quotient of $CZ$ by
the identifications $[c,s_i z] \sim [\sigma_i(c), z]$
for $c \in \calC(n)$, $z\in Z^{n-1}$ and $1 \leq i \leq n$,
where $s_i \colon Z^{n-1} \to Z^n$ the map inserting $z_0$ 
as the $i$-th coordinate
and $\sigma_i \colon \calC(n) \to \calC(n-1)$
is given by
\[
	\sigma_i(c) = \gamma(c; (1_\calC)^{i-1},\ast,(1_\calC)^{n-i})
\]
where $1_\calC \in \calC(1)$ denotes the identity element of $\calC$,
$\ast$ is the unique element in $\calC(0)$,
and $\gamma$ denotes composition in $\calC$.
See e.g.\ \cite[Section~4 and p.\ 274]{MayWhatPrecisely}.

We will offer two descriptions of $H_\ast(\tilde{C}Z)$,
given in 
Theorems~\ref{thm:hczdesc1}
and \ref{thm:hczdesc2} below.
The first, more invariant, takes into account the 
$\calE$- and $\calA_p^\op$-bialgebra structures on 
$H_\ast(\tilde{C}Z)$,
while the second, simpler,
only considers the ring structure and 
describes $H_\ast(\tilde{C}Z)$
as a free graded commutative algebra on certain
explicitly given generators.

To prepare for the first description, let us
start by equipping the free graded commutative ring
$F_{\F_p}(\calE\tensor H_\ast(Z))$ by 
$\calE$- and $\calA_p^\op$-bialgebra structures.
The coalgebra structures on $\calE$ and $H_\ast(Z)$
induce a coalgebra structure on $\calE\tensor H_\ast(Z)$,
and the bialgebra structure on 
$F_{\F_p}(\calE\tensor H_\ast(Z))$ 
is obtained by extending the coproduct and counit
from $\calE\tensor H_\ast(Z)$
onto all of $F_{\F_p}(\calE\tensor H_\ast(Z))$ 
by multiplicativity. To describe the $\calE$ and 
$\calA_p^\op$-actions on 
$F_{\F_p}(\calE\tensor H_\ast(Z))$,
notice first that $\calE\tensor H_\ast(Z)$
is in an evident way a $\calE$-module,
and that the $\calA_p^\op$-module structures on $\calE$ and 
$H_\ast(Z)$ induce on $\calE\tensor H_\ast(Z)$ 
the structure of a $\calA_p^\op$-module via the Cartan formula.
It follows formally that 
$F_{\F_p}(\calE\tensor H_\ast(Z)) 
= 
\bigoplus_{k\geq0} ( (\calE\tensor H_\ast(Z))^{\tensor k})_{\Sigma_k}$
is an $\calE$- and $\calA_p^\op$-algebra, with 
the $\calE$- and $\calA_p^\op$-actions
obtained by extending the ones on $\calE\tensor H_\ast(Z)$
via the Cartan formula. Observing first that the 
the coproduct and counit on $\calE\tensor H_\ast(Z)$
are $\calE$- and $\calA_p^\op$-linear, one can verify
that the bialgebra structure maps on 
$F_{\F_p}(\calE\tensor H_\ast(Z))$
also are. We are now ready to state our
first description of $H_\ast(\tilde{C}Z)$.

\begin{theorem}
\label{thm:hczdesc1}
Let $Z$ be a pointed space with basepoint $z_0 \in Z$,
and write $i \colon Z \to \tilde{C}Z$ for the unit of the 
rightmost adjunction in equation \eqref{eq:compositeadj}.
Then the ring homomorphism
\begin{equation}
\label{eq:ehczepi}
	F_{\F_p} (\calE\tensor H_\ast(Z)) \longto H_\ast(\tilde{C} Z)
\end{equation}
defined by the assignment $r\tensor z \mapsto r\circ i_\ast(z)$
for $r\tensor z \in \calE \tensor H_\ast(Z)$
is a morphism of $\calE$-bialgebras and $\calA^\op$-bialgebras.
Moreover, it is onto, and its
kernel is the ring ideal generated by the relations
\begin{equation}
\label{eq:erelunit}
	r \tensor [z_0] = \varepsilon(r) 1
 \end{equation}
for  $r \in \calE$,
\begin{equation}
\label{eq:erel0}
	(-1)^n E^0_n \circ x 
	= 
    \begin{cases}
    x^p - \sum_{k\geq 1} (-1)^{n+k} E^0_{n+k} \circ (\chi P^k)_\ast(x) 
    &
    \text{if $\deg(x) = 2n$ }
	\\
    -\sum_{k\geq 1} (-1)^{n+k} E^0_{n+k} \circ (\chi P^k)_\ast(x) 
    &
    \text{if $\deg(x) > 2n$ }
    \end{cases}
\end{equation}
for $n \geq 0$ and $x \in \calE \tensor H_\ast(Z)$, and
\begin{equation}
\label{eq:erel1}
	(-1)^n E^1_n \circ x 
	= 
    -\sum_{k\geq 1} (-1)^{n+k} E^1_{n+k} \circ (\chi P^k)_\ast(x) 
    -\sum_{k\geq 0} (-1)^{n+k} E^0_{n+k} \circ \beta(\chi P^k)_\ast(x) 
    \quad
\end{equation}
for $n\geq 1$ and $x \in \calE \tensor H_\ast(Z)$ satisfying $\deg(x) \geq 2n$.
Here
$[z_0] \in H_0(Z)$ is the class represented by $z_0\in Z$.
[For $p=2$, omit \eqref{eq:erel1} 
and replace $2n$ by $n$ on the right hand side of \eqref{eq:erel0}.]
\qed
\end{theorem}

In the proof of Theorem~\ref{thm:hczdesc1},
we will need the following lemma.

\begin{lemma}
\label{lm:rgens}
Let $Z$ be a pointed space with basepoint $z_0 \in Z$,
and let $Q$ be the ring quotient of the free graded commutative ring
$F_{\F_p} (\calE\tensor H_\ast(Z))$
by the relations of Theorem~\ref{thm:hczdesc1}.
Let $\{z_\alpha\}_{\alpha\in \Lambda}$ be a basis for $H_\ast(Z)$
such that 
$z_{\alpha_0} = [z_0]$
for some $\alpha_0 \in \Lambda$.
Then $Q$ is generated as a ring 
by the classes represented by the elements
\[
	E_J \tensor z_\alpha \in \calE\tensor H_\ast(Z)
\]
where $\alpha\in \Lambda\setminus\{\alpha_0\}$ 
and $J$ is an allowable sequence satisfying 
\begin{align*}
	\deg_\beta(J) - b(J) + \deg(z_\alpha) &< 2\min(J)
	\\
	[\deg(z_\alpha) &< \min(J) \text{ when $p=2$}].
\end{align*}
\end{lemma}

\begin{proof}
Let $S\subset Q$ be the subring generated by the claimed 
generators for $Q$. Our task is to show that $S=Q$.
Of course, $Q$ is generated by the classes $[r\tensor z]\in Q$
where $r\in \calE$ and $z \in H_\ast(Z)$, so it is enough 
to show that each such class belongs to $S$.
By induction on the degree of $z$, we may assume that
$[r'\tensor z'] \in S$ for all $r'\in \calE$ and $z'\in H_\ast(Z)$
such that $\deg(z') < \deg(z)$. Write  $S_1 \subset S$
for the subring generated by such $[r'\tensor z']\in Q$.
By a further induction on the length of $r$, we may 
further assume that 
$[r'\tensor z'] \in S$ for all $r'\in \calE$ and $z'\in H_\ast(Z)$
such that $\deg(z') = \deg(z)$ and $\len(r')<\len(r)$.
Write  $S_2 \subset S$
for the subring generated by $S_1$ and such $[r'\tensor z']\in Q$.

Writing $r$ in terms of the basis of Theorem~\ref{thm:eandehatbases}
and $z$ in terms of the basis $\{z_\alpha\}_{\alpha\in \Lambda}$,
we see that it is enough to consider 
the case where $r=E_J$ and $z = z_\alpha$
for some allowable sequence $J$ and $\alpha\in\Lambda$.
If $\alpha = \alpha_0$, relation \eqref{eq:erelunit}
implies that  $[E_J \tensor z_\alpha] = \varepsilon(E_J)1_S \in S$.
Suppose $\alpha \neq \alpha_0$.
If 
\begin{align*}
	\deg_\beta(J) - b(J) + \deg(z_\alpha) &< 2\min(J)
	\\
	[\deg(z_\alpha) &< \min(J) \text{ when $p=2$}],
\end{align*}
$[E_J \tensor z_\alpha] \in S$ by the definition of $S$.
Assume 
\begin{align*}
	\deg_\beta(J) - b(J) + \deg(z_\alpha) &\geq 2\min(J)
	\\
	[\deg(z_\alpha) &\geq \min(J) \text{ when $p=2$}].
\end{align*}
Writing $n = \min(J)$ and $\epsilon = b(J)$, 
we then have 
$E_J\tensor z_\alpha = E^\epsilon_n \circ (r_1\tensor z_\alpha)$
where $r_1 \in \calE$, $\len(r_1) = \len(J)-1$, and
$\epsilon \leq n\leq \frac{1}{2}(\deg_\beta(r_1) + \deg(z_\alpha))$ 
[$0\leq n \leq \deg(z_\alpha)$ when  $p=2$].
Thus to complete the proof, it is sufficient to show that 
$[E^\epsilon_n \circ (r_1\tensor z_\alpha)] \in S_2$
for all such $\epsilon$, $n$, and $r_1$.

Let us consider first the case $\epsilon = 0$. 
For $x\in \calE\tensor H_\ast(Z)$ and $m\geq 0$, write 
\[
	\hat{Q}^m(x) 
	= 
	\sum_{k\geq 0} (-1)^{m+k} E^0_{m+k} \circ (\chi P^k)_\ast(x) 
	\in \calE\tensor H_\ast(Z),
\]
and observe that relation~\eqref{eq:erel0} implies that
\begin{equation}
\label{eq:qhatprop}
	[\hat{Q}^m (x)] 
	= 
    \begin{cases}
	[x^p] &\text{if $\deg(x)=2m$} 
	\\
	0 & \text{if $\deg(x) > 2m$}
    \end{cases}
\end{equation}
in $Q$, mimicking the properties of Dyer--Lashof operations.
[When $p=2$, replace $2m$ by $m$ on the right hand side of 
\eqref{eq:qhatprop}.]
By the defining property of the antipode $\chi$,
we have
\begin{align*}
    (-1)^n E^0_n \circ (r_1\tensor z_\alpha) 
    &=
    \sum_{\ell,k\geq 0} (-1)^{n+\ell+k} E^0_{n+\ell+k} \circ (\chi P^k)_\ast P^\ell_\ast  (r_1\tensor z_\alpha) 
    \\
    &=
    \sum_{\ell \geq 0} \hat{Q}^{n+\ell} P^\ell_\ast(r_1 \tensor z_\alpha)
\end{align*}
in $\calE\tensor H_\ast(Z)$,
and the argument in the proof of Proposition~\ref{prop:lowencong}
goes through to show that modulo elements in $S_1$,
$(-1)^n [E_n^0\circ (r_1\tensor z_\alpha)]$ is either $0$ or 
$[P^{\ell_0}_\ast (r_1) \tensor z_\alpha]^p$ for some $\ell_0$.
In either case, we have $[E_n^0\circ (r_1\tensor z_\alpha)] \in S_2$.

It remains to consider the case $\epsilon=1$. Observe that 
the ideal of $F_{\F_p}(\calE\tensor H_\ast(Z))$ 
generated by the relations of Theorem~\ref{thm:hczdesc1}
is closed under the action of the Bockstein $\beta\in\calA_p^\op$,
as for each of the relations in Theorem~\ref{thm:hczdesc1},
applying the Bockstein to it yields either a relation equivalent 
to another such relation or a trivial relation.
It follows that the Bockstein on $F_{\F_p}(\calE\tensor H_\ast(Z))$
induces a Bockstein operation $\beta$ on the quotient $Q$.
Observe also that the subring $S_2$ of $Q$ is closed under the 
action of this Bockstein, as follows by considering 
the action of the Bockstein on generators $[r'\tensor z']$ of $S_2$.
As in the proof of Proposition~\ref{prop:lowencong},
the case $\epsilon=0$ already proven now implies that
\[
	0
	\equiv
	\beta [E^0_n\circ (r_1 \tensor z_\alpha)]
	\equiv
	[E^1_n\circ (r_1 \tensor z_\alpha)] 
	+ 
	[E^0_n\circ (\beta(r_1) \tensor z_\alpha)]
	\equiv
	[E^1_n\circ r_1 \circ z_\alpha]
\]
where the congruences are modulo $S_2$ (considered as an additive 
subgroup of $Q$). The claim follows.
\end{proof}

\begin{proof}[Proof of Theorem~\ref{thm:hczdesc1}]
By Proposition~\ref{prop:ebialgstr}, in particular
equations~\eqref{eq:elincoprod} and \eqref{eq:elincounit},
the map 
\[
	\calE \tensor H_\ast(Z) \longto H_\ast(\tilde{C}Z),
	\quad
	r\tensor z \mapsto r\circ i_\ast(z)
\]
is a map of coalgebras. It follows formally that
\eqref{eq:ehczepi} is a map of bialgebras. 
Evidently the above map is $\calE$-linear,
and by Proposition~\ref{prop:modinaopmods}
it is also $\calA_p^\op$-linear.
$\calE$-linearity and $\calA_p^\op$-linearity of
\eqref{eq:ehczepi} now follow
from the Cartan formulas for the 
$\calE$- and $\calA_p^\op$-actions on $H_\ast(\tilde{C}Z)$.

It remains to show that \eqref{eq:ehczepi} is an epimorphism
with the indicated kernel. 
By \cite[Theorem~I.4.1]{HILS}, the ring $H_\ast(\tilde{C}Z)$
is generated by elements of the form $Q^I i_\ast(z)$,
where $z\in H_\ast(Z)$ and $Q^I$ is an iterated Dyer--Lashof operation.
Using Proposition~\ref{prop:eqformulas}, it follows that 
\eqref{eq:ehczepi} is an epimorphism. 
Moreover, the claimed relations 
hold in $H_\ast(\tilde{C}Z)$:
\eqref{eq:erelunit} holds since $i_\ast[z_0]$ 
is the unit for $\tilde{C}Z$,
and in view of Proposition~\ref{prop:eqformulas},
\eqref{eq:erel0} and \eqref{eq:erel1}
reduce to basic properties of Dyer--Lashof operations.
Thus \eqref{eq:ehczepi} induces an epimorphism
$\phi\colon Q \to H_\ast(\tilde{C}Z)$, where $Q$ is the
ring of Lemma~\ref{lm:rgens}.
Our task is to show that $\phi$ is an isomorphism.

By Lemma~\ref{lm:rgens}, the ring $Q$ receives an epimorphism 
$\theta \colon F_{\F_p}(\calS) \longto Q$
where 
\[
	\calS 
	= 
	\left\{
		E_J \tensor z_\alpha
		\mathrel{\Bigg|} 
		\text{\small\parbox{6.5cm}{ 
			$\alpha\in \Lambda\setminus\{\alpha_0\}$, 
			$J$ an allowable sequence \\ with
			$\deg_\beta(J) - b(J) + \deg(z_\alpha) < 2\min(J)$ \\
			$[$with $\deg(z_\alpha) < \min(J)$ if $p=2]$}
		}
	\right\}.
\]
Thus we have an epimorphism
\[
	\phi \theta \colon F_{\F_p}(\calS) \longto H_\ast(\tilde{C}Z).
\]
To prove the claim, it is enough to show that 
$\phi \theta$ is an isomorphism, which we will do by
a dimension comparison.

By passing to a CW approximation and filtering $Z$ by finite
subcomplexes, we may 
without loss of generality assume that $H_\ast(Z)$ is finite.
By the description of $H_\ast(\tilde{C}Z)$
in terms of Dyer--Lashof operations
\cite[p.~42]{HILS},
$H_\ast(\tilde{C}Z)$ is a free graded commutative ring
where the generators of degree $d$ are in bijection
with the set
\[
	T_d 
	=
	\left\{
		(\alpha,I)
		\mathrel{\Big|} 
		\text{\small\parbox{7.5cm}{ 
    		$\alpha \in \Lambda \setminus \{\alpha_0\}$,
			$I$ an admisible sequence satisfying\\
    		$e(I) + b(I) > \deg(z_\alpha)$
    		and 
    		$\deg(I) + \deg(z_\alpha) = d$
		}}
	\right\}.
\]
See \cite[Definition I.2.1]{HILS}) for 
the definition of an admissible sequence,
the functions $e$ and $b$, and the degree
of an admissible sequence;
notice that the $b$ of \cite{HILS} agrees with  
the $b$ of Definition~\ref{def:seqdefs},
so that there is no clash in notation.

It is straightforward to verify that the maps
\[
	\bigl(\alpha, ((\epsilon_1,i_1),\ldots, (\epsilon_n,i_n))\bigr)
	\mathrel{\substack{\textstyle\longmapsto\\[-0.3ex]
                      \textstyle\longmapsfrom}}
	E_{\left\langle \vphantom{p^\ell}
		((\epsilon_1,k_1),\ldots, (\epsilon_n,k_n))
	\right\rangle} 
	\tensor 
	z_\alpha
\]
for $n \geq 0$ where
\[
	k_s 
	= 
	i_s - (p-1)\sum_{\ell=s+1}^n i_\ell + \sum_{\ell=s+1}^n \epsilon_\ell
\]
and 
\[
	i_s 
	= 
	k_s + \sum_{\ell=s+1}^n (p^{\ell-s} - p^{\ell-s-1}) k_\ell
	-
	\sum_{\ell=s+1}^n p^{\ell-s-1}\epsilon_\ell
\]
for $1\leq s \leq n$
give inverse bijections
between $T_d$
and the set of degree $d$ elements of $\calS$.
Moreover, the assumption that
$H_\ast(Z)$ is finite implies that the latter set is finite.
It follows that the source and target of $\phi\theta$
are of finite type and have the same Poincaré series.
Thus the epimorphism $\phi\theta$ must be an isomorphism, as desired.
\end{proof}

The proof of Theorem~\ref{thm:hczdesc1} also shows
\begin{theorem}
\label{thm:hczdesc2}
Let $Z$ be a pointed space with basepoint $z_0 \in Z$,
and write $i \colon Z \to \tilde{C}Z$ for the unit of the 
rightmost adjunction in equation \eqref{eq:compositeadj}.
Let $\{z_\alpha\}_{\alpha\in \Lambda}$ be a basis for $H_\ast(Z)$
such that 
$z_{\alpha_0} = [z_0]$
for some $\alpha_0 \in \Lambda$.
Then, as rings,
\[
	\pushQED{\qed} 
	H_\ast(\tilde{C} Z)
	= 
	F_{\F_p} \left(E_J \circ i_\ast(z_\alpha) 
		\mathrel{\Bigg|} 
		\text{\parbox{6.7cm}{ 
			$\alpha\in \Lambda\setminus\{\alpha_0\}$, 
			$J$ an allowable sequence  with
			$\deg_\beta(J) - b(J) + \deg(z_\alpha) < 2\min(J)$
			$[$with $\deg(z_\alpha) < \min(J)$ if $p=2]$}
		}
	\right). 
	\qedhere
	\popQED
\]
\end{theorem}

\begin{remark}
The homology of any $E_\infty$-space $X$ can be expressed
as an $\calE$- and $\calA_p^\op$-bialgebra quotient 
of the homology of a free $E_\infty$-space.
To see this, it suffices to notice that the 
map $\theta\colon CX \to X$ providing the $C$-algebra structure on $X$
is a map of $C$-algebras, and that this map has a section 
given by the unit of $C$.
\end{remark}

We conclude the section by pointing out that 
Theorems~\ref{thm:hczdesc1} and  \ref{thm:hczdesc2}
can also be used to provide descriptions
of the homology of the free infinite loop space
\[
	QZ = \colim_n \loops^n\suspension^n Z
\]
generated by a pointed space $Z$. 
The space $QZ$ is an algebra over
certain $E_\infty$-operad $\calC_\infty$
(the infinite little cubes operad),
and there exists a map of $\calC_\infty$-algebras
$\alpha_\infty\colon \tilde{C}_\infty Z \to QZ$ 
where 
$\tilde{C}_\infty \colon \pointedSpaces \to \calC_\infty[\spaces]$ 
is the
functor induced by $\calC_\infty$. See \cite[Theorem~5.2]{MayGILS}.
By \cite[Theorems~I.4.1 and I.4.2]{HILS}
and the diagram on page 41 of \cite{HILS}, the map 
\[
	(\alpha_\infty)_\ast 
	\colon 
	H_\ast(\tilde{C}_\infty Z) \longto H_\ast(QZ)
\]
amounts to localization with respect to the subset 
$\pi_0 (Z) \subset H_0(Z)$, so that $\alpha_\infty$ 
induces an isomorphism
\[
	(\alpha_\infty)_\ast 
	\colon 
	H_\ast(\tilde{C}_\infty Z)[\pi_0(Z)^{-1}] 
	\xto{\ \isom\ }
	H_\ast(QZ).
\]
Thus we have, for example, from Theorem~\ref{thm:hczdesc2}

\begin{theorem}
\label{thm:qzdesc}
Let $Z$ be a pointed space with basepoint $z_0 \in Z$,
let $\{z_\alpha\}_{\alpha\in \Lambda}$ be a basis for $H_\ast(Z)$
containing all elements of $\pi_0(Z)\subset H_0(Z)$,
let $\Lambda_0 \subset \Lambda$ consist 
of the elements $\alpha \in \Lambda$ such that $z_\alpha\in \pi_0(Z)$,
and let $\alpha_0 \in \Lambda_0$ be the element such that 
$z_{\alpha_0} = [z_0]$. Let $i\colon Z \to QZ$ 
be the natural map. Then 
\[
	H_\ast(QZ)
	= 
	F_{\F_p} \bigl(
		E_J \circ i_\ast(z_\alpha) 
		\mid
		(\alpha,J) \in S
	\bigr)
	\tensor
	\F_p \bigl[
		i_\ast(z_{\alpha})^{\pm 1} 
		\mid 
		\alpha\in \Lambda_0\setminus\{\alpha_0\}
	\bigr],
\]
where 
\[
	\pushQED{\qed} 
	S 
	= 
	\left\{
		(\alpha,J) 
		\mathrel{\Bigg|} 
		\text{\parbox{6.7cm}{ 
			$\alpha\in \Lambda\setminus \Lambda_0$, 
			$J$ an allowable sequence  with
			$\deg_\beta(J) - b(J) + \deg(z_\alpha) < 2\min(J)$
			$[$with $\deg(z_\alpha) < \min(J)$ if $p=2]$
		}}
	\right\}.
	\qedhere
	\popQED
\]
\end{theorem}

\section{The coalgebraic perspective}
\label{sec:coalgebraicalgebra}

As we pointed out in Remark~\ref{rk:largerring},
the action of the ring $\calE$ on the homology 
of $E_\infty$-spaces extends to the
action of the larger ring 
$(\bigoplus_{k\geq 0} H_\ast(\Sigma_k),\circ)$.
Our aim in this section is to study the action of this 
larger ring.
To do so, we will borrow concepts from
coalgebraic algebra.
See e.g.\ \cite{RavenelWilson,HuntonTurner}
and the survey \cite{HopfRingsSurvey}.
It turns out that 
$\bigoplus_{k\geq 0} H_\ast(\Sigma_k)$
has the structure of a coalgebraic semiring
(also called a Hopf semiring), 
and that the aforementioned action of 
$\bigoplus_{k\geq 0} H_\ast(\Sigma_k)$
on the homology of an $E_\infty$-space $X$
makes $H_\ast(X)$ into a coalgebraic semimodule
over $\bigoplus_{k\geq 0} H_\ast(\Sigma_k)$.
We start by recalling the notions
of a commutative semiring and a semimodule over such.
Informally, these notions are what one obtains 
from the notions of a commutative ring and a module over such
by dropping the requirement that additive inverses exist.
\begin{definition}
\label{def:sr}
A \emph{commutative semiring} $R = (R,+,0,\cdot,1)$ 
consists of a set $R$ together with binary operations $+$, $\cdot$ on $R$
and distinguished elements $0\in R$ and $1\in R$ such that the following
axioms are satisfied for all $r,s,t\in R$:
\begin{enumerate}[(i)]
\item $(R,+,0)$ and $(R,\cdot,1)$  are commutative monoids
\item\label{it:srdist}
	$r(s+t) = rs+rt$
\item\label{it:srnull}
	 $r0 = 0$.
\end{enumerate}
\end{definition}

\begin{definition}
\label{def:sm}
A \emph{semimodule} over a commutative semiring $R = (R,+,0_R,\cdot,1_R)$ 
consists of a commutative monoid $M = (M,+,0_M)$
together with a map $\cdot \colon R\times M \to M$ such that 
the following axioms are satisfied for all $r,s\in R$ and $x,y\in M$:
\begin{multicols}{3}
\begin{enumerate}[(i)]
\item $(rs)x = r(sx)$
\item $1_R x = x$
\item $r(x+y) = rx + ry$
\item $r 0_M = 0_M$
\item $(r+s)x = rx + sx$
\item $0_R x = 0_M$
\end{enumerate}
\end{multicols}
\end{definition}

It is straightforward to rephrase Definitions~\ref{def:sr} 
and \ref{def:sm} in terms of commutative diagrams.
For example,
axioms~\ref{it:srdist} and \ref{it:srnull} of Definition~\ref{def:sr}
amount to the commutativity of the diagrams
\[
    \vcenter{\xymatrix{
    	R \times R \times R 
    	\ar[r]^-{\id\times +}
    	\ar[d]_{\Delta\times \id \times \id}
    	&
    	R\times R
    	\ar[dr]^{\cdot}
    	\\
    	R \times R \times R \times R
    	\ar[d]_{\id\times \tau\times \id}
    	&&
    	R
    	\\
    	R \times R \times R \times R
    	\ar[r]^-{\cdot\times \cdot}
    	&
    	R\times R
    	\ar[ur]_+	
    }}
    \qquad\text{and}\qquad
	\vcenter{\xymatrix{
    	R \times \pt
    	\ar[r]^-{\id\times 0}
    	\ar[d]
    	&
    	R\times R
    	\ar[d]^{\cdot}
    	\\
    	\pt
    	\ar[r]^{0}
    	&
    	R
    }}    
\]
respectively,
where $\tau$ denotes the coordinate interchange map.
The resulting diagrams can be interpreted in any category with 
finite products, and doing so yields the
definitions of a \emph{commutative semiring object}
and a \emph{semimodule object} in a category
with finite products. (We will usually write 
$\cdot$, $1$, $\circ$  and $[1]$ 
in place of  
$+$, $0$, $\cdot$ and $1$
in the context of semiring and semimodule objects, however.)

The important categories with finite products
for us will be $\ho(\spaces)$ and, crucially,
the category $\coalgebras$ of $\F_p$-coalgebras 
(which, we remind the reader, we assume to be 
graded, graded cocommutative, coassociative, and counital).
The categorical product of objects $A_1, A_2\in \coalgebras$
is given by the tensor product $A_1 \tensor A_2$.
The coproduct and counit on $A_1 \tensor A_2$
are given by the composites
\[
	\psi
	\colon 
	A_1 \tensor A_2 
	\xto{\ \psi\tensor\psi\ }
	A_1 \tensor A_1 \tensor A_2 \tensor A_2
	\xto{\ \id\times \tau\times \id}
	A_1 \tensor A_2 \tensor A_1 \tensor A_2
\]
and
\[
	\varepsilon
	\colon 
	A_1\tensor A_2
	\xto{\ \varepsilon\tensor\varepsilon\ }
	\F_p\tensor \F_p
	\xto{\ \isom\ }
	\F_p,
\]
respectively,
the projections $A_1\tensor A_2 \to A_1$ and $A_1\tensor A_2 \to A_2$ 
are given by the composites
\[
	A_1\tensor A_2 
	\xto{\ \id\tensor\varepsilon\ }
	A_1\tensor \F_p
	\xto{\ \isom\ }
	A_1
	\quad\text{and}\quad
	A_1\tensor A_2 
	\xto{\ \varepsilon\tensor\id\ }
	\F_p \tensor A_2
	\xto{\ \isom\ }
	A_2,
\]
and the map $(f_1,f_2)\colon B \to A_1\tensor A_2$
induced by maps $f_i\colon B\to A_i$, $i=1,2$, in $\coalgebras$
is the composite
\[
	B \xto{\ \psi\ } B\tensor B \xto{\ f_1\tensor f_2\ } A_1 \tensor A_2.
\]
In particular, the diagonal map $\Delta\colon A \to A\times A$
of an object $A\in \coalgebras$ is the coproduct 
$\psi\colon A \to A\tensor A$.
The terminal object in $\coalgebras$ is $\F_p$,
with the counit $\varepsilon\colon A \to \F_p$
giving the unique map from an object  $A\in \coalgebras$ to $\F_p$.

By a \emph{commutative coalgebraic semiring} 
(resp.\  a \emph{coalgebraic semimodule})
we mean a commutative semiring object (resp.\ a semimodule object) in
the category $\coalgebras$. As these notions will be important
for us, we expand out the definitions in detail. 
\begin{definition}
\label{def:coalgring}
A \emph{commutative coalgebraic semiring} 
$R = (R,\psi,\varepsilon,\cdot,1,\circ,[1])$
consists of a coalgebra $(R,\psi,\varepsilon)$
together with homomorphisms $\cdot,\circ\colon R \tensor R \to R$
and $1,[1]\colon \F_p \to R$
of coalgebras such that
\begin{enumerate}[(i)]
\item $(R,\psi,\varepsilon,\cdot,1)$ and $(R,\psi,\varepsilon,\circ,[1])$
	are commutative bialgebras (see Definition~\ref{def:bialgebra})
\item\label{it:coalgringdist}
	$r\circ (st) = \sum (-1)^{\deg(r'')\deg(s)} (r'\circ s) (r''\circ t)$
\item\label{it:coalgringnull}
	$r\circ 1 = \varepsilon(r) 1$
\end{enumerate}
for all $r,s,t\in R$.
\end{definition}

\begin{definition}
\label{def:coalgsemimod}
A \emph{coalgebraic semimodule} over 
a commutative coalgebraic semiring 
$R = (R,\psi,\varepsilon,\cdot,1_R,\circ,[1])$ 
is a commutative bialgebra $M = (M,\psi,\varepsilon,\cdot,1_M)$
together with a homomorphism $\circ \colon R\tensor M \to M$ 
of coalgebras such that 
\begin{enumerate}[(i)]
\item\label{it:coalgsmassoc}
	$(r\circ s)\circ x = r\circ (s\circ x)$
\item \label{it:coalgsmunit}
	$[1]\circ x = x$
\item\label{it:coalgsemimoddist1}
	$r\circ(xy) = \sum (-1)^{\deg(r'')\deg(x)}(r'\circ x)(r''\circ y)$
\item\label{it:coalgsmone1}
	$r \circ 1_M = \varepsilon(r) 1_M$
\item\label{it:coalgsemimoddist2}
	$(rs)\circ x = \sum (-1)^{\deg(s)\deg(x')}(r\circ x')(s\circ x'')$
\item \label{it:coalgsmone2}
	$1_R \circ x = \varepsilon(x) 1_M$
\end{enumerate}
for all $r,s\in R$ and $x,y\in M$.
\end{definition}

There are also evident notions of homomorphisms of 
commutative semirings/semiring objects/coalgebraic semirings 
and semimodules/semimodule objects/coalgebraic semimodules.
The details are safely left to the reader.

\begin{remark}
By axioms~\ref{it:coalgsmassoc} and \ref{it:coalgsmunit} of 
Definition~\ref{def:coalgsemimod},
a coalgebraic semimodule $M$ over a commutative coalgebraic semiring $R$
is in particular a module over the ring $(R,\circ,[1])$.
We will generally view this module structure as the 
part of $M$ that is of primary interest while
the rest of coalgebraic semimodule 
axioms provide useful additional information.
\end{remark}

\begin{notation}
\label{ntn:bracketn}
For a commutative coalgebraic semiring $R$ and $n\geq 0$, we write $[n]$
for the element $[n] = [1]^n \in R$. In particular, $[0] = 1\in R$.
For all $m,n\geq 0$, we have 
$[m][n] = [m+n]$ and $[m]\circ[n] = [mn]$.
If $M$ is a coalgebraic semimodule over $R$, then 
axioms~\ref{it:coalgsemimoddist2}, \ref{it:coalgsmunit}
and \ref{it:coalgsmone2}
of Definition~\ref{def:coalgsemimod} imply that
$[n]\circ x = \sum x_{(1)}\cdots x_{(n)}$
for all $n\geq 0$ and $x\in M$ (where the right hand side
should be interpreted as $\varepsilon(x)1\in M$ when $n=0$).
\end{notation}

\begin{example}
Suppose $\calA$ and $\calB$ are categories with finite products,
and assume $F \colon \calA \to \calB$ is a functor
preserving finite products. Then $F$ takes 
commutative semiring objects in $\calA$ to commutative semiring objects
in $\calB$, and similarly for semimodule objects.
In particular, if $R$ is a commutative semiring object in $\ho(\spaces)$,
its homology $H_\ast(R)$ is a commutative coalgabraic semiring,
and if $M \in \ho(\spaces)$ is a semimodule object over $R$,
the homology $H_\ast(M)$ is a coalgebraic semimodule over $H_\ast(R)$.
\end{example}

\begin{example}
\label{ex:sigmasemiring}
The direct sum
$\bigoplus_{k\geq0} H_\ast(\Sigma_k)$
admits the structure of a commutative coalgebraic semiring.
The coproduct $\psi$ and counit $\varepsilon$ were described in 
Example~\ref{ex:extrastrone},
and the product $\circ$ and the unit $[1]$ were described around 
equation~\eqref{eq:prodsigma}.
The product $\cdot$ is induced by the homomorphisms
\begin{equation}
\label{eq:sumsigma} 
	\cdot
	\colon
	\Sigma_m \times \Sigma_n \longto \Sigma_{m+n},
	\qquad
	m,n\geq 0
\end{equation}
obtained by choosing bijections
\begin{equation}
\label{eq:sumbij}
	\{1,\ldots,m\}\sqcup \{1,\ldots,n\} 
	\xto{\ \isom\ }
	\{1,\ldots,m+n\};
\end{equation}
notice that again,
the homomorphisms \eqref{eq:sumsigma} are, up to conjugacy, 
independent of the choice of the bijections 
\eqref{eq:sumbij}, so in particular
the maps they induce on homology are independent of the choices made.
The unit $1$ is the canonical generator 
of $H_0(\Sigma_0)$. More generally,
the element $[n] \in \bigoplus_{k\geq0} H_\ast(\Sigma_k)$
is the canonical generator of $H_0(\Sigma_n)$.
\end{example}
\begin{example}
\label{ex:bsigmasemiring}
The commutative coalgebraic semiring structure of the previous example
lifts to a commutative semiring object structure on 
the disjoint union $\bigsqcup_{k\geq 0} B\Sigma_k\in \ho(\spaces)$:
the products $\circ$ and $\cdot$ are induced by the homomorphisms 
\eqref{eq:prodsigma} and \eqref{eq:sumsigma}, respectively,
and the units $[1]$ and $1$ 
are given by the inclusions of the subspaces 
$B\Sigma_1\homeom\pt$ and $B\Sigma_0\homeom\pt$
into $\bigsqcup_{k\geq 0} B\Sigma_k$.
The coalgebraic semiring 
$\bigoplus_{k\geq0} H_\ast(\Sigma_k)$
is recovered by applying homology to 
the semiring object
$\bigsqcup_{k\geq 0} B\Sigma_k\in \ho(\spaces)$.
We will give an alternative perspective on this 
commutative semiring object structure
in terms of the canonical homotopy equivalence 
$\bigsqcup_{k\geq 0} B\Sigma_k \homot C(\pt)$
and a commutative semiring object structure on $C(\pt)\in \ho(\spaces)$
in Proposition~\ref{prop:cptsemiring} below.
\end{example}

\begin{example} 
\label{ex:freecoalgsemiring}
Suppose $\calB = (\calB,\psi,\varepsilon,\circ,[1])$ is
a commutative bialgebra. 
Then the free graded commutative ring $F_{\F_p}(\calB)$
is a commutative coalgebraic semiring. 
The product $\cdot$ and the unit $1$
are those of the usual ring structure on $F_{\F_p}(\calB)$, the coproduct
and counit on  $F_{\F_p}(\calB)$ are obtained by extending
the coproduct and counit on $\calB$ by multiplicativity,
the unit $[1]$ on $F_{\F_p}(\calB)$ is given by the unit
$[1]\in \calB \subset F_{\F_p}(\calB)$, and the product $\circ$
on $F_{\F_p}(\calB)$ is determined by the product $\circ$
on $\calB$ via the distributivity and nullity axioms 
\ref{it:coalgringdist} and 
\ref{it:coalgringnull}
of Definition~\ref{def:coalgring}
and the commutativity of $\circ$.
\end{example}

We have the following description of the coalgebraic semiring 
structure on $\bigoplus_{k\geq 0} H_\ast(\Sigma_k)$.
\begin{theorem}
\label{thm:sigmasemiringdesc}
With the coalgebraic semiring structures on the 
source and target given by Examples~\ref{ex:freecoalgsemiring}
and \ref{ex:sigmasemiring}, respectively, 
the ring map
\begin{equation}
\label{eq:fesigmamap}
	F_{\F_p}(\calE) \longto \Bigl(\bigoplus_{k\geq 0} H_\ast(\Sigma_k),\cdot\Bigr)
\end{equation}
induced by the inclusion of $\calE$ into 
$\bigoplus_{k\geq 0} H_\ast(\Sigma_k)$
is an epimorphism of commutative coalgebraic semirings.
Moreover, its kernel is the ring ideal
generated by the relations
\begin{equation}
\label{eq:ejrel}
	E_J 
	=
    (-1)^{\min(J)} (E_{J'})^p 
    \qquad
\end{equation}
for admissible sequences $J$ satisfying
$2\min(J) = \deg_\beta(J)$ and $b(J)=0$
[for admissible sequences $J$ satisfying
 $\min(J) = 0$ when $p=2$],
where 
\[
	J' 
	= 
	\bigl\langle ((\epsilon_2,i_2),\ldots,(\epsilon_n,i_n))\bigr\rangle
	\qquad\text{for}\qquad
	J 
	= 
	\bigl\langle ((\epsilon_1,i_1),\ldots,(\epsilon_n,i_n))\bigr\rangle.
\]
\end{theorem}

The following theorem is the main takeaway from this section.
\begin{theorem}
\label{thm:coalgsemimod}
Let $X$ be an $E_\infty$-space. Then
the $(\bigoplus_{k\geq 0} H_\ast(\Sigma_k),\circ)$-module structure
of Proposition~\ref{prop:circaction}
and the bialgebra structure of Example~\ref{ex:hastxbialg}
make $H_\ast(X)$ into a 
coalgebraic semimodule over the commutative coalgebraic semiring 
$\bigoplus_{k\geq 0} H_\ast(\Sigma_k)$.
\end{theorem}

\begin{remark}
\label{rk:largerring2}
Let $X$ be an $E_\infty$-space.
In view of Theorem~\ref{thm:sigmasemiringdesc}
and the distributivity axiom \ref{it:coalgsemimoddist2}
of Definition~\ref{def:coalgsemimod},
it follows from Theorem~\ref{thm:coalgsemimod} that
the action of the ring $(\bigoplus_{k\geq 0} H_\ast(\Sigma_k),\circ)$
on $H_\ast(X)$ is completely determined by the action
of the subring $(\calE,\circ)$ of 
$(\bigoplus_{k\geq 0} H_\ast(\Sigma_k),\circ)$ on $H_\ast(X)$,
as asserted in Remark~\ref{rk:largerring}.
In view of the relations 
\eqref{eq:ejrel} and the 
distributivity axiom \ref{it:coalgsemimoddist2}
of Definition~\ref{def:coalgsemimod},
the fact that the $\calE$-action on $H_\ast(X)$
extends to a coalgebraic semimodule structure 
over $\bigoplus_{k\geq 0} H_\ast(\Sigma_k)$
places constraints on the $\calE$-action, however.
\end{remark}

We will postpone the proof of 
Theorem~\ref{thm:sigmasemiringdesc}
until after Proposition~\ref{prop:cptsemiring}
and
the proof of  Theorem~\ref{thm:coalgsemimod} until
after Proposition~\ref{prop:modstr}.
Along the way, we will construct  
on $C(\pt)\in \ho(\spaces)$ a commutative semiring object structure
recovering the coalgebraic semiring structure of
Example~\ref{ex:sigmasemiring},
and show that every $E_\infty$-space is a 
semimodule object over $C(\pt)$ in a way that recovers
the module structure of Proposition~\ref{prop:circaction}.
Some of the required work can be done in the generality of an 
arbitrary operad in $\spaces$.

We will from now on 
take the liberty of indexing our operads by finite sets
instead of just natural numbers, 
with $\calO(n)$ working as a shorthand for 
$\calO(\{1,\ldots,n\})$. In this perspective,
an operad $\calO$ consists of a functor
\[
	\calO\colon \finiteSets \longto \spaces
\]
from the category $\finiteSets$ of finite sets and bijections
together with a unit $1_\calO \in \calO(1)$
and composition maps
\[
	\gamma
	\colon
	\calO(A) \times \prod_{a\in A} \calO(B_a) 
	\longto
	\calO\Bigl(\bigsqcup_{a\in A} B_a\Bigr)
\]
natural in the finite sets $A$ and $B_a$, $a\in A$,
such that the evident unitality and associativity axioms are satisfied.
Moreover, an $\calO$-algebra structure on a space $Z$
consists of  maps 
\[
	\theta\colon \calO(A) \times Z^A \longto Z
\]
for $A\in \finiteSets$
such that the diagram
\[\xymatrix@C+2em{
	\calO(A) \times Z^B
	\ar[r]^-{\calO(f) \times \id}
	\ar[d]_{\id\times f^\ast}
	&
	\calO(B) \times Z^B
	\ar[d]^\theta
	\\
	\calO(A) \times Z^A
	\ar[r]^-\theta
	&
	Z
}\]
commutes for all bijections $f\colon A\to B$ of finite sets.
For a space $Z$, the value at $Z$ of the monad $O$ associated to
$\calO$ can now be computed as a coend:
\[
	OZ 
	= 
	\bigsqcup_{n\geq 0} \calO(n) \times_{\Sigma_n} Z^n
	=
	\int^{A\in \finiteSets} \calO(A)\times Z^A.
\]

\begin{theorem}
\label{thm:cmonfun}
Suppose $O \colon \spaces \to \spaces$ is the monad
associated to an operad $\calO$ in $\spaces$.
For $X,Y\in \spaces$, let
\[
	O_\tensor \colon O(X)\times O(Y) \longto O(X\times Y)
\]
be the map, natural in $X$ and $Y$, defined by the maps
\begin{equation}
\label{eq:pretensorconstraint}
\begin{aligned}
	(\calO(A) \times X^A) \times (\calO(B) \times Y^B)
	&\longto
	\calO(A \times B) \times (X\times Y)^{A \times B}
	\\
	( o_A, (x_a)_{a\in A}, o_B, (y_b)_{b\in B} )
	&\longmapsto
	(\gamma(o_A;o_B^A), (x_a,y_b)_{(a,b)\in A\times B})	 
\end{aligned} 
\end{equation}
for finite sets $A$ and $B$, 
where $\gamma$ denotes composition in $\calO$.
Moreover, let 
\[
	O_I \colon \pt \longto O(\pt)	
\]
be the map given by the unit of $O$.
Then the maps $O_\tensor$ and $O_I$ make $O$ into a lax monoidal functor 
$(\spaces,\times) \to (\spaces,\times)$.
Furthermore, if $\calO$ is an $E_\infty$-operad,
the composite
\begin{equation}
\label{eq:ctohoc}
	(\spaces,\times) 
	\xto{\ O\ } 
	(\spaces,\times) 
	\longto 
	(\ho(\spaces),\times)
\end{equation}
of $O$ with the projection from $\spaces$ to $\ho(\spaces)$
is a lax \emph{symmetric} monoidal functor.
\end{theorem}
\begin{proof}
Write $\bar{O}_\tensor$ for the maps of equation
\eqref{eq:pretensorconstraint},
and notice that the map $O_I\colon \pt \to O(\pt)$ 
is induced by the map
\begin{equation}
\label{eq:unitpt}
	\bar{O}_I \colon \pt \longto \calO(\pt) \times \pt^\pt,
	\quad
	\pt \longmapsto (1_\calO,\pt)
\end{equation}
where $1_\calO$ is the unit of the operad $\calO$.
The monoidality of the functor
$O\colon (\spaces,\times) \to (\spaces,\times)$
follows in a straightforward way from the operad axioms of $\calO$.
For example, verifying right unitality of $O$, i.e.\ 
the commutativity (for all $X$) of the square
\[\xymatrix{
	O(X) \times \pt 
	\ar[r]^-\rho_-\isom
	\ar[d]_{1\times O_I}
	& 
	O(X)
	\\
	O(X) \times O(\pt)
	\ar[r]^-{O_\tensor}
	&
	O(X\times \pt)
	\ar[u]_{O(\rho)}^\isom
}\]
where $\rho$ denotes the right unit constraint for $(\spaces,\times)$,
reduces to checking the commutativity
for all finite sets $A$ of the square
\begin{equation}
\label{sq:rightunitality}
\vcenter{\xymatrix{
	(\calO(A) \times X^A) \times \pt
	\ar[r]_-\isom^-\rho
	\ar[d]_{1\times \bar{O}_I}
	&
	\calO(A) \times X^A
	\\
	(\calO(A) \times X^A) \times (\calO(\pt) \times \pt^\pt)
	\ar[r]^-{\bar{O}_\tensor}
	&
	\calO(A\times \pt) \times (X\times \pt)^{A\times \pt}
	\ar[u]_\isom
}}
\end{equation}
where the right hand vertical map is given by the
evident bijection $A\times \pt \xto{\isom} A$ and 
homeomorphism $\rho\colon X\times\pt \xto{\isom} X$;
and the commutativity of \eqref{sq:rightunitality}
follows from one of the unit axioms for the operad $\calO$.

Suppose now $\calO$ is an $E_\infty$-operad.
To verify that the composite \eqref{eq:ctohoc} is symmetric monoidal,
we must verify that the diagram 
\begin{equation}
\label{sq:smdiag}
\vcenter{\xymatrix{
	O(X) \times O(Y)
	\ar[r]^{\tau}
	\ar[d]_{O_\tensor}
	&
	O(Y)\times O(X)
	\ar[d]^{O_\tensor}
	\\
	O(X\times Y)
	\ar[r]^{O(\tau)}
	&
	O(Y\times X)
}}
\end{equation}
commutes up to homotopy for all $X$ and $Y$.
Here $\tau$ denotes the coordinate interchange map.
Suppose $A$ and $B$ are finite sets,
and consider the diagram
\begin{equation}
\label{diag:smdiagpre}
\vcenter{\xymatrix@!0@C=5.5em@R=9ex{
	(\calO(A)\times X^A) \times (\calO(B) \times Y^B)
	\ar[dd]_{\bar{O}_\tensor}
	\ar[rrrr]_\isom^\tau
	&&&&
	(\calO(B)\times Y^B) \times (\calO(A)\times X^A)
	\ar[d]^{\bar{O}_\tensor}
	\\
	&&&& 
	\calO(B\times A) \times (Y\times X)^{B\times A}
	\\
	\calO(A\times B) \times (X\times Y)^{A\times B}
	\ar[rrr]_\isom^{\id\times \tau^{A\times B}}
	&&&
	\calO(A\times B) \times (Y\times X)^{A\times B}
	\ar[ur]_(0.55)\isom^(0.4){\calO(t)\times (t^{-1})^\ast}
}}
\end{equation}
underlying \eqref{sq:smdiag}
where $t \colon A\times B \xto{\ \isom\ } B\times A$ is the coordinate
interchange map. Notice that the slanted arrow commutes
with the maps from its source and target into $O(Y\times X)$.
The composite map from 
the top left hand corner 
into $\calO(B\times A) \times (Y\times X)^{B\times A}$
through the top right hand corner sends
a point  
\[
	(o_A,(x_a)_{a\in A}, o_B, (y_b)_{b\in B})
	\in 
	(\calO(A)\times X^A) \times (\calO(B) \times Y^B)
\]
to the point 
\[
	(\gamma(o_B; o_A^B), (y_b,x_a)_{(b,a)\in B\times A})
	\in
	\calO(B\times A) \times (Y\times X)^{B\times A}
\]
while the composite through the bottom left hand corner
sends the same point to the point
\[
	(\calO(t)\gamma(o_A;o_B^A), (y_b,x_a)_{(b,a)\in B\times A}).
\]
Here the two maps $\calO(A) \times \calO(B) \longto \calO(B\times A)$
given by
\[
	(o_A,o_B) \longmapsto \gamma(o_B; o_A^B)
	\qquad\text{and}\qquad
	(o_A,o_B) \longmapsto \calO(t)\gamma(o_A; o_B^A)	
\]
are both $\phi$-equivariant for the evident homomorphism
$\phi\colon \Sigma_A\times\Sigma_B \to\Sigma_{B\times A}$.
Consequently, by the assumption that $\calO$ is an $E_\infty$-operad, 
the two maps are $\Sigma_A\times \Sigma_B$-equivariantly homotopic.
Thus diagram~\eqref{diag:smdiagpre} commutes
up to a $\Sigma_A\times \Sigma_B$-equivariant homotopy.
A choice for each pair $(m,n)$ of natural numbers of
such a homotopy for the sets $A=\{1,\ldots,m\}$ and $B=\{1,\ldots,n\}$
now defines a homotopy between the two composites
from the top left hand corner to the bottom right hand corner 
in diagram~\eqref{sq:smdiag}.
\end{proof}

\begin{remark}
\label{rk:monmod}
Suppose $\calO$ is an operad in $\spaces$ and $O$ is the 
associated monad. By Theorem~\ref{thm:cmonfun},
the (trivial) monoid structure on the one-point space $\pt$
induces on $O(\pt)$ a monoid structure
which is commutative up to homotopy if $\calO$ is an
$E_\infty$-operad. Explicitly,
the multiplication on $O(\pt)$ is given by the composite
\begin{equation}
\label{eq:cptmonoidstr} 
	\circ
	\colon
	O(\pt) \times O(\pt) 
	\xto{\ O_\tensor\ }
	O(\pt\times \pt)
	\xto{\ \isom\ }
	O(\pt)
\end{equation}
where the latter map is induced by the isomorphism 
$\pt\times \pt \isom \pt$, and the unit for $O(\pt)$
is given by the map
\[
	[1] \colon\pt \xto{\ O_I\, } O(\pt).
\]
It is not difficult to show  (compare with 
the proof of Proposition~\ref{prop:circaction})
that for every $\calO$-algebra $X$,
the composite map
\begin{equation}
\label{eq:cptmodstr} 
\xymatrix{
	\circ
	\colon
	O(\pt) \times X 
	\ar[r]^-{1\times \eta}
	&
	O(\pt)\times O(X)
	\ar[r]^-{O_\tensor}
	&
	O(\pt\times X)
	\ar[r]_-\isom
	&
	O(X)
	\ar[r]^-\theta
	&
	X
}
\end{equation}
makes $X$ into a module over $O(\pt)$.
Here $\eta$ denotes the unit of the monad $O$,
$\theta$ is given by the $O$-algebra structure on $X$,
and the third morphism is induced by the 
canonical isomorphism $\pt \times X \isom X$.
For the argument, it is helpful to observe that 
\eqref{eq:cptmodstr}
is the map induced by the composites
\begin{equation}
\label{eq:modunderlying} 
\xymatrix{
	\calO(A) \times X 
	\ar[r]^-{1\times \Delta}
	&
	\calO(A) \times X^A 
	\ar[r]^-\theta
	&
	X
}
\end{equation}
for finite sets $A$ where $\Delta$ denotes 
the diagonal map and $\theta$ is given by 
the $\calO$-algebra structure on $X$.
\end{remark}

\begin{remark}
\label{rk:modstrscoincide}
Suppose $\calO$ is an operad in $\spaces$
and let $O$ be the associated monad. For a space $Z$,
we now have two $O(\pt)$-module structure on $OZ$ 
in evidence, namely the module structure of Remark~\ref{rk:monmod}
given by the composite~\eqref{eq:cptmodstr} (with $X = OZ$),
and the module structure arising from the 
$\pt$-module structure on $Z$ as the composite
\[
	O(\pt) \times OZ 
	\xto{\ O_\tensor\ }
	O(\pt\times Z)
	\xto{\ \isom\ }
	OZ
\]
where the second map is induced by the isomorphism $\pt\times Z \isom Z$.
Fortunately, it is easy to see that the two module structures agree:
an inspection of the definitions reveals that both are induced by 
the maps
\[
\begin{aligned}
	\calO(A) \times \calO(B) \times X^B 
	&\longto 
	\calO(A\times B) \times X^{A\times B}
	\\
	(o_A,o_B,(x_b)_{b\in B})
	&\longmapsto
	(\gamma(o_A,o_B^A),(x_b)_{(a,b)\in A\times B})
\end{aligned}
\]
for finite sets $A$ and $B$.
In particular, the $O(\pt)$-module structure on $O(\pt)$
obtained by taking $X = O(\pt)$ in \eqref{eq:cptmodstr}
agrees with the one given by the monoid structure of $O(\pt)$.
\end{remark}

We now return to working with our $E_\infty$-operad $\calC$ and 
its associated monad $C$.
From  Remark~\ref{rk:monmod},
we have a homotopy commutative multiplication $\circ$ 
on $C(\pt)$
with unit $[1]$. Moreover,
the $E_\infty$-structure on $C(\pt)$ gives us another
homotopy commutative product $\cdot$ on $C(\pt)$,
well defined up to homotopy,
with unit $1$.

\begin{proposition}
\label{prop:cptsemiring}
The products and units $\cdot$, $1$, $\circ$ and $[1]$
make $C(\pt)$ into a commutative semiring object in $\ho(\spaces)$
isomorphic via the homotopy equivalence 
$C(\pt) \homot \bigsqcup_{k\geq 0} B\Sigma_k$
to the one of Example~\ref{ex:bsigmasemiring}.
In particular, under the identification
$\bigoplus_{k\geq 0} H_\ast(\Sigma_k) = H_\ast(C(\pt))$,
the commutative coalgebraic semiring structure on 
$\bigoplus_{k\geq 0} H_\ast(\Sigma_k)$ is induced by 
the commutative semiring object structure  on $C(\pt)$. 
\qed
\end{proposition}
\begin{proof}
It suffices to show that under the homotopy equivalence
$C(\pt) \homot \bigsqcup_{k\geq 0} B\Sigma_k$,
each of the products and units on $C(\pt)$
corresponds to the one with the same name
in Example~\ref{ex:bsigmasemiring}. 
To see this for $\circ$, observe that $\circ$ on $C(\pt)$
is induced by the maps
\[
	\calC(A)\times \calC(B) \longto \calC(A\times B),
	\quad
	(c_A,c_B) \longmapsto \gamma(c_A,c_B^A)
\]
for finite sets $A$ and $B$,
and that these maps are $\phi$-equivariant for the 
evident homomorphisms 
$\phi\colon \Sigma_A\times\Sigma_B \to \Sigma_{A\times B}$;
and to see the same for $\cdot$, observe that $\cdot$ on $C(\pt)$
is induced by the maps
\[
	\calC(A)\times \calC(B) \longto \calC(A\sqcup B),
	\quad
	(c_A,c_B) \longmapsto \gamma(c_2;c_A,c_B)
\]
for finite sets $A$ and $B$ (with $c_2\in \calC(2)$ a fixed element),
and that these maps are $\phi$-equivariant for the 
evident maps $\Sigma_A\times \Sigma_B \to \Sigma_{A\sqcup B}$.
\end{proof}

\begin{proof}[Proof of Theorem~\ref{thm:sigmasemiringdesc}]
That \eqref{eq:fesigmamap} is a morphism of 
commutative coalgebraic semirings 
follows formally from the fact that the
inclusion of $\calE$ into 
$(\bigoplus_{k\geq 0} H_\ast(\Sigma_k),\circ)$
is a morphism of bialgebras.
It remains to show that \eqref{eq:fesigmamap} is a 
ring epimorphism with the indicated kernel,
or what is the same, that the composite
\begin{equation}
\label{eq:fehcptcomp}
	F_{\F_p}(\calE) 
	\longto 
	\bigl(\textstyle\bigoplus_{k\geq 0} H_\ast(\Sigma_k), \cdot \bigr)
	\xto{\ \isom\ }
	\bigl(H_\ast(C(\pt)),\cdot\bigr) 
\end{equation}
of \eqref{eq:fesigmamap} and the isomorphism induced 
by the homotopy equivalence 
$\bigsqcup_{k\geq 0} B\Sigma_k \homot C(\pt)$ is.
By inspecting the proof of Proposition~\ref{prop:for1},
we see that the latter map in \eqref{eq:fehcptcomp} 
agrees with the isomorphism
\[\xymatrix@C+1em{
	\bigoplus_{k\geq 0} H_\ast(\Sigma_k) 
	\ar[r]^-{\circ x_0}_-\isom
	&
	 H_\ast(C(\pt))
}\]
implied by Proposition~\ref{prop:for1}. Thus the 
composite~\eqref{eq:fehcptcomp} is the ring map
induced by the map 
\[
	\calE \longto  H_\ast(C(\pt)), \quad  r \mapsto r\circ x_0.
\]
The claimed relations now belong to the kernel of \eqref{eq:fehcptcomp}
by Proposition~\ref{prop:ejcong}.
Moreover, Theorems~\ref{thm:eandehatbases} and 
\ref{thm:hczdesc2} imply that these relations 
suffice to generate the kernel,
and that \eqref{eq:fehcptcomp} is onto.
\end{proof}

\begin{proposition}
\label{prop:semimodstr}
Let $X$ be an $E_\infty$-space.
Then the $C(\pt)$-module structure of Remark~\ref{rk:monmod} 
and the product $\cdot \colon X \times X \to X$
and unit $1\colon \pt \to X$ 
given by the $E_\infty$-structure on $X$
make $X$ into a semimodule object in $\ho(\spaces)$
over $C(\pt)$.
\end{proposition}

\begin{proof}
We will verify the distributivity axiom \ref{it:coalgsemimoddist1}
of Definition~\ref{def:coalgsemimod}
as a model and leave the verification of the other
axioms to the reader.
Our task is to check the commutativity of the diagram
\begin{equation}
\vcenter{\xymatrix{
	C(\pt)\times X\times X
	\ar[r]^-{\id\times \cdot}
	\ar[d]_{\Delta\times \id\times \id}
	&
	C(\pt)\times X
	\ar[dr]^-{\circ}
	\\
	C(\pt)\times C(\pt) \times X \times X
	\ar[d]_{\id\times \tau \times \id}
	&&
	X
	\\
	C(\pt)\times X \times C(\pt) \times X 
	\ar[r]^-{\circ \times \circ}
	&
	X \times X
	\ar[ur]_-{\cdot}
}}
\end{equation}
in $\ho(\spaces)$
where $\tau$ denotes the coordinate interchange map.
To do so, it is enough to verify that the diagram
\begin{equation}
\label{diag:dist1underlying}
\vcenter{\xymatrix{
	\calC(A) \times X \times X
	\ar[r]^-{\id\times \cdot}
	\ar[d]_{\Delta\times \id \times \id}
	&
	\calC(A) \times	X
	\ar[dr]^-{m}
	\\
	\calC(A)\times \calC(A) \times X \times X
	\ar[d]_{\id\times \tau \times \id}
	&&
	X
	\\
	\calC(A)\times  X \times \calC(A) \times X
	\ar[r]^-{m\times m}
	&
	X \times X
	\ar[ur]_-{\cdot}	
}}
\end{equation}
commutes up to $\Sigma_A$-equivariant homotopy 
for all finite sets $A$. Here $m$ denotes the 
composite map \eqref{eq:modunderlying} for $\calO = \calC$.
Let $c_2 \in \calC(2)$ be the element defining the
product $\cdot \colon X \times X \to X$.
In diagram~\eqref{diag:dist1underlying},
the composite map from $\calC(A)\times X \times X$
to $X$ along the top and the bottom of the diagram are then
given by
\[
	(c,x,y) 
	\longmapsto 
	\theta(c,\theta(c_2;x,y)^A) 
	= 
	\theta(\gamma(c;c_2^A),(x,y)^A)
\]
and
\[
	(c,x,y) 
	\longmapsto 
	\theta(c_2;\theta(c,x^A),\theta(c,y^A)) 
	= 
	\theta(\gamma(c_2;c,c), x^A, y^A)
	= 
	\theta(\calC(t)\gamma(c_2;c,c), (x,y)^A)
\]
respectively, where 
$t\colon \{1,2\} \times A \xto{\isom} A\times \{1,2\}$
is the coordinate interchange map.
Here the two maps
\[
	\calC(A) \longto \calC(A\times \{1,2\})
\]
given by
\[
	c \longmapsto \gamma(c;c_2^A)
	\qquad\text{and}\qquad
	c \longmapsto \calC(t)\gamma(c_2;c,c)
\]
are both $\phi$-equivariant for the
evident homomorphism $\Sigma_A \to \Sigma_{A\times \{1,2\}}$.
Since $\calC$ is an $E_\infty$-operad,
it follows that the two maps are $\Sigma_A$-equivariantly homotopic.
The claim follows.
\end{proof}

In view of the description of the $C(\pt)$-module structure
on an $E_\infty$-space $X$ given in 
equation~\eqref{eq:modunderlying},
we have the following proposition.

\begin{proposition}
\label{prop:modstr}
Let $X$ be an $E_\infty$-space. 
Under the identification 
$\bigoplus_{k\geq 0} H_\ast(\Sigma_k) = H_\ast(C(\pt))$,
the $(\bigoplus_{k\geq 0} H_\ast(\Sigma_k),\circ)$-module structure
of Proposition~\ref{prop:circaction} on 
$H_\ast(X)$ is induced by the map
$\circ \colon C(\pt)\times X \to X$ of equation~\eqref{eq:cptmodstr}.
\qed
\end{proposition}

\begin{proof}[Proof of Theorem~\ref{thm:coalgsemimod}]
In view of Propositions~\ref{prop:cptsemiring} and \ref{prop:modstr},
the theorem follows from Proposition~\ref{prop:semimodstr}
by applying homology.
\end{proof}

\begin{remark}
It might be interesting to understand the 
structures induced by the monoid and module structures of 
Remark~\ref{rk:monmod} on homology even 
for non-$E_\infty$-operads $\calO$.
We will not pursue such a study here, however.
\end{remark}

\begin{remark}
Suppose $X$ is a grouplike $E_\infty$-space,
i.e.\ an $E_\infty$-space such that $\pi_0 X$
equipped with the multiplication induced by the 
$E_\infty$-structure is a group. 
(Equivalently, $X$ is an infinite loop space.)
Then it is known that 
the homology $H_\ast(X)$ is a coalgebraic \emph{module}
(not just a coalgebraic semimodule) over the coalgebraic \emph{ring}
(not just a coalgebraic semiring) $H_\ast(QS^0)$.
See e.g.\ \cite[p.\ 299]{HuntonTurner}.
Theorem~\ref{thm:coalgsemimod} is simply the 
expected analogue of this result in the context of 
not-necessarily-grouplike $E_\infty$-spaces.
\end{remark}

\section{The stability and transgression properties}
\label{sec:stabilityandtransgression}

If $X$ is an $E_\infty$-space, the mapping spaces
$X^I$, $PX = \{\alpha\in X^I \mid \alpha(0) = x_0\} \subset X^I$
and $\loops X = \loops (X,x_0)\subset PX$
become $E_\infty$-spaces
under pointwise actions of our $E_\infty$-operad $\calC$.
Here $x_0=1\in X$
is the basepoint given by the $E_\infty$-structure on $X$.
In this section, our aim is to investigate the relationship of the 
$(\bigoplus_{k\geq 0} H_\ast(\Sigma_k),\circ)$-module structures
on $H_\ast(\loops X)$ and $H_\ast(X)$.
We will prove the following two results;
compare with \cite[Theorem I.1.1.(7)]{HILS}.

\begin{theorem}[Stability]
\label{thm:homologysuspension}
Suppose $X$ is an $E_\infty$-space. Then the homology suspension
map
\[
	\sigma\colon \tilde{H}_\ast(\loops X) \longto \tilde{H}_{\ast+1}(X)
\]
is a map of $(\bigoplus_{k\geq 0} H_\ast(\Sigma_k),\circ)$-modules.
\end{theorem}

\begin{theorem}[Transgression property]
\label{thm:transgressionproperty}
Let $X$ be a simply connected $E_\infty$-space,
and consider the Serre spectral sequence of the 
path--loop fibration
\[
	\loops X \longto PX \xto{\ \ev_1\ } X
\]
where $\ev_1$ denotes the evaluation map $\alpha \mapsto \alpha(1)$.
Suppose in the spectral sequence $x\in H_s(X)$ 
is a transgressive element
with transgression $y\in H_{s-1}(\loops X)$.
Then for every $r\in \bigoplus_{k\geq 0} H_\ast(\Sigma_k)$,
the element $r\circ x \in H_\ast(X)$ is transgressive with 
transgression $r\circ y \in H_\ast (\loops X)$.
\end{theorem}

We remind the reader that for a pointed space $Z$,
the \emph{homology suspension} is the map
\[
	\sigma\colon\tilde{H}_\ast(\loops Z) \longto \tilde{H}_{\ast+1}(Z)
\]
given by the composite of the suspension isomorphism
$\tilde{H}_\ast(\loops Z) \isom \tilde{H}_{\ast+1}(\suspension\loops Z)$
and the map induced by the counit
\begin{equation}
\label{eq:suspensionloopscounit}
	\suspension\loops Z \longto Z,  
	\quad 
	\alpha \smashprod t \longmapsto \alpha(t)
\end{equation}
of the $(\suspension,\loops)$-adjunction.
In Theorem~\ref{thm:homologysuspension},
the $\bigoplus_{k\geq 0} H_\ast(\Sigma_k)$-module structure on 
the reduced homology $\tilde{H}_\ast(Y)$ of an $E_\infty$-space
$Y$ is obtained by 
viewing  $\tilde{H}_\ast(Y)$ as the quotient 
\[
	\tilde{H}_\ast(Y) = H_\ast(Y) / H_\ast(y_0)
\]
of $\bigoplus_{k\geq 0} H_\ast(\Sigma_k)$-modules, 
where $y_0\in Y$ is the basepoint given by the 
$E_\infty$-structure on~$Y$.

\begin{proof}[Proof of Theorems~\ref{thm:homologysuspension} and \ref{thm:transgressionproperty}]
As the functor
\[
	(-)_+ \colon (\spaces,\times) \longto (\pointedSpaces,\smashprod)
\]
adjoining a disjoint basepoint 
is symmetric monoidal, it sends the monoid $C(\pt)$
in $(\spaces,\times)$ to a monoid 
$C(\pt)_+$ in $(\pointedSpaces,\smashprod)$
and $C(\pt)$-modules in $(\spaces,\times)$ to $C(\pt)_+$-modules in 
$(\pointedSpaces,\smashprod)$. In particular, an $E_\infty$-space $Y$
gives rise to an $C(\pt)_+$-module $Y_+$ in $\pointedSpaces$.
If $Z$ is a $C(\pt)_+$-module in $\pointedSpaces$,
then applying $\tilde{H}_\ast$, we obtain on $\tilde{H}_\ast(Z)$ 
a module structure over 
$\tilde{H}_\ast(C(\pt)_+) = H_\ast(C(\pt)) = \bigoplus_{k\geq 0} H_\ast(\Sigma_k)$; in the case where $Z = Y_+$
for an $E_\infty$-space $Y$, this recovers the usual 
$\bigoplus_{k\geq 0} H_\ast(\Sigma_k)$-module structure on 
$\tilde{H}_\ast(Y_+) = H_\ast(Y)$.
Noting that for an $E_\infty$-space $Y$, the basepoint $y_0$ of $Y$
is invariant under the $C(\pt)$-module structure on $Y$,
we see that the $C(\pt)_+$-module structure on $Y_+$ also 
induces one on the quotient $Y = Y_+ /\{y_0\}_+$. On reduced homology,
this structure recovers the  
$\bigoplus_{k\geq 0} H_\ast(\Sigma_k)$-module structure
we imposed on $\tilde{H}_\ast(Y) = H_\ast(Y)/H_\ast(y_0)$.

Observe now that if  $\mu \colon C(\pt)_+\smashprod Z \longto Z$
makes $Z \in \pointedSpaces$ into a $C(\pt)_+$-module,
then for any  $A\in \pointedSpaces$, the map
\[
	\mu \smashprod A 
	\colon  
	C(\pt)_+\smashprod Z \smashprod A
	\longto 
	Z\smashprod A
\]
makes  $Z \smashprod A$ into a $C(\pt)_+$-module.
In particular, taking $A=S^1$, we see that 
$\suspension \mu$ makes $\suspension Z = Z\smashprod S^1$
into a $C(\pt)_+$-module.
By construction, the suspension isomorphism 
$\tilde{H}_\ast(Z) \isom \tilde{H}_{\ast+1}(\suspension Z)$ is 
then an isomorphism of 
$\bigoplus_{k\geq 0} H_\ast(\Sigma_k)$-modules.
Theorem~\ref{thm:homologysuspension}
now follows by taking $Z = \loops X$ and noticing that 
the counit
\[
	\suspension \loops X \longmapsto X, 
	\quad 
	\alpha\smashprod t \longmapsto \alpha(t)
\]
is a map of $C(\pt)_+$-modules.

It remains to prove Theorem~\ref{thm:transgressionproperty}.
By \cite[Theorem~6.6]{SSguide}, 
the transgression (which we take to be defined in terms of the 
differentials $E^{\ast,0} \to E^{0,\ast-1}$
in the spectral sequence) is given by the diagram
\begin{equation}
\label{eq:transgressiondiag}
\vcenter{\xymatrix{
	&
	H_\ast(PX,\loops X)
	\ar[r]^-{\partial}
	\ar[d]^{(\ev_1)_\ast}
	&
	H_{\ast-1}(\loops X)
	\\
	H_\ast(X)
	\ar[r]^-{j_\ast}
	&
	H_\ast(X,x_0)
}}
\end{equation}
as the additive relation $\partial (\ev_1)_\ast^{-1}j_\ast$,
where $j\colon X \incl (X,x_0)$ is the inclusion.
Underlying diagram~\eqref{eq:transgressiondiag} is the 
diagram
\begin{equation}
\label{eq:transgressiondiagunderlying}
\vcenter{\xymatrix{
	&
	\mathrm{Cone}((\loops X)_+ \incl (PX)_+)
	\ar[r]^-q
	\ar[d]^{\ev_1}
	&
	\suspension ((\loops X)_+)
	\\	
	X_+
	\ar[r]^-j
	&
	\mathrm{Cone}((x_0)_+ \incl X_+)	
}}
\end{equation}
where $\mathrm{Cone}(f)$ denotes the mapping cone of $f$
and $q$ is the quotient map collapsing the base 
$(PX)_+$ of 
$\mathrm{Cone}((\loops X)_+ \incl (PX)_+)$;
diagram~\eqref{eq:transgressiondiag}
is obtained from diagram~\eqref{eq:transgressiondiagunderlying}
by applying $\tilde{H}_\ast$ and 
using the suspension isomorphism to identify 
$\tilde{H}_\ast (\suspension ((\loops X)_+))$
with 
$H_{\ast-1} (\loops X)$.
The claim will therefore follow 
once we verify that 
the two mapping cones in diagram
\eqref{eq:transgressiondiagunderlying}
carry $C(\pt)_+$-module structures making 
all maps in the diagram into maps of $C(\pt)_+$-modules.

Observe that given a zigzag 
\[\xymatrix{
	Z_1 
	&
	\ar[l]
	Z_0
	\ar[r]
	&
	Z_2
}\]
of $C(\pt)_+$-modules and maps of such,
the pushout $Z_1\cup_{Z_0} Z_2$ in $\pointedSpaces$
has the structure of a 
$C(\pt)_+$-module making the maps 
$Z_1 \to Z_1\cup_{Z_0} Z_2$ and $Z_2 \to Z_1\cup_{Z_0} Z_2$ 
into $C(\pt)_+$-module maps; explicitly, the $C(\pt)_+$-module structure
on $Z_1\cup_{Z_0} Z_2$ is given by the composite
\[
	C(\pt)_+ \smashprod (Z_1\cup_{Z_0} Z_2)
	\isom
	(C(\pt)_+ \smashprod Z_1)\cup_{C(\pt)_+ \smashprod Z_0} (C(\pt)_+ \smashprod Z_2)
	\longto 
	(Z_1\cup_{Z_0} Z_2)
\]
where
the second map is 
induced by the $C(\pt)_+$-module structures on $Z_0$, $Z_1$ and $Z_2$
and the homeomorphism arises from the fact that as a left adjoint,
the functor $C(\pt)_+\smashprod(-)$  commutes with colimits.

As a special case, we deduce that for $f\colon Z_1\to Z_2$
a map of $C(\pt)_+$-modules, the 
mapping cone $\mathrm{Cone}(f) = Z_2 \cup_{Z_1} (Z_1\smashprod (I,1))$
has the structure of a $C(\pt)_+$-module 
such that  the inclusion of $Z_2$ into 
$\mathrm{Cone}(f)$ is a  $C(\pt)_+$-module map.
Thus the two mapping cones
in diagram~\eqref{eq:transgressiondiagunderlying}
are  $C(\pt)_+$-modules, and the map $j$
is a $C(\pt)_+$-module map.
Moreover, as the  $C(\pt)_+$-module structure on 
$\mathrm{Cone}(f)$ is clearly functorial in $f$, 
the map $\ev_1$ in diagram~\eqref{eq:transgressiondiagunderlying}
is a $C(\pt)_+$-module map.

As another special case, we deduce that for $Z$ a 
$C(\pt)_+$-module and $B \subset Z$ an $C(\pt)_+$-invariant
subspace containing the basepoint, the quotient $Z/B = Z\cup_B \pt$ 
has the structure of a $C(\pt)_+$-module making the
quotient map $Z \to Z/B$  a map of   $C(\pt)_+$-modules.
Thus the suspension $\suspension ((\loops X)_+)$
in diagram~\eqref{eq:transgressiondiagunderlying}
has the structure of a $C(\pt)_+$-module
such that the map $q$ is a $C(\pt)_+$-module map.
By inspection, the $C(\pt)_+$-module structure on
$\suspension ((\loops X)_+)$
obtained this way
agrees with the one induced from 
the $C(\pt)_+$-module structure on $(\loops X)_+$.
The claim follows.
\end{proof}

\section{The homology of \texorpdfstring{$E_\infty$}{E-infinity}--ring spaces}
\label{sec:ringspaces}

In this section, we turn our attention to 
$E_\infty$--ring spaces, with a focus  on relating
the structures on the homology of an $E_\infty$--ring space 
$X$ induced by the additive and multiplicative
$E_\infty$-structures on $X$.
The main results are Theorem~\ref{thm:cpttoxringmap}
(together with Corollary~\ref{cor:rxyassoc})
relating the additive 
$(\bigoplus_{k\geq 0} H_\ast(\Sigma_k),\circ)$-module
structure on $H_\ast(X)$ to the 
product on $H_\ast(X)$ induced by the multiplicative
$E_\infty$-structure,
and Theorem~\ref{thm:mixedademrels}
giving ``mixed Adem relations'' relating
the $(\bigoplus_{k\geq 0} H_\ast(\Sigma_k),\circ)$-module
structures on the homology of an $E_\infty$--ring space 
arising from the additive and multiplicative
$E_\infty$-structures.
We also prove a ``mixed Cartan formula,'' 
Theorem~\ref{thm:mixedcartan},
explaining how the 
$E$-operations induced by the multiplicative
$E_\infty$-structure on $X$ interact with the 
product on $H_\ast(X)$ induced by the additive $E_\infty$-structure.
The key result underlying 
Theorems~\ref{thm:cpttoxringmap} and \ref{thm:mixedcartan}
is, via Corollary~\ref{cor:modstrgalg}, Proposition~\ref{prop:laxmongalglift}, motivating in part
our approach to the 
$(\bigoplus_{k\geq 0} H_\ast(\Sigma_k),\circ)$-module
structure on the homology of an $E_\infty$-space
via the $C(\pt)$-module structure of 
equation~\eqref{eq:cptmodstr}.
We begin by recalling the definition of an $E_\infty$--ring space.
For more information on $E_\infty$--ring spaces,
we refer the reader to the sequence of papers
\cite{MayWhatPrecisely,MayBiperm,MayGoodFor}.

\begin{definition}
Let $\calC$ and $\calG$ be operads.
An \emph{action of\/ $\calG$ on\/ $\calC$} consists
of maps 
\begin{equation}
\label{eq:gclambda}
	\lambda \colon
	\calG(A) \times \prod_{a\in A} \calC(B_a)
	\longto
	\calC(\textstyle\prod_{a\in A} B_a)
\end{equation}
for finite sets $A$ and $B_a$, $a\in A$,
satisfying the axioms set out in 
\cite[Definition~4.2]{MayBiperm}.
An \emph{operad pair} $(\calC,\calG)$ consists of 
operads $\calC$ and $\calG$ and an action of $\calG$
on $\calC$. An operad pair  $(\calC,\calG)$
is called an \emph{$E_\infty$-operad pair}
if  $\calC$ and $\calG$ are both $E_\infty$-operads.
\end{definition}

Suppose $(\calC,\calG)$ is an operad pair,
and write $C$ and $G$ for the monads associated to $\calC$ and $\calG$,
respectively.
The axioms required of the action of $\calG$ on $\calC$
imply that the monad $C$ in $\spaces$
lifts in a unique way to a monad $C$ in the category $G[\spaces]$
of $G$-algebras,
so that for a $G$-algebra $X$, the space 
$C(X)$ has the structure of a $G$-algebra,
and the unit $X \to C(X)$ and the multiplication
$CC(X) \to C(X)$ of the monad $C$ are maps of $G$-algebras.
See \cite[Proposition~10.1]{MayBiperm}.

\begin{definition}
Suppose $(\calC,\calG)$ is an operad pair. 
By a \emph{$(\calC,\calG)$-space},
we mean a $C$-algebra in $G[\spaces]$,
and by a map of $(\calC,\calG)$-spaces, we mean 
a map of $C$-algebras in $G[\spaces]$.
\end{definition}

The definition above agrees with May's definition of the 
same notions by \cite[Proposition~10.1]{MayBiperm}.
More explicitly, a $(\calC,\calG)$-space
consists of a space $X$ with 
a $\calC$-algebra structure (denoted by $\theta$) and a 
$\calG$-algebra structure (denoted by $\xi$) such that the diagram
\begin{equation}
\label{diag:xithetadist}
\vcenter{\xymatrix{
	\calG(A) \times \prod_{a\in A} (\calC(B_a) \times X^{B_a})
	\ar[r]^-{1 \times \theta^A}
	\ar[d]_\xi
	&
	\calG(A) \times X^A
	\ar[d]^\xi
	\\
	\calC(\prod_{a\in A} B_a) \times X^{\prod_{a\in A} B_a}
	\ar[r]^-\theta
	&
	X
}}
\end{equation}
encoding distributivity relations
commutes for all finite sets $A$ and $B_{a}$, $a\in A$.
Here the map $\xi$ on the left is given by
\begin{equation}
\label{eq:distxidef}
	\xi\bigl(g,\bigl(c_a,(x_{a,b})_{b\in B_a}\bigr)_{a\in A}\bigr)
	=
	\bigl(
		\lambda\bigl(g, (c_a)_{a\in A}\bigr), 
		\bigl(
			\xi(
				g, 
				(x_{a,b_a})_{a\in A}
			)
		\bigr)_{(b_a)_{a\in A} \in \prod_{a\in A} B_a}
	\bigr).
\end{equation}

\emph{From now on, we assume that $(\calC,\calG)$
is a fixed $E_\infty$-operad pair, and write $C$ and $G$
for the monads associated to $\calC$ and $\calG$, respectively.}

\begin{definition}
By an \emph{$E_\infty$--ring space}, we mean a $(\calC,\calG)$-space
for our $E_\infty$-operad pair $(\calC,\calG)$.
If $X$ is an $E_\infty$--ring space, we call
the $\calC$-algebra structure the \emph{additive}
$E_\infty$-structure and the $\calG$-algebra structure the
\emph{multiplicative} $E_\infty$-structure on $X$.
By a map of $E_\infty$--ring spaces, we mean a map of 
$(\calC,\calG)$-spaces.
\end{definition}

We note parenthetically that $E_\infty$--ring spaces
might more appropriately be called $E_\infty$--\emph{semiring} spaces,
as the definition does not require additive inverses to exist
in any sense. We do not wish to deviate from the 
established terminology, however.

\begin{remark}
In our definition of an $E_\infty$-operad $\calO$ 
(Definition~\ref{def:einftyoperad}), we included the
somewhat nonstandard assumption that for every $n$,
the space $\calO(n)$ is $\Sigma_n$-homotopy equivalent to a 
$\Sigma_n$-CW complex. This assumption is harmless
in the context of spaces with operad pair actions as well:
if $(\calO,\calK)$ is an operad pair
such that $\calO$ and $\calK$ satisfy all the other 
requirements for being $E_\infty$-operads,
then, as in Remark~\ref{rk:equivariantcwproperty},
an application of the product-preserving
CW approximation functor $|\Sing_\bullet(-)|$
to $\calO$ and $\calK$
yields an $E_\infty$-operad pair $(\calO',\calK')$ 
such that any $(\calO,\calK)$-space 
has an induced $(\calO',\calK')$-space structure.
\end{remark}

\begin{example}
As $\pt$ is evidently a $G$-algebra, $C(\pt)$ is an $E_\infty$--ring space.
\end{example}

\begin{notation}
For an $E_\infty$--ring space $X$, the two $E_\infty$-structures on $X$
give rise to two commutative monoid structures
on $X$ considered as an object of $\ho(\spaces)$.
We continue
to write $\cdot$ and $1$ 
for the multiplication and unit of the monoid structure
coming from the additive 
$E_\infty$-structure, and write $\circ$ and $[1]$ for the
multiplication and unit coming from the multiplicative
$E_\infty$-structure. We also use the same notation
for the induced products and units on $H_\ast(X)$.
\end{notation}

\begin{notation}
Continuing to assume that $X$ is an $E_\infty$--ring space,
the two $E_\infty$-structures make $H_\ast(X)$
into a coalgebraic semimodule over 
$\bigoplus_{k\geq 0}H_\ast(\Sigma_k)$ in two different ways.
We continue to write $\circ$ for the 
$(\bigoplus_{k\geq 0}H_\ast(\Sigma_k),\circ)$-module 
structure on $H_\ast(X)$
arising from the additive $E_\infty$-structure on $X$,
and write $\sharpop$ for the 
$(\bigoplus_{k\geq 0}H_\ast(\Sigma_k),\circ)$-module structure 
arising from the multiplicative $E_\infty$-structure on~$X$.
\end{notation}

The first part of 
the following result can be viewed as a sanity check on the 
definition of an $E_\infty$--ring space.
\begin{theorem}
\label{thm:semiringfromringspace}
Suppose $X$ is an $E_\infty$--ring space.
Then $(X,\cdot,1,\circ,[1])$ is a commutative semiring object in 
$\ho(\spaces)$. Consequently, 
$(H_\ast(X),\psi,\varepsilon,\cdot,1,\circ,[1])$ is a 
commutative coalgebraic semiring.
\end{theorem}

\begin{proof}
Only the analogues of the distributivity axiom~\ref{it:srdist}
and nullity axiom~\ref{it:srnull} of Definition~\ref{def:sr}
require further comment. 
To check distributivity, we must check that the diagram
\[\xymatrix{
	X\times X \times X
	\ar[r]^-{\id\times \circ}
	\ar[d]_{\Delta\times \id \times \id}
	&
	X\times X
	\ar[dr]^\circ
	\\
	X\times X\times X\times X
	\ar[d]_{\id \times \tau\times \id}
	&&
	X
	\\
	X\times X\times X\times X
	\ar[r]^-{\circ\times \circ}
	&
	X\times X
	\ar[ur]_\cdot
}\]
commutes up to homotopy, 
where $\tau$ denotes the coordinate interchange map.
Let $g_2\in \calG(2)$ and $c_2\in \calC(2)$
be the elements defining $\circ$ and $\cdot$, respectively,
and let $1_\calC\in \calC(1)$ be the unit of $\calC$.
The required commutativity now follows from the 
following instance of diagram~\eqref{diag:xithetadist}
\[\xymatrix@C+1em{
	\calG(2) \times (\calC(1)\times X) \times (\calC(2) \times X^2)
	\ar[r]^-{\id\times \theta^2}
	\ar[d]_\xi
	&
	\calG(2) \times X^2
	\ar[d]^\xi
	\\
	\calC(2)\times X^2
	\ar[r]^-\theta
	&
	X
}\]
by considering the image of a point $(g_2,1_\calC,x_1,c_2,x_2,x_3)$
under the two different composites from 
$\calG(2) \times (\calC(1)\times X) \times (\calC(2) \times X^2)$
into $X$, and using the fact that $\calC(2)$ is contractible.
To verify the nullity axiom, we must check that the diagram
\[\xymatrix{
	\pt\times X 
	\ar[d]
	\ar[r]^{1\times \id}
	&
	X\times X
	\ar[d]^\circ
	\\
	\pt
	\ar[r]^{1}
	&
	X
}\]
commutes up to homotopy, which follows by observing that 
this diagram is homotopy equivalent to 
the following instance of diagram~\eqref{diag:xithetadist}:
\[\xymatrix@C+1em{
	\calG(2) \times (\calC(0)\times X^0) \times (\calC(1) \times X)
	\ar[r]^-{\id\times \theta^2}
	\ar[d]_\xi
	&
	\calG(2) \times X^2
	\ar[d]^\xi
	\\
	\calC(0)\times X^0
	\ar[r]^-\theta
	&
	X
}\]
The proof is complete.
\end{proof}

\begin{remark}
The commutative semiring object structure on
$C(\pt)$ given by Theorem~\ref{thm:semiringfromringspace}
agrees with the one of Proposition~\ref{prop:cptsemiring}. 
To see this, observe that
the multiplication $\circ$ of 
Theorem~\ref{thm:semiringfromringspace}
on $C(\pt)$
is induced by the maps
\begin{equation}
\label{eq:circcptdef}
	\calC(A) \times \calC(B) 
	\longto 
	\calC(A \times B),
	\qquad
	(c_A,c_B) 
	\longmapsto 
	\lambda(g_2;c_A,c_B),	 
\end{equation}
for finite sets $A$ and $B$
where $\lambda$ is the map of equation \eqref{eq:gclambda}
and $g_2$ is a fixed element of $\calG(2)$,
and that \eqref{eq:circcptdef} is $\phi$-equivariant for the 
evident map $\phi \colon \Sigma_A \times \Sigma_B \to \Sigma_{A\times B}$.\end{remark}

The following proposition, through Corollary~\ref{cor:modstrgalg}, 
is the key result underlying
Theorems~\ref{thm:cpttoxringmap} and \ref{thm:mixedademrels}.

\begin{proposition}
\label{prop:laxmongalglift}
When $X$ and $Y$ are $G$-algebras, the maps
\[
	C_I \colon \pt \longto C(\pt)
	\qquad\text{and}\qquad
	C_\tensor \colon C(X)\times C(Y) \to C(X\times Y)
\]
of Theorem~\ref{thm:cmonfun}
are morphisms of $G$-algebras.
\end{proposition}
\begin{proof}
$C_I$ is a morphism of $G$-algebras since 
it is given by the unit of the monad $C$ and the space
$\pt$ is a $G$-algebra.
To show that $C_\tensor$ is a morphism of $G$-algebras,
we must verify that the square
\[\xymatrix{
	G(CX\times CY)
	\ar[d]_{\xi}
	\ar[r]^-{GC_\tensor}
	&
	GC(X\times Y)
	\ar[d]^{\xi}
	\\
	CX \times CY
	\ar[r]^{C_\tensor}
	&
	C(X\times Y)
}\]
commutes.
For this, it suffices to show that the diagram
\[\xymatrix@!0@C=1.7em@R=1.65ex{
	\calG(A)
	\times 
	\prod_{a\in A}
	(\calC(B_a) \times X^{B_a} \times \calC(B'_a) \times Y^{B'_a})
	\ar[dddddrrrrrrrrrr]^(0.55){\id\times \prod_{a\in A}\bar{C}_\tensor}
	\ar[dddddddddd]
	\\ \\ \\ \\ \\
	&&&&&&&&&&
	\calG(A) 
	\times 
	\prod_{a\in A}
	(\calC(B_a\times B'_a) \times (X\times Y)^{B_a\times B'_a})
	\ar[dddddddddd]^-{\xi}
	\\ \\ \\ \\ \\
	*+{\left[\begin{array}{c}
    	\calG(A) \times \prod_{a\in A} (\calC(B_a) \times X^{B_a})
		\\
    	\times 
		\\
    	\calG(A) \times \prod_{a\in A} (\calC(B'_a) \times Y^{B'_a})
	\end{array}\right]}
	\ar[dddddddddd]_{\xi\times \xi}
	\\ \\ \\ \\ \\
	&&&&&&&&&&
	*+{\left[\begin{array}{c}
    	\calC\bigl(\prod_{a\in A} (B_a\times B'_a)\bigr) 
    	\\
    	\times 
    	\\
    	(X\times Y)^{\prod_{a\in A} (B_a\times B'_a)}
	\end{array}\right]}
	\\ 	\\ \\ \\ \\
	*+{\left[\begin{array}{c}
    	\calC(\prod_{a\in A} B_a) \times X^{\prod_{a\in A} B_a} 
		\\
    	\times 
		\\
    	\calC(\prod_{a\in A} B'_a) \times Y^{\prod_{a\in A} B'_a}
	\end{array}\right]}
	\ar[ddddrrrrrrrr]^-{\bar{C}_\tensor}
	\\ 	\\ \\ \\
	&&&&&&&&
	\calC(\prod_{a\in A} B_a \times \prod_{a\in A} B'_a) 
	\times 
	(X\times Y)^{\prod_{a\in A} B_a \times \prod_{a\in A} B'_a}
	\ar[uuuuuuuuurr]_-\isom
}\]
commutes for all finite sets $A$, $B_a$ and $B'_a$, $a\in A$.
Here $\bar{C}_\tensor$ is as in the proof of Theorem~\ref{thm:cmonfun},
the isomorphism in the lower right hand corner is induced by the 
evident bijection
\[
	\prod_{a\in A} B_a \times \prod_{a\in A} B'_a
	\xto{\ \isom\ } 
	\prod_{a\in A} (B_a\times B'_a),
\]
and the unlabeled vertical morphism on the left is given by
\[
	\bigl(
		g, 
		(
			c_a, (x_{a,b})_{b\in B_a}, 
			c'_a, (y_{a,b'})_{b'\in B'_a} 
		)_{a\in A}
	\bigr)
	\mapsto
	\bigl(
		g, 
		(c_a, (x_{a,b})_{b\in B_a})_{a\in A},
		g, 
		(c'_a, (y_{a,b'})_{b'\in B'_a})_{a\in A}
	\bigr).
\]
The commutativity of the latter diagram follows from
axiom (ii) of \cite[Definition 4.2]{MayBiperm}.
\end{proof}

\begin{corollary}
\label{cor:modstrgalg}
Suppose $X$ is an $E_\infty$--ring space.
Then the map $\circ\colon C(\pt)\times X \to X$ 
of equation~\eqref{eq:cptmodstr}
is a morphism of $G$-algebras.
\end{corollary}

\begin{proof}
By Proposition~\ref{prop:laxmongalglift},
all of the maps in equation~\eqref{eq:cptmodstr} 
are morphisms of $G$-algebras.
\end{proof}

\begin{theorem}
\label{thm:cpttoxringmap}
Suppose $X$ is an $E_\infty$--ring space. Then the composite 
\begin{equation}
\label{eq:phidef} 
\xymatrix@C+0.2em{
	\phi 
	\colon
	C(\pt)
	\ar[r]^-\isom
	&
	C(\pt) \times \pt
	\ar[r]^-{1\times [1]}
	&
	C(\pt)\times X
	\ar[r]^-{\circ}
	&
	X
}
\end{equation}
is a map of $E_\infty$--ring spaces making the triangle
\begin{equation}
\label{eq:circandcirctriangle}
\vcenter{\xymatrix{
	C(\pt) \times X 
	\ar[rr]^{\phi\times \id}
	\ar[dr]_\circ
	&&
	X\times X
	\ar[dl]^{\circ}
	\\
	&
	X	
}}
\end{equation}
commute up to homotopy.
Consequently, the map 
\[
	\phi_\ast
	\colon 
	\bigoplus_{k\geq 0} H_\ast(\Sigma_k) \longto H_\ast(X),
	\quad
	r \longmapsto r\circ [1]
\]
is a homomorphism of commutative coalgebraic semirings,
and
\begin{equation}
\label{eq:circandcirc}
	r \circ x = \phi_\ast(r) \circ x 
\end{equation}
for all $r\in \bigoplus_{k\geq 0} H_\ast(\Sigma_k)$ and $x\in H_\ast(X)$.
\end{theorem}
\begin{proof}
To show that $\phi$ is a morphism of
$E_\infty$--ring spaces, we need to verify that it is 
a morphism of $G$- and $C$-algebras.
By Corollary~\ref{cor:modstrgalg}, all of the morphisms in
the equation~\eqref{eq:phidef} are morphisms of $G$-algebras,
so $\phi$ is a morphism of $G$-algebras.
To verify $\phi$ is a morphism of $C$-algebras,
we must show that  the square
\begin{equation}
\label{eq:calgmapsq}
\vcenter{\xymatrix{
	CC(\pt)
	\ar[r]^-{C\phi}
	\ar[d]_{\theta}
	&
	CX
	\ar[d]^\theta
	\\
	C(\pt)
	\ar[r]^-{\phi}
	&
	X
}}
\end{equation}
commutes. 
Unpacking definitions, we see that $\phi$ is induced by the morphisms
\[
	\bar{\phi} \colon \calC(A) \longto X,
	\qquad
	c \longmapsto \theta(c,[1]^A)
\]
for finite sets $A$,
and checking the commutativity of \eqref{eq:calgmapsq} reduces to 
verifying that the square
\[\xymatrix@C+5em{
	\calC(A)\times \prod_{a\in A} \calC(B_a)
	\ar[r]^-{\id\times  \prod_{a\in A} \bar{\phi}}
	\ar[d]_\gamma
	&
	\calC(A) \times X^A
	\ar[d]^\theta
	\\
	\calC(\bigsqcup_{a\in A} B_a)
	\ar[r]^-{\bar{\phi}}
	&
	X
}\]
commutes for all finite sets $A$ and $B_a$, $a\in A$.
The commutativity of the latter square follows from the 
associativity of the $\calC$-action on $X$.

To check that the triangle~\eqref{eq:circandcirctriangle}
commutes up to homotopy, consider the square
\[\xymatrix{
	C(\pt)\times X
	\ar[dr]_{\id}
	\ar[r]^-{\alpha}
	&
	(C(\pt)\times X) \times (C(\pt)\times X)
	\ar[r]^-{\circ\times \circ}
	\ar[d]_\circ
	&
	X\times X
	\ar[d]^{\circ}
	\\
	& 
	C(\pt) \times X
	\ar[r]^-{\circ}
	&
	X
}\]
where the product $\circ$ in the central column arises
from the $G$-algebra structure on $C(\pt)\times X$ and 
$\alpha$ is the map $(c,x) \mapsto ((c,[1]),([1],x))$
Then the triangle on the left commutes up to homotopy
by properties of the unit, and Corollary~\ref{cor:modstrgalg}
implies that the square on the right commutes up to homotopy.
The commutativity of
\eqref{eq:circandcirctriangle} up to homotopy 
now follows by observing 
that the composite of the maps in the top row is $\phi\times \id$.
\end{proof}

\begin{corollary}
\label{cor:rxyassoc}
Suppose $X$ is an $E_\infty$--ring space. Then
\begin{equation}
\label{eq:rxyassoc}
 	r\circ (x\circ y) = (r\circ x) \circ y
\end{equation}
for all $r\in \bigoplus_{k\geq 0} H_\ast(\Sigma_k)$ and $x,y\in H_\ast(X)$.
\end{corollary}
\begin{proof}
With the notation of Theorem~\ref{thm:cpttoxringmap},
we have
\[
	r\circ (x\circ y) 
	= 
	\phi_\ast(r) \circ (x  \circ y)
	=
	(\phi_\ast(r)\circ x) \circ y
	=
	(r\circ x) \circ y
\]
where the first and last equalities hold by Theorem~\ref{thm:cpttoxringmap}
and the middle one holds by associativity of $\circ$ on $H_\ast(X)$.
\end{proof}

\begin{remark}
\label{rk:rxyassoc}
Corollary~\ref{cor:rxyassoc} explains in particular how the 
$E$-operations on $H_\ast(X)$ induced by the \emph{additive} 
$E_\infty$-structure on $X$
(the first $\circ$'s on the two sides of equation~\eqref{eq:rxyassoc})
interact with the 
product on $H_\ast(X)$ induced by the \emph{multiplicative}
$E_\infty$-structure on $X$ (the second $\circ$'s on the two sides
of the equation).
In the context of Dyer--Lashof operations,
the analogue of~\eqref{eq:rxyassoc} is given by the markedly 
more complicated equations
\begin{align}
	\label{eq:qsxy}
	Q^s(x) \circ y 
	&= 
	\sum_{i\geq 0} Q^{s+i} (x \circ P^i_\ast y) 
	\\
	\intertext{and}
	\label{eq:bqsxy}
	\beta Q^s(x) \circ y 
	&= 
	\sum_{i\geq 0} \beta Q^{s+i} (x \circ P^i_\ast y) 
	-
	\sum_{i\geq 0} (-1)^{\deg(x)} Q^{s+i} (x \circ P^i_\ast\beta y) 
\end{align}
for $x,y \in H_\ast(X)$. See \cite[Proposition~II.1.6]{HILS}.
\end{remark}

Via the identification 
$H_\ast(C(\pt)) = \bigoplus_{k\geq 0} H_\ast(\Sigma_k)$,
the ($\bigoplus_{k\geq 0} H_\ast(\Sigma_k),\circ)$-module structure 
$\sharpop$ on $H_\ast(C(\pt))$ 
yields a product
\begin{equation}
\label{eq:sigmaksharpop}
	\sharpop
	\colon 
	\bigoplus_{k\geq 0} H_\ast(\Sigma_k) 
	\tensor 
	\bigoplus_{k\geq 0} H_\ast(\Sigma_k)
	\longto
	\bigoplus_{k\geq 0} H_\ast(\Sigma_k).
\end{equation}
 making 
$(\bigoplus_{k\geq 0} H_\ast(\Sigma_k),\psi,\varepsilon,\circ,[1])$
into a coalgebraic semimodule over 
$\bigoplus_{k\geq 0} H_\ast(\Sigma_k)$.

\begin{remark}
\label{rk:sigmaksharpopdesc}
Let us work out a more concrete description 
of the product~\eqref{eq:sigmaksharpop}.
By Proposition~\ref{prop:modstr},
the module structure $\sharpop$ on $H_\ast(C(\pt))$
is induced by the map
\begin{equation}
\label{eq:gptmodstroncpt}
	G(\pt)\times C(\pt) \longto C(\pt)
\end{equation}
constructed in equation~\eqref{eq:cptmodstr}. 
Unpacking definitions and taking into account the
description given in
equation \eqref{eq:modunderlying},
we see that \eqref{eq:gptmodstroncpt} is induced by the maps
\[
	\calG(A) \times \calC(B) 
	\xto{\ 1\times \Delta\ }
	\calG(A) \times \calC(B)^A 
	\xto{\ \lambda\ } 
	\calC(B^A)
\]
for finite sets $A$ and $B$.
These are $\phi$-equivariant for the obvious 
homomorphisms $\Sigma_A\times\Sigma_B \to \Sigma_{B^A}$.
Thus \eqref{eq:sigmaksharpop}
is induced by the homomorphisms,
well defined up to conjugacy,
\[
	\Sigma_m \times \Sigma_n \longto \Sigma_{n^m},
	\quad
	m,n\geq 0
\]
given by the evident homomorphisms
\[
	\Sigma_m \times \Sigma_n \longto \Sigma_{\{1,\ldots,n\}^{\{1,\ldots,m\}}}
\]
and a choice of bijections
\[
	\{1,\ldots,n\}^{\{1,\ldots,m\}}
	\isom 
	\{1,\ldots,n^m\}.
\]
\end{remark}

The following result describes how the $\sharpop$-module structure
interacts with $\circ$-module structure on the homology
of $E_\infty$--ring spaces.

\begin{theorem}[Mixed Adem relations]
\label{thm:mixedademrels}
Suppose $X$ is an $E_\infty$--ring space.
Then for all $r,s \in \bigoplus_{k\geq 0} H_\ast(\Sigma_k)$
and $x \in H_\ast(X)$, we have
\[
	r \sharpop (s\circ x) 
	= 
	\sum (-1)^{\deg(s)\deg(r'')} (r'\sharpop s) \circ (r'' \sharpop x)
\]
\end{theorem}

\begin{proof}
Immediate from
Proposition~\ref{prop:modstr},
Corollary~\ref{cor:modstrgalg} and Remark~\ref{rk:genextcartan}.
\end{proof}

The products $r\sharpop s$
can (in principle at least) be evaluated explicitly
for all $r, s\in \bigoplus_{k\geq 0} H_\ast(\Sigma_k)$.
See Remark~\ref{rk:sharopformulas}.

\begin{theorem}[Mixed Cartan formula]
\label{thm:mixedcartan}
Suppose $X$ is an $E_\infty$--ring space, and let $x,y\in H_\ast(X)$.
Then for all $r\in H_\ast(\Sigma_p)$
\[
	r \sharpop (xy) 
	=
	\,\sum_{\mathclap{\psi^p(r \tensor x\tensor y)}}\,
	T_0((r \tensor x\tensor y)_{(0)}) 
	\cdots 
	T_p((r \tensor x\tensor y)_{(p)}) 
\]
where the maps $T_i\colon H_\ast(\Sigma_p)\tensor H_\ast(X) \tensor H_\ast(X) \to H_\ast(X)$ are defined by
\[
	T_i(s\tensor z \tensor w)
	=
	\begin{cases}
	\varepsilon(w)(s \sharpop z) & \text{if $i=0$}
	\\[3pt]
	\bigl[\frac{1}{p}{p\choose i}\bigr]
	\circ 
	s
	\circ    	
	([p-i]\sharpop z)
	\circ 
	([i] \sharpop w),
	&
	\text{if $0<i<p$}
	\\[3pt]
	\varepsilon(z)(s \sharpop w) & \text{if $i=p$}	
	\end{cases}
\]
for $s\in H_\ast(\Sigma_p)$ and $z,w\in H_\ast(X)$.
In particular, when $p=2$,
\begin{equation}
\label{eq:mixedcartanat2}
	E^0_k\sharpop (xy) 
	= 
	\sum_{a+b+c = k}\, \sum_{\psi(x)}\,\sum_{\psi(y)}
	(E^0_a \sharpop x') (E^0_b\circ x''\circ y') (E^0_c \sharpop y'').
 \end{equation}
\end{theorem}
\begin{remark}
The $\sharpop$-products $[p-i]\sharpop z$ and $[i] \sharpop w$
in the definition of $T_i$ for $0<i<p$
could be expanded further by using the identity
$[n] \sharpop x = \sum x_{(1)} \circ \cdots\circ x_{(n)}$
for $x$ a class in the homology of an $E_\infty$--ring space.
See Notation~\ref{ntn:bracketn}.
\end{remark}

\begin{proof}[Proof of Theorem~\ref{thm:mixedcartan}]
One possible approach to the proof would be 
to adapt the proof of the analogous result for
Dyer--Lashof operations \cite[Theorem~II.2.5]{HILS}. Another, which
we will adopt here, is to use the statement of \cite[Theorem~II.2.5]{HILS}
as the starting point and use Proposition~\ref{prop:eqformulasps}.
Write $\tilde{Q}^k$ for the $k$-th Dyer--Lashof operation 
on $H_\ast(X)$ arising from the multiplicative $E_\infty$-structure on $X$,
and define
\[
	\tilde{Q}_i^k(z\tensor w)
	=	
    \begin{cases}
	\varepsilon(w)\tilde{Q}^k(z) & \text{if $i=0$}
	\\[3pt]
	\bigl[\frac{1}{p}{p\choose i}\bigr]
	\circ 
	Q^k\bigl(    	
    	([p-i]\sharpop z)
    	\circ 
    	([i] \sharpop w)
	\bigr)
	&
	\text{if $0<i<p$}
	\\[3pt]
	\varepsilon(z)\tilde{Q}^k(w) & \text{if $i=p$}
    \end{cases}
\]
for $k\geq 0$, $0\leq i \leq p$ and $z,w\in H_\ast(X)$.
For $0\leq i \leq p$, 
write $\tilde{Q}_i(t) = \sum_{k\geq 0} \tilde{Q}_i^k t^k$
and define an operation $\sharpop_i$ by setting
$s \sharpop_i (z\tensor w) = T_i(s \tensor z\tensor w)$.
Using Proposition~\ref{prop:eqformulasps},
it is straightforward to verify that 
\begin{equation}
\label{eq:sharopitildeqi}
	E^0(-t) \sharpop_i (z\tensor w) 
	= 
	\tilde{Q}_i(t)P_\ast(t^{-1})(z\tensor w) 
\end{equation}
for all $z,w\in H_\ast(X)$, $0\leq i \leq p$.
Now
\begin{align*}
	E^0(-t) \sharpop (xy)
	&=
	\tilde{Q}(t) P_\ast(t^{-1})(xy)
	\\
	&=
	\tilde{Q}(t) (P_\ast(t^{-1})x P_\ast(t^{-1}) y)
	\\
	&=
	\sum_{\psi^p( P_\ast(t^{-1})x \tensor P_\ast(t^{-1})y)}
	\prod_{i=0}^p 
	\tilde{Q}_i(t)\bigl(
		\bigl(P_\ast(t^{-1})x \tensor P_\ast(t^{-1})y\bigr)_{(i)}
	\bigr)
	\\
	&=
	\sum_{\psi^p(x\tensor y)} 
	\prod_{i=0}^p 
	\tilde{Q}_i(t) P_\ast(t^{-1})
	\bigl((x \tensor y)_{(i)}\bigr)
	\\
	&=
	\sum_{\psi^p(x\tensor y)}
	\prod_{i=0}^p 
	E^0(-t)\sharpop_i \bigl((x \tensor y)_{(i)}\bigr)
	\\
	&=
	\sum_{\psi^p(E^0(-t)\tensor x\tensor y)}
	\prod_{i=0}^p 
	T_i\bigr( 
		(E^0(-t)\tensor x \tensor y)_{(i)}
	\bigr)
\end{align*}
where the first equality follows from 
Proposition~\ref{prop:eqformulasps};
the second holds by the Cartan formula;
the third follows from \cite[Theorem~II.2.5]{HILS};
the fourth holds by the computation
\begin{align*}
	\psi^p( P_\ast(t^{-1})x \tensor P_\ast(t^{-1})y)
	&=
	P_\ast(t^{-1})\psi^p(x\tensor y)
	\\
	&=
	\,\sum_{\mathclap{\psi^p(x\tensor y)}}\,
	P_\ast(t^{-1})((x \tensor y)_{(0)})
	\tensor \cdots\tensor 
	P_\ast(t^{-1})((x \tensor y)_{(p)});
\end{align*}
the fifth follows from equation~\eqref{eq:sharopitildeqi};
and the last follows from the computation
\[
	\psi^p(E^0(-t)\tensor x\tensor y)
	= 
	\sum_{\psi^p(x\tensor y)}
	E^0(-t) \tensor (x\tensor y)_{(0)}
	\tensor \cdots \tensor
	E^0(-t) \tensor (x\tensor y)_{(p)}
\]
and the definition of $\sharpop_i$.
This proves the claim when $r = E^0_k$ for some $k$.
In the case $r=E^1_k$, the claim follows by 
a lengthy computation from the case already 
proven by using the identity
\[
	E^1(t) \sharpop(xy) 
	= 
	\beta(E^0(t)\sharpop(xy))
	-
	E^0(t)\sharpop(\beta(x)y)
	-
	(-1)^{\deg(x)}	E^0(t)\sharpop(x\beta(y)).	\qedhere
\]
\end{proof}

\begin{remark}
\label{rk:sharopformulas}
Let $X$ be an $E_\infty$--ring space. We summarize the
behaviour of $\sharpop$ with respect to the various products and
units on $H_\ast(X)$ and $\bigoplus_{k\geq 0} H_\ast(\Sigma_k)$.
For all $r,s\in H_\ast(\Sigma_k)$ and $x,y \in H_\ast(X)$, we have
\begin{enumerate}[(i)]
\item\label{it:rcircssharpopx}
	 $(r\circ s)\sharpop x = r\sharpop (s\sharpop x)$
\item\label{it:unitsharpopx}
	 $[1]\sharpop x = x$
\item\label{rsharpopxcircy}
	$r\sharpop (x\circ y) 
	= 
	\sum (-1)^{\deg(r'') \deg(x)}(r'\sharpop x) \circ (r''\sharpop y)$
\item\label{it:rsharpopunit}
	 $r \sharpop [1] = \varepsilon(r)[1]$
\item\label{it:rssharpopx}
	$(rs)\sharpop x 
	= 
	\sum (-1)^{\deg(s)\deg(x')} (r\sharpop x') \circ (s \sharpop x'')$
\item\label{it:onesharpopx}
	$1\sharpop x = \varepsilon(x)[1]$
\item\label{it:rsharpopxy}
	$r\sharpop (xy)$: described in Theorem~\ref{thm:mixedcartan}
	for $r \in H_\ast(\Sigma_p)$
\item\label{it:rsharpopone}
	 $r \sharpop 1 = \varepsilon(r)1$
\end{enumerate}
Here \ref{it:rcircssharpopx} through \ref{it:onesharpopx}
follow from the axioms of a coalgebraic semimodule,
and \ref{it:rsharpopone} follows from Remark~\ref{rk:sigmaksharpopdesc}
when $X = C(\pt)$ and from Theorem~\ref{thm:cpttoxringmap}
for general $X$.
The above formulas in particular 
suffice to reduce the evaluation of
$r\sharpop s$ for arbitrary
$r,s\in \bigoplus_{k\geq 0} H_\ast(\Sigma_k)$
to the case where $r = E^\epsilon_m$ and $s = E^\delta_n$
for some $m$, $n$ and $\epsilon$, $\delta$,
in which case the product $r \sharpop s$ is computed
in Theorem~\ref{thm:esharopeformulas} below.
\end{remark}

\begin{theorem}
\label{thm:esharopeformulas}
The $\sharpop$-products between the generators
$E^\epsilon_n$ of $\bigoplus_{k\geq 0} H_\ast(\Sigma_k)$ 
are determined 
by the identity
\begin{equation}
\label{eq:sharpop00p2}
 	E^0(s)\sharpop E^0(t) = E^0(s+t)E^0(t)
\end{equation}
when $p=2$ and by the identities
\begingroup
\allowdisplaybreaks
\begin{align}
	\label{eq:sharpop00}
	\tilde{E}^0(s) \sharpop \tilde{E}^0(t)
	&=
	\bigl([p^{p-2}-1] \circ \tilde{E}^0(s)\circ \tilde{E}^0(t)\bigr)
	\prod_{c=0}^{p-1}\tilde{E}^0(cs+t)
	\\
	\label{eq:sharpop01}
	\tilde{E}^0(s) \sharpop \tilde{E}^1(t)
	&=
	\left(\text{\small$
		\begin{array}{l}		
		\displaystyle
		\phantom{+}\,
    	\bigl([p^{p-2}-1] \circ \tilde{E}^0(s)\circ \tilde{E}^1(t)\bigr)
    	\prod_{c=0}^{p-1}\tilde{E}^0(cs+t)
    	\\
    	+\,
		\displaystyle
    	\bigl([p^{p-2}-1] \circ \tilde{E}^0(s)\circ \tilde{E}^0(t)\bigr)
    	\sum_{k=0}^{p-1} \prod_{c=0}^{p-1}\tilde{E}^{\delta_{c,k}}(cs+t)
		\end{array}		
	$}\right)
	\\
	\label{eq:sharpop10}
	\tilde{E}^1(s) \sharpop \tilde{E}^0(t)
	&=
	\left(\text{\small$
    	\begin{array}{l}
		\displaystyle
		\phantom{+}\,
    	\bigl([p^{p-2}-1] \circ \tilde{E}^1(s)\circ \tilde{E}^0(t)\bigr)
    	\prod_{c=0}^{p-1}\tilde{E}^0(cs+t)
    	\\
    	+\,
		\displaystyle
    	\bigl([p^{p-2}-1] \circ \tilde{E}^0(s)\circ \tilde{E}^0(t)\bigr)
    	\sum_{k=0}^{p-1} k \prod_{c=0}^{p-1}\tilde{E}^{\delta_{c,k}}(cs+t)
    	\end{array}
	$}\right)
	\\
	\label{eq:sharpop11}
	\tilde{E}^1(s) \sharpop \tilde{E}^1(t)
	&= 
	\left(\text{\small$
		\begin{array}{l}
		\displaystyle
		\phantom{+}\,
    	\bigl([p^{p-2}-1] \circ \tilde{E}^1(s)\circ \tilde{E}^1(t)\bigr)
    	\prod_{c=0}^{p-1}\tilde{E}^0(cs+t)
		\\
		-\,
		\displaystyle
    	\bigl([p^{p-2}-1] \circ \tilde{E}^0(s)\circ \tilde{E}^1(t)\bigr)
    	\sum_{k=0}^{p-1} k \prod_{c=0}^{p-1}\tilde{E}^{\delta_{c,k}}(cs+t)
		\\
		+\,
		\displaystyle
    	\bigl([p^{p-2}-1] \circ \tilde{E}^1(s)\circ \tilde{E}^0(t)\bigr)
    	\sum_{k=0}^{p-1} \prod_{c=0}^{p-1}\tilde{E}^{\delta_{c,k}}(cs+t)
		\\
		-\,
		\displaystyle
    	\bigl([p^{p-2}-1] \circ \tilde{E}^0(s)\circ \tilde{E}^0(t)\bigr)
    	\,\sum_{\mathclap{0\leq \ell < k  \leq p-1}}\,(k-\ell)\prod_{c=0}^{p-1}\tilde{E}^{\delta_{c,k} + \delta_{c,\ell}}(cs+t)
		\end{array}
	$}\right)
\end{align}
\endgroup
when $p$ is odd, where $\delta_{i,j}$ is the Kronecker delta.
\end{theorem}
\begin{remark}
Equation~\eqref{eq:sharpop00p2} is the special case $p=2$ of 
equation~\eqref{eq:sharpop00}.
\end{remark}

\begin{proof}[Proof of Theorem~\ref{thm:esharopeformulas}]
Let $A = B = \{1,\ldots,p\}$.
As discussed in Remark~\ref{rk:sigmaksharpopdesc},
the map 
\[
	\sharpop
	\colon
	H_\ast(\Sigma_p) \tensor H_\ast(\Sigma_p) 
	\longto 
	H_\ast(\Sigma_{p^p})	
\]
is induced by the composite
\begin{equation}
\label{eq:sigmabamap}
	\Sigma_A\times \Sigma_B \longto \Sigma_{B^A}
	\xto{\ \isom\ }
	\Sigma_{p^p}
\end{equation}
where the first map is the evident homomorphism and 
the second is induced by an identification of 
$B^A$ with $\{1,\ldots,p^p\}$.
Recall that $\pi$ denotes the cyclic group of order $p$,
and that we have fixed an inclusion 
$\pi \incl \Sigma_p = \Sigma_A = \Sigma_B$
giving a free and transitive action of $\pi$ on $A$ and $B$.
Let us write $i$ for this inclusion.
We will analyze \eqref{eq:sigmabamap} by 
decomposing $B^A$ into $(\pi\times\pi)$-orbits
under the action given by the composite
\[
	\pi\times \pi 
	\xto{\ i\times i\ } 
	\Sigma_A \times \Sigma_B 
	\longto 
	\Sigma_{B^A}.
\]
In formulas, the $(\pi\times \pi)$-action on $B^A$ is given by
\[
	(g,h) \cdot (b_a)_{a\in A} = \big(h(b_{g^{-1} a})\big)_{a\in A}
\]
for $(g,h)\in \pi \times \pi$ and  $(b_a)_{a\in A}\in B^A$.

Suppose $(b_a)_{a\in A} \in B^A$ belongs to a non-free 
$(\pi\times\pi)$-orbit. Then 
\begin{equation}
\label{eq:stabeq}
	\left(h_0\bigl(b_{g_0^{-1}a }\bigr)\right)_{a\in A} 
	= 
	\bigl(b_a\bigr)_{a\in A}
\end{equation}
for some non-identity element $(g_0,h_0)\in \pi \times \pi$.
If $h_0 = e$, then $g_0 \neq e$ by the assumption that $(g_0,h_0)$
is not the identity element, and if $h_0\neq e$, 
it is necessary to have $g_0\neq e$ for \eqref{eq:stabeq} to hold.
Thus in either case $g_0 \neq e$, so that  $g_0$  generates $\pi$.
The stabilizer of $(b_a)_{a\in A}$ contains the subgroup 
$\langle (g_0,h_0) \rangle \leq \pi\times \pi$, which by the 
above analysis is equal to the graph
\[
	\Gamma(\phi) 
	= 
	\{(g,\phi(g)) \mid g\in \pi\} \leq \pi\times \pi
\]
of the homomorphism $\phi\colon\pi \to \pi$ 
characterized by $\phi(g_0)= h_0$.
Moreover, since $(e,h)$ for $h\neq e$
acts nontrivially on $(b_a)_{a\in A}$,
the stabilizer of $(b_a)_{a\in A}$ is not all of $\pi \times \pi$,
and we may conclude that the stabilizer of $(b_a)_{a\in A}$
is precisely $\Gamma(\phi)$.

Since $\pi\times \pi$ is abelian, stabilizers of elements 
are constant along orbits. Thus,
writing $\calS$ for the set of non-free orbits in the 
$(\pi\times\pi)$-action on $B^A$, we have a well-defined map
\[
	\Phi \colon \calS \longto \Hom(\pi,\pi)
\]
sending the orbit of an element $(b_a)_{a\in A}$
to the homomorphism $\phi\colon \pi \to \pi$ such that the 
stabilizer of $(b_a)_{a\in A}$ is $\Gamma(\phi)$.
Notice that given 
$\phi\colon \pi \to \pi$, the elements $(b_a)_{a\in A}\in B^A$
whose stabilizer is $\Gamma(\phi)$ are precisely the 
$\phi$-equivariant maps $A\to B$. As for any $\phi$
there exists a $\phi$-equivariant map $A\to B$, the map $\Phi$ is
surjective; and as any two $\phi$-equivariant maps $A\to B$
are related by the action of an element of the form 
$(e,h)\in \pi \times \pi$, the map $\Phi$ is injective.
Thus $\Phi$ is a bijection. We conclude that as a $(\pi\times\pi)$-set, 
$B^A$ is isomorphic to the disjoint union of the
$(\pi\times\pi)$-sets
$(\pi\times \pi) / \Gamma(\phi)$ for $\phi\in \Hom(\pi,\pi)$
and a number of copies of the free $(\pi\times\pi)$-set
$(\pi\times\pi)/1$. 
For $n\in \Z$, 
write $\phi_n \colon\pi\to\pi$ for the homomorphism $g\mapsto g^n$.
Then $\phi_n = \phi_m$ if $n\equiv m$ mod $p$, and we have 
\[
	\Hom(\pi,\pi) = \{ \phi_n \mid n \in \Z/p \}.
\]	
Counting points, we conclude that the number of 
copies of $(\pi\times\pi)/1$ occurring in $B^A$ is
$p^{p-2}-1$.

As a $(\pi\times \pi)$-set,
$(\pi\times\pi)/1$
is isomorphic to the one given by the composite
\[
	\alpha
	\colon
	\pi \times \pi 
	\xto{\ i\times i\ }
	\Sigma_p \times \Sigma_p
	\xto{\ \circ\ }	
	\Sigma_{p^2},
\]
where the last map is the one inducing the $\circ$-product on 
homology, 
while
$(\pi\times \pi) / \Gamma(\phi_n)$
is isomorphic to the one given by the composite
\[
	\beta_n
	\colon
	\pi\times\pi 
	\xto{\ s_n\ }
	\pi
	\xto{\ i\ }
	\Sigma_p
\]
where $s_n$ is the homomorphism $(g,h)\mapsto g^{-n} h$. 
Thus the composite
\begin{equation}
\label{eq:pipisigmapp1}
	\pi\times \pi 
	\xto{\ i\times i\ } 
	\Sigma_A \times \Sigma_B 
	\longto 
	\Sigma_{B^A}
	\xto{\ \isom\ }
	\Sigma_{p^p}	
\end{equation}
of  $i\times i$ and  \eqref{eq:sigmabamap}  agrees, up to conjugacy,
with the composite
\begin{equation}
\label{eq:pipisigmapp2}
\newcommand{\entry}{\pi \times \pi\;}
\vcenter{\xymatrix@!0@C=6.5em{
	*!R{\entry}
	\ar[r]^-\Delta
	&
	*!L{\;(\pi \times \pi) \times (\pi\times\pi)^p}
	\\
	*!R{\phantom{\entry}}
	\ar[r]^{\alpha\times \prod_{c=0}^{p-1} \beta_{-c} }
	&
	*!L{\; \Sigma_{p^2}\times \Sigma_p^p} 
	\\
	*!R{\phantom{\entry}}
	\ar[r]^{\Delta\times 1}
	&
	*!L{\;\Sigma_{p^2}^{p^{p-2}-1}\times \Sigma_p^p} 
	\\
	*!R{\phantom{\entry}}
	\ar[r]^{\cdot}
	&
	*!L{\; \Sigma_{p^p}}
}}
\end{equation}
where the last map is the one inducing iterated
$\cdot$-product on homology.

We now aim to understand the map 
induced by the composite~\eqref{eq:pipisigmapp2}
on homology. 
For $\epsilon\in\{0,1\}$, let 
\[
	e^\epsilon(s) 
	= 
	\sum_{k\geq 0} e_{2k+\epsilon} s^k 
	\in 
	H_\ast(\pi)\llbracket s\rrbracket
\]
and observe that 
$i_\ast (e^\epsilon(s)) 
= 
\tilde{E}^\epsilon(s)\in H_\ast(\Sigma_p)\llbracket s\rrbracket$.
[When $p=2$, let
$e^0(s) = \sum_{k\geq 0} e_{k} s^k$
and observe that 
$i_\ast (e^0(s)) 
= 
E^0 (s)\in H_\ast(\Sigma_2)\llbracket s\rrbracket$.]
For $c\in\Z$, the map  induced by $\beta_{-c}$ on homology is given by
\[
	e^{\epsilon_1}(s) \tensor e^{\epsilon_2}(t)
	\longmapsto
    \begin{cases}
		0 & \text{if $\epsilon_1=\epsilon_2 = 1$}
		\\
		c^{\epsilon_1} \tilde{E}^{\epsilon_1+\epsilon_2}(cs+t) 
		& 
		\text{otherwise}		
    \end{cases}
\]
while the map induced by $\alpha$ is given by
\[
	e^{\epsilon_1}(s) \tensor e^{\epsilon_2}(t)
	\longmapsto
	\tilde{E}^{\epsilon_1}(s)\circ \tilde{E}^{\epsilon_2}(t).
\]
Moreover, the composite
\[
	\Sigma_{p^2} 
	\xto{\ \Delta\ }
	\Sigma_{p^2}^{p^{p-2}-1} 
	\xto{\ \cdot\ }
	\Sigma_{p^p-p^2}
\]
induces on homology the map $r\mapsto [p^{p-2}-1]\circ r$.
Finally, on homology, the map induced by the diagonal map
\begin{equation}
\label{eq:pipidiag}
	\Delta \colon \pi \times \pi 
	\longto 
	(\pi\times \pi)\times (\pi\times\pi)^p
\end{equation}
satisfies
\begin{align*}
	e^0(s) \tensor e^0(t)
	&\longmapsto
	(e^0(s) \tensor e^0(t))
	\tensor
	\bigotimes_{c=0}^{p-1} (e^0(s) \tensor e^0(t))
	\\
	e^0(s) \tensor e^1(t)
	&\longmapsto
	\left(		
		\begin{array}{l}
		\phantom{+}(e^0(s) \tensor e^1(t))
    	\tensor
    	\bigotimes_{c=0}^{p-1} (e^0(s) \tensor e^0(t))
		\\[3pt]
    	+
    	(e^0(s)\tensor e^0(t)) 
    	\tensor
    	\sum_{k=0}^{p-1}\bigotimes_{c=0}^{p-1} (e^0(s) \tensor e^{\delta{c,k}}(t))
		\end{array}    	
	\right)
	\\
	e^1(s) \tensor e^0(t)
	&\longmapsto
	\left(		
		\begin{array}{l}
		\phantom{+}(e^1(s) \tensor e^0(t))
    	\tensor
    	\bigotimes_{c=0}^{p-1} (e^0(s) \tensor e^0(t))
		\\[3pt]
    	+
    	(e^0(s)\tensor e^0(t)) 
    	\tensor
    	\sum_{k=0}^{p-1}\bigotimes_{c=0}^{p-1} (e^{\delta{c,k}}(s) \tensor e^0(t))
		\end{array}    	
	\right)
\end{align*}
Equations~\eqref{eq:sharpop00p2}, 
\eqref{eq:sharpop00}, \eqref{eq:sharpop01}, and \eqref{eq:sharpop10}
now follow by using the aforementioned calculations to compute
the map induced by the composite map~\eqref{eq:pipisigmapp2}
on homology, and by observing that the map $i\times i$ in 
\eqref{eq:pipisigmapp1} induces on 
homology the map
\[
	e^{\epsilon_1}(s) \tensor e^{\epsilon_2}(t) 
	\longmapsto
	\tilde{E}^{\epsilon_1}(s) \tensor \tilde{E}^{\epsilon_2}(t) 
\]
[the map $e^0(s) \tensor e^0(t) \mapsto E^0(s) \tensor E^0(t)$
when $p=2$].
Equation~\eqref{eq:sharpop11} could be proven 
similarly by first working out the image of $e^1(s)\tensor e^1(t)$
under the map induced by \eqref{eq:pipidiag}, but
it seems slightly simpler to derive \eqref{eq:sharpop11}
from equation \eqref{eq:sharpop01} (or \eqref{eq:sharpop10})
by applying the Bockstein operation to both sides.
\end{proof}

\begin{remark}
\label{rk:emixedademrels}
Suppose $X$ is an $E_\infty$--ring space.
As the subsring $\calE$ of $\bigoplus_{k\geq 0} H_\ast(\Sigma_k)$
is not closed under the $\sharpop$-product,
our statement of the ``mixed Adem relations''
in Theorem~\ref{thm:mixedademrels}
makes use of the $\circ$-action of the larger
ring $\bigoplus_{k\geq 0} H_\ast(\Sigma_k)$
on $H_\ast(X)$ even when $r$ and $s$
belong to the subring $\calE$.
Using Theorem~\ref{thm:esharopeformulas} and 
the distributivity property
\ref{it:coalgsemimoddist2}
of Definition~\ref{def:coalgsemimod},
it would be straightforward 
to derive from Theorem~\ref{thm:mixedademrels}
``mixed Adem relations'' for evaluating expressions
of the form
\[
	E^{\epsilon}_{m} \sharpop (E^{\delta}_{n} \circ x),
	\qquad
	x\in H_\ast(X),
\]
only featuring the action of the subring $\calE$.
For example,  when $p=2$, one obtains the identity
\[	
	E^0(s)\sharpop (E^0(t)\circ x) 
	= 
	\sum \bigl(E^0(s+t)\circ (E^0(s)\sharpop x')\bigr)
		 \bigl(E^0(t)\circ (E^0(s)\sharpop x'')\bigr).
\]
Especially for odd primes, the resulting formulas
are very complicated, however, and, we feel, not as illuminating
as the aforementioned ingredients for their derivation.
The simplicity of  the ``mixed Adem relations''
of Theorem~\ref{thm:mixedademrels}
is one of our main motivations for studying the 
action of the larger ring $\bigoplus_{k\geq 0} H_\ast(\Sigma_k)$.
\end{remark}

\section*{Acknowledgements}
The author would like to thank Ib Madsen
and Andrew Baker for useful conversations regarding this work.

\smallskip
Supported by the Danish National Research Foundation through the Centre for Symmetry and Deformation (DNRF92)
\smallskip 

\noindent
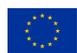
\framebox[\textwidth]{
\begin{tabular*}{0.96\textwidth}{@{\extracolsep{\fill} }cp{0.84\textwidth}}
\raisebox{-0.7\height}{%
    \begin{tikzpicture}[y=0.80pt, x=0.8pt, yscale=-1, inner sep=0pt, outer sep=0pt, 
    scale=0.12]
    \definecolor{c003399}{RGB}{0,51,153}
    \definecolor{cffcc00}{RGB}{255,204,0}
    \begin{scope}[shift={(0,-872.36218)}]
      \path[shift={(0,872.36218)},fill=c003399,nonzero rule] (0.0000,0.0000) rectangle (270.0000,180.0000);
      \foreach \myshift in 
           {(0,812.36218), (0,932.36218), 
    		(60.0,872.36218), (-60.0,872.36218), 
    		(30.0,820.36218), (-30.0,820.36218),
    		(30.0,924.36218), (-30.0,924.36218),
    		(-52.0,842.36218), (52.0,842.36218), 
    		(52.0,902.36218), (-52.0,902.36218)}
        \path[shift=\myshift,fill=cffcc00,nonzero rule] (135.0000,80.0000) -- (137.2453,86.9096) -- (144.5106,86.9098) -- (138.6330,91.1804) -- (140.8778,98.0902) -- (135.0000,93.8200) -- (129.1222,98.0902) -- (131.3670,91.1804) -- (125.4894,86.9098) -- (132.7547,86.9096) -- cycle;
    \end{scope}
    \end{tikzpicture}%
}
&
This project has received funding from the European Union's Horizon 2020 research and innovation programme under grant agreement No 800616.
\end{tabular*}
}

\bibliographystyle{amsalpha}
\bibliography{homology-operations.bib}

\def\cprime{$'$}
\providecommand{\bysame}{\leavevmode\hbox to3em{\hrulefill}\thinspace}
\providecommand{\MR}{\relax\ifhmode\unskip\space\fi MR }
\providecommand{\MRhref}[2]{%
  \href{http://www.ams.org/mathscinet-getitem?mr=#1}{#2}
}
\providecommand{\href}[2]{#2}
\begin{thebibliography}{May09b}

\bibitem[Ada78]{AdamsInfiniteLoopSpaces}
John~Frank Adams, \emph{Infinite loop spaces}, Annals of Mathematics Studies,
  vol.~90, Princeton University Press, Princeton, N.J.; University of Tokyo
  Press, Tokyo, 1978. \MR{505692}

\bibitem[Bro82]{Brown}
Kenneth~S. Brown, \emph{Cohomology of groups}, Graduate Texts in Mathematics,
  vol.~87, Springer-Verlag, New York-Berlin, 1982. \MR{672956 (83k:20002)}

\bibitem[Ch{\ohorn}16]{Chon}
Phan~Ho\`ang Ch{\ohorn}n, \emph{Modular coinvariants and the mod {$p$} homology
  of {$QS^k$}}, Proc. Lond. Math. Soc. (3) \textbf{112} (2016), no.~2,
  351--374. \MR{3471252}

\bibitem[CLM76]{HILS}
Frederick~R. Cohen, Thomas~J. Lada, and J.~Peter May, \emph{The homology of
  iterated loop spaces}, Lecture Notes in Mathematics, Vol. 533,
  Springer-Verlag, Berlin-New York, 1976. \MR{0436146 (55 \#9096)}

\bibitem[DL62]{DyerLashof}
Eldon Dyer and R.~K. Lashof, \emph{Homology of iterated loop spaces}, Amer. J.
  Math. \textbf{84} (1962), 35--88. \MR{141112}

\bibitem[HT98]{HuntonTurner}
John~R. Hunton and Paul~R. Turner, \emph{Coalgebraic algebra}, J. Pure Appl.
  Algebra \textbf{129} (1998), no.~3, 297--313. \MR{1631257}

\bibitem[KA56]{KudoAraki}
Tatsuji Kudo and Sh\^{o}r\^{o} Araki, \emph{Topology of {$H_n$}-spaces and
  {$H$}-squaring operations}, Mem. Fac. Sci. Ky\={u}sy\={u} Univ. A \textbf{10}
  (1956), 85--120. \MR{87948}

\bibitem[Lin]{Lin}
Weinan Lin, personal communication, May 2019.

\bibitem[May70]{MaySteenrodOps}
J.~Peter May, \emph{A general algebraic approach to {S}teenrod operations}, The
  {S}teenrod {A}lgebra and its {A}pplications ({P}roc. {C}onf. to {C}elebrate
  {N}. {E}. {S}teenrod's {S}ixtieth {B}irthday, {B}attelle {M}emorial {I}nst.,
  {C}olumbus, {O}hio, 1970), Lecture Notes in Mathematics, Vol. 168, Springer,
  Berlin, 1970, pp.~153--231. \MR{0281196}

\bibitem[May72]{MayGILS}
J.~P. May, \emph{The geometry of iterated loop spaces}, Springer-Verlag,
  Berlin, 1972, Lectures Notes in Mathematics, Vol. 271. \MR{0420610 (54
  \#8623b)}

\bibitem[May09a]{MayBiperm}
\bysame, \emph{The construction of {$E_\infty$} ring spaces from bipermutative
  categories}, New topological contexts for {G}alois theory and algebraic
  geometry ({BIRS} 2008), Geom. Topol. Monogr., vol.~16, Geom. Topol. Publ.,
  Coventry, 2009, pp.~283--330. \MR{2544392}

\bibitem[May09b]{MayGoodFor}
\bysame, \emph{What are {$E_\infty$} ring spaces good for?}, New topological
  contexts for {G}alois theory and algebraic geometry ({BIRS} 2008), Geom.
  Topol. Monogr., vol.~16, Geom. Topol. Publ., Coventry, 2009, pp.~331--365.
  \MR{2544393}

\bibitem[May09c]{MayWhatPrecisely}
\bysame, \emph{What precisely are {$E_\infty$} ring spaces and {$E_\infty$}
  ring spectra?}, New topological contexts for {G}alois theory and algebraic
  geometry ({BIRS} 2008), Geom. Topol. Monogr., vol.~16, Geom. Topol. Publ.,
  Coventry, 2009, pp.~215--282. \MR{2544391}

\bibitem[McC01]{SSguide}
John McCleary, \emph{A user's guide to spectral sequences}, second ed.,
  Cambridge Studies in Advanced Mathematics, vol.~58, Cambridge University
  Press, Cambridge, 2001. \MR{1793722}

\bibitem[MM79]{MadsenMilgram}
Ib~Madsen and R.~James Milgram, \emph{The classifying spaces for surgery and
  cobordism of manifolds}, Annals of Mathematics Studies, vol.~92, Princeton
  University Press, Princeton, N.J.; University of Tokyo Press, Tokyo, 1979.
  \MR{548575 (81b:57014)}

\bibitem[RW77]{RavenelWilson}
Douglas~C. Ravenel and W.~Stephen Wilson, \emph{The {H}opf ring for complex
  cobordism}, J. Pure Appl. Algebra \textbf{9} (1976/77), no.~3, 241--280.
  \MR{448337}

\bibitem[Ste83]{Steiner}
Richard Steiner, \emph{Homology operations and power series}, Glasgow Math. J.
  \textbf{24} (1983), no.~2, 161--168. \MR{706145}

\bibitem[Tur97]{Turner}
Paul~R. Turner, \emph{Dickson coinvariants and the homology of {$QS^0$}}, Math.
  Z. \textbf{224} (1997), no.~2, 209--228. \MR{1431193 (98e:55023)}

\bibitem[Wil00]{HopfRingsSurvey}
W.~Stephen Wilson, \emph{Hopf rings in algebraic topology}, Expo. Math.
  \textbf{18} (2000), no.~5, 369--388. \MR{1802339}

\end{thebibliography}

\end{document}